\documentclass[oneside,english,reqno]{amsart}
\usepackage[T1]{fontenc}
\usepackage[latin9]{inputenc}
\usepackage{babel}
\usepackage{refstyle}
\usepackage{float}
\usepackage{mathrsfs}
\usepackage{amsthm}
\usepackage{amstext}
\usepackage{amssymb}
\usepackage{graphicx}
\usepackage{esint}
\usepackage[all]{xy}
\PassOptionsToPackage{normalem}{ulem}
\usepackage{ulem}
\usepackage[unicode=true,pdfusetitle,
 bookmarks=true,bookmarksnumbered=false,bookmarksopen=false,
 breaklinks=false,pdfborder={0 0 1},backref=false,colorlinks=false]
 {hyperref}
\hypersetup{
 colorlinks=true,citecolor=blue,linkcolor=blue,linktocpage=true}
\usepackage{breakurl}

\makeatletter

\AtBeginDocument{\providecommand\lemref[1]{\ref{lem:#1}}}
\AtBeginDocument{\providecommand\secref[1]{\ref{sec:#1}}}
\AtBeginDocument{\providecommand\figref[1]{\ref{fig:#1}}}
\AtBeginDocument{\providecommand\thmref[1]{\ref{thm:#1}}}
\AtBeginDocument{\providecommand\corref[1]{\ref{cor:#1}}}
\AtBeginDocument{\providecommand\defref[1]{\ref{def:#1}}}
\AtBeginDocument{\providecommand\propref[1]{\ref{prop:#1}}}
\AtBeginDocument{\providecommand\subref[1]{\ref{sub:#1}}}
\AtBeginDocument{\providecommand\exref[1]{\ref{ex:#1}}}
\AtBeginDocument{\providecommand\tabref[1]{\ref{tab:#1}}}
\providecommand{\tabularnewline}{\\}
\RS@ifundefined{subref}
  {\def\RSsubtxt{section~}\newref{sub}{name = \RSsubtxt}}
  {}
\RS@ifundefined{thmref}
  {\def\RSthmtxt{theorem~}\newref{thm}{name = \RSthmtxt}}
  {}
\RS@ifundefined{lemref}
  {\def\RSlemtxt{lemma~}\newref{lem}{name = \RSlemtxt}}
  {}

\numberwithin{equation}{section}
\numberwithin{figure}{section}
\numberwithin{table}{section}
\usepackage{enumitem}		
\theoremstyle{plain}
\newtheorem{thm}{\protect\theoremname}[section]
  \theoremstyle{definition}
  \newtheorem{defn}[thm]{\protect\definitionname}
  \theoremstyle{remark}
  \newtheorem{rem}[thm]{\protect\remarkname}
  \theoremstyle{plain}
  \newtheorem{lem}[thm]{\protect\lemmaname}
  \theoremstyle{plain}
  \newtheorem{cor}[thm]{\protect\corollaryname}
  \theoremstyle{plain}
  \newtheorem{prop}[thm]{\protect\propositionname}
  \theoremstyle{remark}
  \newtheorem*{claim*}{\protect\claimname}
  \theoremstyle{definition}
  \newtheorem{example}[thm]{\protect\examplename}
  \theoremstyle{remark}
  \newtheorem*{acknowledgement*}{\protect\acknowledgementname}

\allowdisplaybreaks
\providecommand{\MR}[1]{}

\usepackage{mathtools}
\usepackage{refstyle}
\newref{thm}{
         name      = {theorem~},
         names     = {theorems~},
         Name      = {Theorem~},
         Names     = {Theorems~},
         rngtxt    = \RSrngtxt,
         lsttwotxt = \RSlsttxt,
         lsttxt    = \RSlsttxt
         }  
\newref{prop}{
         name      = {proposition~},
         names     = {propositions~},
         Name      = {Proposition~},
         Names     = {Propositions~},
         rngtxt    = \RSrngtxt,
         lsttwotxt = \RSlsttxt,
         lsttxt    = \RSlsttxt
         }   
         
\newref{cor}{
         name      = {corollary~},
         names     = {corollaries~},
         Name      = {Corollary~},
         Names     = {Corollaries~},
         rngtxt    = \RSrngtxt,
         lsttwotxt = \RSlsttxt,
         lsttxt    = \RSlsttxt
         }   
         \newref{rem}{
         name      = {remark~},
         names     = {remarks~},
         Name      = {Remark~},
         Names     = {Remarks~},
         rngtxt    = \RSrngtxt,
         lsttwotxt = \RSlsttxt,
         lsttxt    = \RSlsttxt
         }   
         \newref{ex}{
         name      = {example~},
         names     = {examples~},
         Name      = {Example~},
         Names     = {Examples~},
         rngtxt    = \RSrngtxt,
         lsttwotxt = \RSlsttxt,
         lsttxt    = \RSlsttxt
         }   
         
         \newref{def}{
         name      = {definition~},
         names     = {definitions~},
         Name      = {Definition~},
         Names     = {Definitions~},
         rngtxt    = \RSrngtxt,
         lsttwotxt = \RSlsttxt,
         lsttxt    = \RSlsttxt
         }   

\renewcommand{\section}{%
\@startsection{section}{1}%
  \z@{.7\linespacing\@plus\linespacing}{.5\linespacing}%
  {\normalfont\scshape\centering\bfseries}}

\renewcommand{\subsection}{%
\@startsection{subsection}{2}%
  \z@{.5\linespacing\@plus.7\linespacing}{.5\linespacing}%
  {\normalfont\bfseries}}

\renewcommand{\subsubsection}{%
\@startsection{subsubsection}{2}%
  \z@{.5\linespacing\@plus.7\linespacing}{.5\linespacing}%
  {\normalfont\bfseries}}

\newref{sec}{%
      name      = \RSsectxt,
      names     = \RSsecstxt,
      Name      = \RSSectxt,
      Names     = \RSSecstxt,
      refcmd    = {\ifRSstar\S\fi\ref{#1}},
      rngtxt    = \RSrngtxt,
      lsttwotxt = \RSlsttwotxt,
      lsttxt    = \RSlsttxt}

\theoremstyle{definition}
\newtheorem{app}{Application}[section]

\makeatother

  \providecommand{\acknowledgementname}{Acknowledgement}
  \providecommand{\claimname}{Claim}
  \providecommand{\corollaryname}{Corollary}
  \providecommand{\definitionname}{Definition}
  \providecommand{\examplename}{Example}
  \providecommand{\lemmaname}{Lemma}
  \providecommand{\propositionname}{Proposition}
  \providecommand{\remarkname}{Remark}
\providecommand{\theoremname}{Theorem}

\begin{document}

\title{von Neumann indices and classes of positive definite functions}

\author{Palle Jorgensen and Feng Tian}

\address{(Palle E.T. Jorgensen) Department of Mathematics, The University
of Iowa, Iowa City, IA 52242-1419, U.S.A. }

\email{palle-jorgensen@uiowa.edu}

\urladdr{http://www.math.uiowa.edu/\textasciitilde{}jorgen/}

\address{(Feng Tian) Department of Mathematics, Wright State University, Dayton,
OH 45435, U.S.A.}

\email{feng.tian@wright.edu}

\urladdr{http://www.wright.edu/\textasciitilde{}feng.tian/}



\maketitle
\pagestyle{myheadings}
\markright{VON NEUMANN INDICES AND POSITIVE DEFINITE FUNCTIONS}
\begin{abstract}
With view to applications, we establish a correspondence between two
problems: (i) the problem of finding continuous positive definite
extensions of functions $F$ which are defined on open bounded domains
$\Omega$ in $\mathbb{R}$, on the one hand; and (ii) spectral theory
for elliptic differential operators acting on $\Omega$, (constant
coefficients.) A novelty in our approach is the use of a reproducing
kernel Hilbert space $\mathscr{H}_{F}$ computed directly from $\left(\Omega,F\right)$,
as well as algorithms for computing relevant orthonormal bases in
$\mathscr{H}_{F}$.
\end{abstract}
\tableofcontents{}

\section{Introduction}

From the postulates of quantum physics, we know that measurements
of observables are computed from associated selfadjoint operators---observables.
From the corresponding spectral resolutions, we get probability measures,
and of course uncertainty. There are many philosophical issues (which
we bypass here), and we do not yet fully understand quantum reality.
See for example, \cite{Sla03,CJK+12}.

The axioms are as follows: An observable is a Hermitian (self-adjoint)
linear operator mapping a Hilbert space, the space of states, into
itself. The values obtained in a physical measurement are, in general,
described by a probability distribution; and the distribution represents
a suitable \textquotedblleft average\textquotedblright{} (or \textquotedblleft expectation\textquotedblright )
in a measurement of values of some quantum observable in a state of
some prepared system. The states are (up to phase) unit vectors in
the Hilbert space, and a measurement corresponds to a probability
distribution (derived from a projection-valued spectral measure).
The spectral type may be continuous (such as position and momentum)
or discrete (such as spin). Information about the measures $\mu$
are computed with the use of \emph{generating functions} (on $\mathbb{R}$),
i.e., spectral (Bochner/Fourier) transforms of the corresponding measure.
Generating functions are positive definite continuous functions $F\left(=F_{\mu}\right)$
on $\mathbb{R}$. One then tries to recover $\mu$ from information
about $F$. In our paper we explore the cases when information about
$F\left(x\right)$ is only available for $x$ in a bounded interval.

We prove a number of results for positive definite (p.d.) continuous
functions $F$ defined only locally, i.e., defined on some fixed bounded
connected subset in $\mathbb{R}^{d}$. We show that such partially
defined positive definite functions have a number of intriguing features,
apart from the properties they inherit from possibly being restrictions
of continuous positive definite functions defined on all of $\mathbb{R}^{d}$.

Our emphasis will be on explicit formulas, which is why we specialize
to the case $d=1$, and the case when $F$ is defined initially on
a symmetric interval $\left(-a,a\right)$. (In this case, $F$ automatically
has positive definite continuous extensions to all of the real line
$\mathbb{R}$.) Our applications include boundary value problems,
differential equations, positive definite integral operators, reproducing
kernel Hilbert spaces, orthogonal bases, interpolation, spectral theory,
and harmonic analysis.

In probability theory, normalized continuous positive definite functions
$F$, i.e., $F(0)=1$, arise as \emph{generating functions} for probability
measures, and one passes from information about one to the other;
-- from generating function to probability measure is called \textquotedblleft the
inverse problem\textquotedblright , see e.g., \cite{DM85}. Hence
the study of partially defined p.d. functions addresses the inverse
question: ambiguity of measures when only partial information for
a possible generating function is available.

The study of locally defined positive definite (p.d.) continuous functions
$F$ is important in pure and applied mathematics, and it entails
both analysis and geometry. The study of domains in $\mathbb{R}^{d}$,
for $d>1$, and domains in Lie groups combine the geometric and analytic
issues, see e.g., \cite{MR1069255,Jor91,JPT14}. In this paper, we
wish to stress such analytic questions as spectral theory, stochastic
analysis \cite{Ito06}, reproducing kernels $\mathscr{H}_{F}$, explicit
orthonormal bases in $\mathscr{H}_{F}$, and boundary-value problems.
For this reason, we have restricted here the setting to domains in
$\mathbb{R}$, so one dimension, $d=1$. And we have narrowed our
focus on specific cases. Even in this more narrow setting, the applications
include scattering theory (in physics) \cite{OH13}, sampling theory
(in engineering) \cite{MMJ12,ADK13}, and statistical learning theory,
see e.g., \cite{SZ05,SZ07}.

Another important difference between the cases $d=1$, and $d>1$,
is that for $d=1$, every positive definite continuous function defined
on an interval, has continuous p.d. extensions to $\mathbb{R}$. The
analogous result is known to be false when $d>1$; see \cite{Ru70}.

\section{The Reproducing Kernel Hilbert Space $\mathscr{H}_{F}$}

Associated to a pair $\left(\Omega,F\right)$, where $F$ is a prescribed
continuous positive definite function defined on $\Omega$, we outline
a reproducing kernel Hilbert space $\mathscr{H}_{F}$ which will serve
as a key tool in our analysis. The particular RKHSs we need here will
have additional properties (as compared to a general framework); which
allow us to give explicit formulas for our solutions. 

In a general setup, reproducing kernel Hilbert spaces were pioneered
by Aronszajn in the 1950s \cite{Aro50}; and subsequently they have
been used in a host of applications; e.g., \cite{SZ09,SZ07}. 

As for positive definite functions, their use and applications are
extensive and includes such areas as stochastic processes, see e.g.,
\cite{JorPea13,AJSV13,JP12,AJ12}; harmonic analysis (see \cite{JO00})
, and the references there); potential theory \cite{Fu74b,KL14};
operators in Hilbert space \cite{Al92,AD86}; and spectral theory
\cite{AH13,Nus75,Dev72,Dev59}. We stress that the literature is vast,
and the above list is only a small sample.
\begin{defn}
Let $G$ be a Lie group. Fix $\Omega\subset G$, non-empty, open and
connected. A continuous function 
\begin{equation}
F:\Omega^{-1}\cdot\Omega\rightarrow\mathbb{C}\label{eq:li-1}
\end{equation}
is \emph{positive definite }(p.d.) if
\begin{equation}
\sum_{i}\sum_{j}\overline{c_{i}}c_{j}F\left(x_{i}^{-1}x_{j}\right)\geq0,\label{eq:li-2}
\end{equation}
for all finite systems $\left\{ c_{i}\right\} \subset\mathbb{C}$,
and points $\left\{ x_{i}\right\} \subset\Omega$. 

Equivalently,
\begin{equation}
\int_{\Omega}\int_{\Omega}\overline{\varphi\left(x\right)}\varphi\left(y\right)F\left(x^{-1}y\right)dxdy\geq0,\label{eq:li-3}
\end{equation}
for all $\varphi\in C_{c}\left(\Omega\right)$; where $dx$ denotes
a choice of \uline{left}-invariant Haar measure on $G$.
\end{defn}
For simplicity we focus on the case $G=\mathbb{R},$ indicating the
changes needed for general Lie groups.
\begin{defn}
Fix $0<a<\infty$, set $\Omega:=\left(0,a\right)$. Let $F:\Omega-\Omega\rightarrow\mathbb{C}$
be a continuous p.d. function. The \emph{reproducing kernel Hilbert
space (RKHS),} $\mathscr{H}_{F}$, is the completion of
\begin{equation}
\sum_{\text{finite}}c_{j}F\left(\cdot-x_{j}\right):c_{j}\in\mathbb{C}\label{eq:H1}
\end{equation}
with respect to the inner product
\[
\left\langle F\left(\cdot-x\right),F\left(\cdot-y\right)\right\rangle _{\mathscr{H}_{F}}=F\left(x-y\right),\;\forall x,y\in\Omega,\;\mbox{and}
\]
\begin{equation}
\big\langle\sum_{i}c_{i}F\left(\cdot-x_{i}\right),\sum_{j}c_{j}F\left(\cdot-x_{j}\right)\big\rangle_{\mathscr{H}_{F}}=\sum_{i}\sum_{j}\overline{c_{i}}c_{j}F\left(x_{i}-x_{j}\right),\label{eq:ip-discrete}
\end{equation}
\end{defn}
\begin{rem}
Throughout, we use the convention that the inner product is conjugate
linear in the first variable, and linear in the second variable. When
more than one inner product is used, subscripts will make reference
to the Hilbert space. 

\textbf{Notation.} Inner product and norms will be denoted $\left\langle \cdot,\cdot\right\rangle $,
and $\left\Vert \cdot\right\Vert $ respectively. Often more than
one inner product is involved, and subscripts are used for identification.\end{rem}
\begin{lem}
\label{lem:RKHS-def-by-integral}The reproducing kernel Hilbert space
(RKHS), $\mathscr{H}_{F}$, is the Hilbert completion of the functions
\begin{equation}
F_{\varphi}\left(x\right)=\int_{\Omega}\varphi\left(y\right)F\left(x-y\right)dy,\;\forall\varphi\in C_{c}^{\infty}\left(\Omega\right),x\in\Omega\label{eq:H2}
\end{equation}
with respect to the inner product
\begin{equation}
\left\langle F_{\varphi},F_{\psi}\right\rangle _{\mathscr{H}_{F}}=\int_{\Omega}\int_{\Omega}\overline{\varphi\left(x\right)}\psi\left(y\right)F\left(x-y\right)dxdy,\;\forall\varphi,\psi\in C_{c}^{\infty}\left(\Omega\right).\label{eq:hi2}
\end{equation}
In particular, 
\begin{equation}
\left\Vert F_{\varphi}\right\Vert _{\mathscr{H}_{F}}^{2}=\int_{\Omega}\int_{\Omega}\overline{\varphi\left(x\right)}\varphi\left(y\right)F\left(x-y\right)dxdy,\;\forall\varphi\in C_{c}^{\infty}\left(\Omega\right)\label{eq:hn2}
\end{equation}
and 
\begin{equation}
\left\langle F_{\varphi},F_{\psi}\right\rangle _{\mathscr{H}_{F}}=\int_{\Omega}\overline{\varphi\left(x\right)}F_{\psi}\left(x\right)dx,\;\forall\phi,\psi\in C_{c}^{\infty}(\Omega).
\end{equation}
\end{lem}
\begin{proof}
Apply standard approximation, see \lemref{dense} below.
\end{proof}
Among the cases of partially defined positive definite functions,
the following example $F\left(x\right)=e^{-\left|x\right|}$, in the
symmetric interval $\left(-1,1\right)$, will play a special role.

There are many reasons for this:
\begin{enumerate}[leftmargin=*,label=(\roman{enumi})]
\item It is of independent interest, and its type 1 extensions (see \secref{types})
can be written down explicitly.
\item Its applications include stochastic analysis \cite{Ito06} as follows.
Given a random variable $X$ in a process; if $\mu$ is its distribution,
then there are two measures of concentration for $\mu$, one called
\textquotedblleft degree of concentration,\textquotedblright{} and
the other \textquotedblleft dispersion,\textquotedblright{} both computed
directly from $F\left(x\right)=e^{-\left|x\right|}$ applied to $\mu$,
see \ref{app1} Application, below.
\item In addition, there are analogous relative notions for comparing different
samples in a fixed stochastic process. These notions are defined with
the use of example $F\left(x\right)=e^{-\left|x\right|}$, and it
will frequently be useful to localize the $x$-variable in a compact
interval.
\item Additional reasons for special attention to example $F\left(x\right)=e^{-\left|x\right|}$,
for $x\in\left(-1,1\right)$ is its use in sampling theory, and analysis
of de Branges spaces \cite{DM85}, as well as its role as a Greens
function for an important boundary value problem. 
\item Related to this, the reproducing kernel Hilbert space $\mathscr{H}_{F}$
associated to this p.d. function $F$ has a number of properties that
also hold for wider families of locally defined positive definite
function of a single variable. In particular, $\mathscr{H}_{F}$ has
Fourier bases: The RKHS $\mathscr{H}_{F}$ has orthogonal bases of
complex exponentials $e_{\lambda}$ with aperiodic frequency distributions,
i.e., frequency points $\left\{ e_{\lambda}\right\} $ on the real
line which do not lie on any arithmetic progression, see Fig \ref{fig:expExt}.
For details on this last point, see Corollaries \ref{cor:elambda},
\ref{cor:HFinner2}, \ref{cor:elambda1}, and \ref{cor:Fext}. 
\end{enumerate}
The remaining of this section is devoted to a number of technical
lemmas which will be used throughout the paper. Given a locally defined
continuous positive definite function $F$, the issues addressed below
are: approximation (\lemref{dense}), a reproducing kernel Hilbert
space (RKHS) $\mathscr{H}_{F}$ built from $F$, an integral transform,
and a certain derivative operator $D^{\left(F\right)}$, generally
unbounded in the RKHS $\mathscr{H}_{F}$. We will be concerned with
boundary value problems for $D^{\left(F\right)}$, and in order to
produce suitable orthonormal bases in $\mathscr{H}_{F}$, we be concerned
with an explicit family of skew-adjoint extensions of $D^{\left(F\right)}$,
as well as the associated spectra, see Corollaries \ref{cor:defg}
and \ref{cor:ptspec}.
\begin{lem}
\label{lem:dense}Let $\varphi$ be a function s.t.
\begin{enumerate}
\item $\mathrm{supp}\left(\varphi\right)\subset\left(0,a\right)$;
\item $\varphi\in C_{c}^{\infty}\left(0,a\right)$, $\varphi\geq0$;
\item $\int_{0}^{a}\varphi\left(t\right)dt=1$. 
\end{enumerate}

Fix $x\in\left(0,a\right)$, and set $\varphi_{n,x}\left(t\right):=n\varphi\left(n\left(t-x\right)\right)$.
Then $\lim_{n\rightarrow\infty}\varphi_{n,x}=\delta_{x}$, i.e., the
Dirac measure at $x$; and 
\begin{equation}
\left\Vert F_{\varphi_{n,x}}-F\left(\cdot-x\right)\right\Vert _{\mathscr{H}_{F}}\rightarrow0,\;\mbox{as }n\rightarrow\infty.\label{eq:approx}
\end{equation}
Hence $\left\{ F_{\varphi}:\varphi\in C_{c}^{\infty}\left(0,a\right)\right\} $
spans a dense subspace in $\mathscr{H}_{F}$. See \figref{approx}.

\end{lem}
\begin{figure}[H]
\includegraphics[scale=0.4]{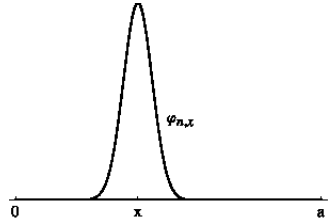}

\protect\caption{\label{fig:approx}The approximate identity $\varphi_{n,x}$}
\end{figure}

Recall, the following facts about $\mathscr{H}_{F},$ which follow
from the general theory \cite{Aro50} of RKHS:
\begin{itemize}
\item $F(0)>0,$ so we can always arrange $F(0)=1.$
\item $F(-x)=\overline{F(x)}$
\item $\mathscr{H}_{F}$ consists of continuous functions $\xi:\Omega-\Omega\rightarrow\mathbb{C}.$
\item The reproducing property:
\[
\left\langle F\left(\cdot-x\right),\xi\right\rangle _{\mathscr{H}_{F}}=\xi\left(x\right),\;\forall\xi\in\mathscr{H}_{F},\forall x\in\Omega,
\]
is a direct consequence of (\ref{eq:ip-discrete}).\end{itemize}
\begin{rem}
It follows from the reproducing property that if $F_{\phi_{n}}\to\xi$
in $\mathscr{H}_{F},$ then $F_{\phi_{n}}$ converges uniformly to
$\xi$ in $\Omega.$ In fact 
\begin{align*}
\left|F_{\phi_{n}}\left(x\right)-\xi\left(x\right)\right| & =\left|\left\langle F\left(\cdot-x\right),F_{\phi_{n}}-\xi\right\rangle _{\mathscr{H}_{F}}\right|\\
 & \leq\left\Vert F\left(\cdot-x\right)\right\Vert _{\mathscr{H}_{F}}\left\Vert F_{\phi_{n}}-\xi\right\Vert _{\mathscr{H}_{F}}\\
 & =F\left(0\right)\left\Vert F_{\phi_{n}}-\xi\right\Vert _{\mathscr{H}_{F}}.
\end{align*}
\end{rem}
\begin{lem}
\label{lem:TFvarphi}Let $F:\left(-a,a\right)\rightarrow\mathbb{C}$
be a continuous and p.d. function, and let $\mathscr{H}_{F}$ be the
corresponding RKHS. Then:
\begin{enumerate}
\item \label{enu:F1-1}the integral $F_{\varphi}:=\int_{0}^{a}\varphi\left(y\right)F\left(\cdot-y\right)dy$
is convergent in $\mathscr{H}_{F}$ for all $\varphi\in C_{c}\left(0,a\right)$;
and 
\item \label{enu:F1-2}for all $\xi\in\mathscr{H}_{F}$, we have:
\begin{equation}
\left\langle F_{\varphi},\xi\right\rangle _{\mathscr{H}_{F}}=\int_{0}^{a}\overline{\varphi\left(x\right)}\xi\left(x\right)dx.\label{eq:F1-1}
\end{equation}

\end{enumerate}
\end{lem}
\begin{proof}
For simplicity, we assume the following normalization $F\left(0\right)=1$;
then for all $y_{1},y_{2}\in\left(0,1\right)$, we have 
\begin{equation}
\left\Vert F\left(\cdot-y_{1}\right)-F\left(\cdot-y_{2}\right)\right\Vert _{\mathscr{H}_{F}}^{2}=2\left(1-\Re\left\{ F\left(y_{1}-y_{2}\right)\right\} \right).\label{eq:F1-2}
\end{equation}
Now, view the integral in (\ref{enu:F1-1}) as a $\mathscr{H}_{F}$-vector
valued integral. If $\varphi\in C_{c}\left(0,a\right)$, this integral
$\int_{0}^{a}\varphi\left(y\right)F\left(\cdot-y\right)dy$ is the
$\mathscr{H}_{F}$-norm convergent. Since $\mathscr{H}_{F}$ is a
RKHS, $\left\langle \cdot,\xi\right\rangle _{\mathscr{H}_{F}}$ is
continuous on $\mathscr{H}_{F}$, and it passes under the integral
in (\ref{enu:F1-1}). Using 
\begin{equation}
\left\langle F\left(y-\cdot\right),\xi\right\rangle _{\mathscr{H}_{F}}=\xi\left(y\right)\label{eq:F1-3}
\end{equation}
the desired conclusion (\ref{eq:F1-1}) follows.
\end{proof}

\begin{cor}
Let $F:\left(-a,a\right)\rightarrow\mathbb{C}$ be as above, and let
$\mathscr{H}_{F}$ be the corresponding RKHS. For $\varphi\in C_{c}^{1}\left(0,a\right)$,
set 
\begin{equation}
F_{\varphi}\left(x\right)=\left(T_{F}\varphi\right)\left(x\right)=\int_{0}^{a}\varphi\left(y\right)F\left(x-y\right)dy;\label{eq:F1-4}
\end{equation}
then $F_{\varphi}\in C^{1}\left(0,a\right)$, and 
\begin{equation}
\frac{d}{dx}F_{\varphi}\left(x\right)=\left(T_{F}\left(\varphi'\right)\right)\left(x\right),\;\forall x\in\left(0,a\right).\label{eq:F1-5}
\end{equation}
\end{cor}
\begin{proof}
Since $F_{\varphi}\left(x\right)=\int_{0}^{a}\varphi\left(y\right)F\left(x-y\right)dy$,
$x\in\left(0,a\right)$; the desired assertion (\ref{eq:F1-5}) follows
directly from the arguments in the proof of \lemref{TFvarphi}.\end{proof}
\begin{thm}
\label{thm:HF}Fix $0<a<\infty$. A continuous function $\xi:\left(0,a\right)\rightarrow\mathbb{C}$
is in $\mathscr{H}_{F}$ if and only if there exists a finite constant
$A>0$, such that
\begin{equation}
\sum_{i}\sum_{j}\overline{c_{i}}c_{j}\overline{\xi\left(x_{i}\right)}\xi\left(x_{j}\right)\leq A\sum_{i}\sum_{j}\overline{c_{i}}c_{j}F\left(x_{i}-x_{j}\right)\label{eq:bdd}
\end{equation}
for all finite system $\left\{ c_{i}\right\} \subset\mathbb{C}$ and
$\left\{ x_{i}\right\} \subset\left(0,a\right)$. Equivalently, for
all $\varphi\in C_{c}^{\infty}\left(\Omega\right)$, 
\begin{align}
\left|\int_{0}^{a}\varphi\left(y\right)\xi\left(y\right)dy\right|^{2} & \leq A\int_{0}^{a}\int_{0}^{a}\overline{\varphi\left(x\right)}\varphi\left(y\right)F\left(x-y\right)dxdy\label{eq:bdd2}
\end{align}

\end{thm}
We will use these two conditions (\ref{eq:bdd})($\Leftrightarrow$(\ref{eq:bdd2}))
when considering for example the von Neumann deficiency-subspaces
for skew Hermitian operators with dense domain in $\mathscr{H}_{F}$.
\begin{proof}[Proof of \thmref{HF}]
Note, if $\xi\in\mathscr{H}_{F}$, the LHS in (\ref{eq:bdd2}) is
$\big|\left\langle F_{\varphi},\xi\right\rangle _{\mathscr{H}_{F}}\big|^{2}$,
and so (\ref{eq:bdd2}) holds, since $\left\langle \cdot,\xi\right\rangle _{\mathscr{H}_{F}}$
is continuous on $\mathscr{H}_{F}$.

If $\xi$ is continuous on $\left[0,a\right]$, and if (\ref{eq:bdd2})
holds, then 
\[
\mathscr{H}_{F}\ni F_{\varphi}\longmapsto\int_{0}^{a}\varphi\left(y\right)\xi\left(y\right)dy
\]
is well-defined, continuous, linear; and extends to $\mathscr{H}_{F}$
by density (see \lemref{dense}). Hence, by Riesz's theorem, $\exists!$
$k_{\xi}\in\mathscr{H}_{F}$ s.t. 
\[
\int_{0}^{a}\varphi\left(y\right)\xi\left(y\right)dy=\left\langle F_{\varphi},k_{\xi}\right\rangle _{\mathscr{H}_{F}}.
\]
But using the reproducing property in $\mathscr{H}_{F}$, and $F_{\varphi}\left(x\right)=\int_{0}^{a}\varphi\left(x\right)F\left(x-y\right)dy$,
we get 
\[
\int_{0}^{a}\overline{\varphi\left(x\right)}\xi\left(x\right)dx=\int_{0}^{a}\overline{\varphi\left(x\right)}k_{\xi}\left(x\right)dx,\;\forall\varphi\in C_{c}\left(0,a\right)
\]
so 
\[
\int_{0}^{a}\varphi\left(x\right)\left(\xi\left(x\right)-k_{\xi}\left(x\right)\right)dx=0,\;\forall\varphi\in C_{c}\left(0,a\right);
\]
it follows that $\xi-k_{\xi}=0$ on $\left(0,a\right)$ $\Longrightarrow$
$\xi-k_{x}=0$ on $\left[0,a\right]$. \end{proof}
\begin{cor}
\label{cor:Fmu}Let $F:\left(-a,a\right)\rightarrow\mathbb{C}$ be
positive definite and continuous, and let $\mathscr{H}_{F}$ be the
RKHS introduced in \lemref{TFvarphi}. Let $\mu$ be a complex measure
of bounded variation defined on the Borel sigma-algebra of subsets
in $\left[0,a\right]$, then 
\begin{equation}
F_{\mu}\left(x\right)=\int_{0}^{a}F\left(x-y\right)d\mu\left(y\right)\label{eq:mu-1}
\end{equation}
is a well-defined continuous function on $\left[0,a\right]$, and
$F_{\mu}\in\mathscr{H}_{F}$.\end{cor}
\begin{rem}
For the converse of the above result, i.e., whether every $h\in\mathscr{H}_{F}$
takes the form $h=F_{\mu}$ in (\ref{eq:mu-1}), we consider the special
case where $F\left(x\right)=e^{-\left|x\right|}$, $\left|x\right|<1$.
See \corref{expmu} for details. \end{rem}
\begin{proof}
Let $F$ and $\mu$ be as in the statement of the corollary. Then
by the Jordan-decomposition of $\mu$ there are finite positive measures
$\mu_{i}$, $i=1,2,3,4$ such that
\begin{equation}
\mu=\mu_{1}-\mu_{2}+i\left(\mu_{3}-\mu_{4}\right).\label{eq:mu-2}
\end{equation}
As a result, we can assume in the proof that $\mu$ in (\ref{eq:mu-1})
is positive and finite.

Since $F$ has a unique continuous and bounded extension to $\left[-a,a\right]$,
note that then $F\left(x-\cdot\right)\in L^{1}\left(\mu\right)$ for
all $x\in\left(0,a\right)$. Since translation in $L^{1}\left(\mu\right)$
is continuous, it follows that $F_{\mu}\left(\cdot\right)$ in (\ref{eq:mu-1})
is continuous in $x$. (We used that $\mu$ is positive here.) 

Now for $\mu,\nu$ as in the corollary set 
\begin{equation}
\left\langle F_{\mu},F_{\nu}\right\rangle _{\mathscr{H}_{F}}=\int_{0}^{a}\int_{0}^{a}F\left(x-y\right)d\mu\left(x\right)d\nu\left(y\right).\label{eq:mu-3}
\end{equation}
By the Schwarz-inequality, we have 
\begin{equation}
\left|\left\langle F_{\mu},F_{\nu}\right\rangle _{\mathscr{H}_{F}}\right|^{2}\leq\left\langle F_{\mu},F_{\mu}\right\rangle _{\mathscr{H}_{F}}\left\langle F_{\nu},F_{\nu}\right\rangle _{\mathscr{H}_{F}}.\label{eq:mu-4}
\end{equation}
Applying this to $d\nu=\varphi\left(x\right)dx$, $\varphi\in C_{c}\left(0,a\right)$
we get (an application of (\ref{eq:mu-4})):
\begin{equation}
\left|\int_{0}^{a}\overline{F_{\mu}\left(x\right)}\varphi\left(x\right)dx\right|^{2}\leq\left\Vert F_{\mu}\right\Vert _{\mathscr{H}_{F}}^{2}\int_{0}^{a}\int_{0}^{a}\overline{\varphi\left(x\right)}\varphi\left(y\right)F\left(x-y\right)dxdy\label{eq:mu-5}
\end{equation}
where we have 
\begin{equation}
\int_{0}^{a}\overline{F_{\mu}\left(x\right)}\varphi\left(x\right)dx=\left\langle F_{\mu},F_{\varphi}\right\rangle _{\mathscr{H}_{F}}\label{eq:mu-6}
\end{equation}
from an application of \lemref{dense}.

Now set $\xi\left(x\right):=F_{\mu}\left(x\right)$ for $x\in\left[0,a\right]$
where $F_{\mu}$ is defined in (\ref{eq:mu-1}), and is continuous
in $x$. In the \emph{a priori} estimate (\ref{eq:mu-5}), we verified
the property (\ref{eq:bdd2}) in \thmref{HF}; and the conclusion
in \thmref{HF} then yields $\xi=F_{\mu}\in\mathscr{H}_{F}$ as claimed
in the corollary. 
\end{proof}
\begin{app}\label{app1}

In Ito calculus (\cite{Ito06}) the corollary is used for the case
where $F\left(x\right)=e^{-\left|x\right|}$, $x\in\mathbb{R}$; or
the restriction of $F$ to $-1<x<1$. 

If $\mu$ is a probability measure on an interval $J\subset\mathbb{R}$,
Ito defines 
\[
q\left(\mu\right):=\int_{J}\int_{J}e^{-\left|x-y\right|}d\mu\left(x\right)d\mu\left(y\right),\;\mbox{and}
\]
it is called the \textbf{\emph{\uline{degree of concentration}}}.
The number 
\[
\delta\left(\mu\right):=-\log q\left(\mu\right)
\]
is called the \textbf{\emph{\uline{dispersion}}} of $\mu$. 

Note the following:
\[
q\left(\mu\right)=1\Longleftrightarrow\delta\left(\mu\right)=0\Longleftrightarrow\left[\mu=\delta_{x},\;\mbox{for some \ensuremath{x}}\right].
\]

\end{app}
\begin{defn}[The operator $D_{F}$]
\label{def:DF} Let $D_{F}\left(F_{\varphi}\right)=F_{\varphi'}$,
for all $\varphi\in C_{c}^{\infty}\left(0,a\right)$, where $\varphi'=\frac{d\varphi}{dt}$
and $F_{\varphi}$ is as in (\ref{eq:H2}). \end{defn}
\begin{lem}
\label{lem:DF}The operator $D_{F}$ defines a skew-Hermitian operator
with dense domain in $\mathscr{H}_{F}$.\end{lem}
\begin{proof}
By \lemref{dense}, $dom\left(D_{F}\right)$ is dense in $\mathscr{H}_{F}.$
If $\psi\in C_{c}^{\infty}\left(0,a\right)$ and $\left|t\right|<\mathrm{dist}\left(\mathrm{supp}\left(\psi\right),\mbox{endpoints}\right)$,
then
\begin{equation}
\left\Vert F_{\psi\left(\cdot+t\right)}\right\Vert _{\mathscr{H}_{F}}^{2}=\left\Vert F_{\psi}\right\Vert _{\mathscr{H}_{F}}^{2}=\int_{0}^{a}\int_{0}^{a}\overline{\psi\left(x\right)}\psi\left(y\right)F\left(x-y\right)dxdy
\end{equation}
see (\ref{eq:hn2}), so 
\[
\frac{d}{dt}\left\Vert F_{\psi\left(\cdot+t\right)}\right\Vert _{\mathscr{H}_{F}}^{2}=0
\]
which is equivalent to 
\begin{equation}
\left\langle D_{F}F_{\psi},F_{\psi}\right\rangle _{\mathscr{H}_{F}}+\left\langle F_{\psi},D_{F}F_{\psi}\right\rangle _{\mathscr{H}_{F}}=0.
\end{equation}
It follows that $D_{F}$ is well-defined and skew-Hermitian in $\mathscr{H}_{F}$. \end{proof}
\begin{lem}
\label{lem:Dadj}Let $F$ be a positive definite function on $\left(-a,a\right)$,
$0<a<\infty$ fixed. Let $D_{F}$ be as in \defref{DF}, so that $D_{F}\subset D_{F}^{*}$
(\lemref{DF}), where $D_{F}^{*}$ is the adjoint relative to the
$\mathscr{H}_{F}$ inner product. 

Then $\xi\in\mathscr{H}_{F}$ (as a continuous function on $\left[0,a\right]$)
is in $dom\left(D_{F}^{*}\right)$ iff 
\begin{gather}
\xi'\in\mathscr{H}_{F}\;\mbox{where }\xi'=\mbox{distribution derivative, and}\label{eq:Dadj1}\\
D_{F}^{*}\xi=-\xi'\label{eq:Dadj2}
\end{gather}
\end{lem}
\begin{proof}
By \thmref{HF}, a fixed $\xi\in\mathscr{H}_{F}$, i.e., $x\mapsto\xi\left(x\right)$
is a continuous function on $\left[0,a\right]$ s.t. $\exists C$,
$\left|\int_{0}^{a}\varphi\left(x\right)\xi\left(x\right)dx\right|^{2}\leq C\left\Vert F_{\varphi}\right\Vert _{\mathscr{H}_{F}}^{2}$. 

$\xi$ is in $dom\left(D_{F}^{*}\right)$ $\Longleftrightarrow$ $\exists C=C_{\xi}<\infty$
s.t. 
\begin{equation}
\left|\left\langle D_{F}\left(F_{\varphi}\right),\xi\right\rangle _{\mathscr{H}_{F}}\right|^{2}\leq C\left\Vert F_{\varphi}\right\Vert _{\mathscr{H}_{F}}^{2}=C\int_{0}^{a}\int_{0}^{a}\overline{\varphi\left(x\right)}\varphi\left(y\right)F\left(x-y\right)dxdy\label{eq:Dadj3}
\end{equation}
But LHS of (\ref{eq:Dadj3}) under $\left|\left\langle \cdot,\cdot\right\rangle \right|^{2}$
is:
\begin{equation}
\left|\left\langle D_{F}\left(F_{\varphi}\right),\xi\right\rangle _{\mathscr{H}_{F}}\right|^{2}=\left\langle F_{\varphi'},\xi\right\rangle _{\mathscr{H}_{F}}\overset{\left(\ref{eq:F1-1}\right)}{=}\int_{0}^{a}\overline{\varphi'\left(x\right)}\xi\left(x\right)dx,\;\forall\varphi\in C_{c}^{\infty}\left(0,a\right)\label{eq:Dadj4}
\end{equation}
So (\ref{eq:Dadj3}) holds $\Longleftrightarrow$ 
\[
\left|\int_{0}^{a}\overline{\varphi'\left(x\right)}\xi\left(x\right)dx\right|^{2}\leq C\left\Vert F_{\varphi}\right\Vert _{\mathscr{H}_{F}}^{2},\;\forall\varphi\in C_{c}^{\infty}\left(0,a\right)
\]
i.e., 
\[
\left|\int_{0}^{a}\overline{\varphi\left(x\right)}\xi'\left(x\right)dx\right|^{2}\leq C\left\Vert F_{\varphi}\right\Vert _{\mathscr{H}_{F}}^{2},\;\forall\varphi\in C_{c}^{\infty}\left(0,a\right),\;\mbox{and}
\]
$\xi'$ as a distribution is in $\mathscr{H}_{F}$, and 
\[
\int_{0}^{a}\overline{\varphi\left(x\right)}\xi'\left(x\right)dx=\left\langle F_{\varphi},\xi'\right\rangle _{\mathscr{H}_{F}}
\]
where we use the characterization of $\mathscr{H}_{F}$ in (\ref{eq:bdd2}),
i.e., a function $\eta:\left[0,a\right]\rightarrow\mathbb{C}$ is
in $\mathscr{H}_{F}$ $\Longleftrightarrow$ $\exists C<\infty$,
$\left|\int_{0}^{a}\overline{\varphi\left(x\right)}\eta\left(x\right)dx\right|\leq C\left\Vert F_{\varphi}\right\Vert _{\mathscr{H}_{F}}$,
$\forall\varphi\in C_{c}^{\infty}\left(0,a\right)$, and then $\int_{0}^{a}\overline{\varphi\left(x\right)}\eta\left(x\right)dx=\left\langle F_{\varphi},\eta\right\rangle _{\mathscr{H}_{F}}$,
$\forall\varphi\in C_{c}^{\infty}\left(0,a\right)$. See \thmref{HF}.\end{proof}
\begin{cor}
$h\in\mathscr{H}_{F}$ is in $dom\left(\left(D_{F}^{2}\right)^{*}\right)$
iff $h''\in\mathscr{H}_{F}$ ($h''$ distribution derivative) and
$\left(D_{F}^{*}\right)^{2}h=\left(D_{F}^{2}\right)^{*}h=h''$. \end{cor}
\begin{proof}
Application of (\ref{eq:Dadj4}) to $D_{F}\left(F_{\varphi}\right)=F_{\varphi'}$,
we have $D_{F}^{2}\left(F_{\varphi}\right)=F_{\varphi''}=\left(\frac{d}{dx}\right)^{2}F_{\varphi}$,
$\forall\varphi\in C_{c}^{\infty}\left(0,a\right)$, and 
\begin{align*}
\left\langle D_{F}^{2}\left(F_{\varphi}\right),h\right\rangle _{\mathscr{H}_{F}} & =\left\langle F_{\varphi''},h\right\rangle _{\mathscr{H}_{F}}=\int_{0}^{a}\overline{\varphi''\left(x\right)}h\left(x\right)dx\\
 & =\int_{0}^{a}\overline{\varphi\left(x\right)}h''\left(x\right)dx=\left\langle F_{\varphi},\left(D_{F}^{2}\right)^{*}h\right\rangle _{\mathscr{H}_{F}}.
\end{align*}
\end{proof}
\begin{defn}
\cite{DS88b}Let $D_{F}^{*}$ be the adjoint of $D_{F}$ relative
to $\mathscr{H}_{F}$ inner product. The deficiency spaces $DEF^{\pm}$
consists of $\xi_{\pm}\in dom\left(D_{F}^{*}\right)$, such that $D_{F}^{*}\xi_{\pm}=\pm\xi_{\pm}$,
i.e., 
\[
DEF^{\pm}=\left\{ \xi_{\pm}\in\mathscr{H}_{F}:\left\langle F_{\psi'},\xi_{\pm}\right\rangle _{\mathscr{H}_{F}}=\left\langle F_{\psi},\pm\xi_{\pm}\right\rangle _{\mathscr{H}_{F}},\forall\psi\in C_{c}^{\infty}\left(\Omega\right)\right\} .
\]
\end{defn}
\begin{cor}
\label{cor:defg}If $\xi\in DEF^{\pm}$ then $\xi(x)=\mathrm{constant}\, e^{\mp x}.$\end{cor}
\begin{proof}
Immediate from \lemref{Dadj}.
\end{proof}
The role of deficiency indices for the canonical skew-Hermitian operator
$D_{F}$ (\defref{DF}) in the RKHS $\mathscr{H}_{F}$ is as follows:
using von Neumann's conjugation trick \cite{DS88b}, we see that the
deficiency indices can be only $\left(0,0\right)$ or $\left(1,1\right)$.

We conclude that there exists proper skew-adjoint extensions $A\supset D_{F}$
in $\mathscr{H}_{F}$ (in case $D_{F}$ has indices $\left(1,1\right)$).
Then
\begin{equation}
D_{F}\subseteq A=-A^{*}\subseteq-D_{F}^{*}\label{eq:Dadj7}
\end{equation}
(If the indices are $\left(0,0\right)$ then $\overline{D_{F}}=-D_{F}^{*}$;
see \cite{DS88b}.)

Hence, set $U\left(t\right)=e^{tA}:\mathscr{H}_{F}\rightarrow\mathscr{H}_{F}$,
and get the strongly continuous unitary one-parameter group 
\[
\left\{ U\left(t\right):t\in\mathbb{R}\right\} ,\; U\left(s+t\right)=U\left(s\right)U\left(t\right),\:\forall s,t\in\mathbb{R};
\]
and if 
\[
\xi\in dom\left(A\right)=\left\{ \xi\in\mathscr{H}_{F}:\:\mbox{s.t.}\lim_{t\rightarrow0}\frac{U\left(t\right)\xi-\xi}{t}\:\mbox{exists}\right\} 
\]
then 
\begin{equation}
A\xi=\mbox{s.t.}\lim_{t\rightarrow0}\frac{U\left(t\right)\xi-\xi}{t}.
\end{equation}

Now use $F_{x}(\cdot)=F\left(x-\cdot\right)$ defined in $\left(0,a\right)$;
and set 
\begin{equation}
F_{A}\left(t\right):=\left\langle F_{0},U\left(t\right)F_{0}\right\rangle _{\mathscr{H}_{F}},\;\forall t\in\mathbb{R}\label{eq:Fext}
\end{equation}
then using (\ref{eq:approx}), we see that $F_{A}$ is a continuous
positive definite extension of $F$ on $\left(-a,a\right)$. This
extension is in $Ext_{1}\left(F\right)$. 
\begin{cor}
\label{cor:ptspec}Assume $\lambda\in\mathbb{R}$ is in the point
spectrum of $A$, i.e., $\exists\xi_{\lambda}\in dom\left(A\right)$,
$\xi_{\lambda}\neq0$, s.t. $A\xi_{\lambda}=i\lambda\xi_{\lambda}$
holds in $\mathscr{H}_{F}$, then $\xi_{\lambda}=\mbox{const}\cdot e_{\lambda}$,
i.e., 
\begin{equation}
\xi_{\lambda}\left(x\right)=\mbox{const}\cdot e^{i\lambda x},\;\forall x\in\left[0,a\right].\label{eq:Dadj8}
\end{equation}
\end{cor}
\begin{proof}
Assume $\lambda$ is in $spec_{pt}\left(A\right)$, and $\xi_{\lambda}\in dom\left(A\right)$
satisfying 
\begin{equation}
\left(A\xi_{\lambda}\right)\left(x\right)=i\lambda\xi_{\lambda}\left(x\right)\;\mbox{in }\mathscr{H}_{F},\label{eq:Dadj9}
\end{equation}
then since $A\subset-D_{F}^{*}$, we get $\xi\in dom\left(D_{F}^{*}\right)$
by \lemref{Dadj} and (\ref{eq:Dadj7}), and $D_{F}^{*}\xi_{\lambda}=-\xi_{\lambda}'$
where $\xi'$ is the distribution derivative (see (\ref{eq:Dadj2}));
and by (\ref{eq:Dadj7})
\begin{equation}
\left(A\xi_{\lambda}\right)\left(x\right)=-\left(D_{F}^{*}\xi_{\lambda}\right)\left(x\right)=\xi'_{\lambda}\left(x\right)\overset{\left(\ref{eq:Dadj9}\right)}{=}i\lambda\xi_{\lambda}\left(x\right),\;\forall x\in\left(0,a\right)\label{eq:Dadj10}
\end{equation}
so $\xi_{\lambda}$ is the distribution derivative solution to 
\begin{eqnarray}
\xi'_{\lambda}\left(x\right) & = & i\lambda\xi_{\lambda}\left(x\right)\label{eq:Dadj11}\\
 & \Updownarrow\nonumber \\
-\int_{0}^{a}\overline{\varphi'\left(x\right)}\xi_{\lambda}\left(x\right)dx & = & i\lambda\int_{0}^{a}\overline{\varphi\left(x\right)}\xi_{\lambda}\left(x\right)dx,\;\forall\varphi\in C_{c}^{\infty}\left(0,a\right)\nonumber \\
 & \Updownarrow\nonumber \\
-\left\langle D_{F}\left(F_{\varphi}\right),\xi_{\lambda}\right\rangle _{\mathscr{H}_{F}} & = & i\lambda\left\langle F_{\varphi},\xi_{\lambda}\right\rangle _{\mathscr{H}_{F}},\;\forall\varphi\in C_{c}^{\infty}\left(0,a\right).\nonumber 
\end{eqnarray}
But by Schwartz, the distribution solutions to (\ref{eq:Dadj11})
are $\xi_{\lambda}\left(x\right)=\mbox{const}\cdot e_{\lambda}\left(x\right)=\mbox{const}\cdot e^{i\lambda x}$. 
\end{proof}

\section{\label{sec:types}Type I v.s. Type II Extensions}

When a pair $\left(\Omega,F\right)$ is given, where $F$ is a prescribed
continuous positive definite function defined on $\Omega$, we consider
the possible continuous positive definite extensions to all of $\mathbb{R}^{n}$.
The reproducing kernel Hilbert space $\mathscr{H}_{F}$ will play
a key role in our analysis. In constructing various classes of continuous
positive definite extensions to $\mathbb{R}^{n}$, we introduce operators
in $\mathscr{H}_{F}$, and their \emph{dilation} to operators, possibly
acting in an enlargement Hilbert space \cite{JPT14,KL14}. Following
techniques from dilation theory we note that every dilation contains
a minimal one. If a continuous positive definite extensions to $\mathbb{R}^{n}$
has its minimal dilation Hilbert space equal to $\mathscr{H}_{F}$,
we say it is type 1, otherwise we say it is type 2.
\begin{defn}
Let $G$ be a locally compact group, and let $\Omega$ be an open
connected subset of $G$. Let $F:\Omega^{-1}\cdot\Omega\rightarrow\mathbb{C}$
be a continuous positive definite function. 

Consider a strongly continuous unitary representation $U$ of $G$
acting in some Hilbert space $\mathscr{K}$, containing the RKHS $\mathscr{H}_{F}$.
We say that $\left(U,\mathscr{K}\right)\in Ext\left(F\right)$ iff
there is a vector $k_{0}\in\mathscr{K}$ such that
\begin{equation}
F\left(g\right)=\left\langle k_{0},U\left(g\right)k_{0}\right\rangle _{\mathscr{K}},\;\forall g\in\Omega^{-1}\cdot\Omega.\label{eq:ext-1-1}
\end{equation}

\begin{enumerate}[leftmargin=*,label=\arabic{enumi}.]
\item The subset of $Ext\left(F\right)$ consisting of $\left(U,\mathscr{H}_{F},k_{0}=F_{e}\right)$
with 
\begin{equation}
F\left(g\right)=\left\langle F_{e},U\left(g\right)F_{e}\right\rangle _{\mathscr{H}_{F}},\;\forall g\in\Omega^{-1}\cdot\Omega\label{eq:ext-1-2}
\end{equation}
is denoted $Ext_{1}\left(F\right)$; and we set 
\[
Ext_{2}\left(F\right):=Ext\left(F\right)\backslash Ext_{1}\left(F\right);
\]
i.e., $Ext_{2}\left(F\right)$, consists of the solutions to problem
(\ref{eq:ext-1-1}) for which $\mathscr{K}\supsetneqq\mathscr{H}_{F}$,
i.e., unitary representations realized in an enlargement Hilbert space.
\\
(We write $F_{e}\in\mathscr{H}_{F}$ for the vector satisfying $\left\langle F_{e},\xi\right\rangle _{\mathscr{H}_{F}}=\xi\left(e\right)$,
$\forall\xi\in\mathscr{H}_{F}$, where $e$ is the neutral (unit)
element in $G$, i.e., $e\, g=g$, $\forall g\in G$.)
\item In the special case, where $G=\mathbb{R}^{n}$, and $\Omega\subset\mathbb{R}^{n}$
is open and connected, we consider 
\[
F:\Omega-\Omega\rightarrow\mathbb{C}
\]
continuous and positive definite. In this case,
\begin{align}
Ext\left(F\right)= & \Bigl\{\mu\in\mathscr{M}_{+}\left(\mathbb{R}^{n}\right)\:\big|\:\widehat{\mu}\left(x\right)=\int_{\mathbb{R}^{n}}e^{i\lambda\cdot x}d\mu\left(\lambda\right)\label{eq:ext-1-4}\\
 & \mbox{ is a p.d. extensiont of \ensuremath{F}}\Bigr\}.\nonumber 
\end{align}

\end{enumerate}
\end{defn}
\begin{rem}
Note that (\ref{eq:ext-1-4}) is consistent with (\ref{eq:ext-1-1}):
For if $\left(U,\mathscr{K},k_{0}\right)$ is a unitary representation
of $G=\mathbb{R}^{n}$, such that (\ref{eq:ext-1-1}) holds; then,
by a theorem of Stone, there is a projection-valued measure (PVM)
$P_{U}\left(\cdot\right)$, defined on the Borel subsets of $\mathbb{R}^{n}$
s.t. 
\begin{equation}
U\left(x\right)=\int_{\mathbb{R}^{n}}e^{i\lambda\cdot x}P_{U}\left(d\lambda\right),\; x\in\mathbb{R}^{n}.\label{eq:ex-1-5}
\end{equation}
Setting 
\begin{equation}
d\mu\left(\lambda\right):=\left\Vert P_{U}\left(d\lambda\right)k_{0}\right\Vert _{\mathscr{K}}^{2},\label{eq:ext-1-6-7}
\end{equation}
it is then immediate that we have: $\mu\in\mathscr{M}_{+}\left(\mathbb{R}^{n}\right)$,
and that the finite measure $\mu$ satisfies 
\begin{equation}
\widehat{\mu}\left(x\right)=F\left(x\right),\;\forall x\in\Omega-\Omega.\label{eq:ext-1-6}
\end{equation}

\end{rem}
Set $n=1$: Start with a local p.d. continuous function $F$, and
let $\mathscr{H}_{F}$ be the corresponding RKHS. Let $Ext(F)$ be
the compact convex set of probability measures on $\mathbb{R}$ defining
extensions of $F$.

We now divide $Ext(F)$ into two parts, say $Ext_{1}\left(F\right)$
and $Ext_{2}\left(F\right)$. 

All continuous p.d. extensions of $F$ come from strongly continuous
unitary representations. So in the case of 1D, from unitary one-parameter
groups of course, say $U(t)$.

Let $Ext_{1}\left(F\right)$ be the subset of $Ext(F)$ corresponding
to extensions when the unitary representation $U(t)$ acts in $\mathscr{H}_{F}$
(internal extensions), and $Ext_{2}\left(F\right)$ denote the part
of $Ext(F)$ associated to unitary representations $U(t)$ acting
in a proper enlargement Hilbert space $\mathscr{K}$ (if any), i.e.,
acting in a Hilbert space $\mathscr{K}$ corresponding to a proper
dilation of $\mathscr{H}_{F}$. 

In \secref{exp} below, we explore the example $F\left(x\right)=e^{-\left|x\right|}$,
initially defined only in the symmetric interval $\left(-1,1\right)$;
and we characterize the corresponding continuous p.d. extensions in
$\mathscr{H}_{F}$ (i.e. type 1.) Section \ref{sub:type2ext} is devoted
to some examples of type 2 extensions.

\section{\label{sec:mercer}Mercer Operators}

In the considerations below, we shall be primarily concerned with
the case when a fixed continuous p.d. function $F$ is defined on
a finite interval $\left(-a,a\right)\subset\mathbb{R}$. In this case,
by a Mercer operator, we mean an operator $T_{F}$ in $L^{2}\left(0,a\right)$
where $L^{2}\left(0,a\right)$ is defined from Lebesgue measure on
$\left(0,a\right)$, given by
\begin{equation}
\left(T_{F}\varphi\right)\left(x\right):=\int_{0}^{a}\varphi\left(y\right)F\left(x-y\right)dy,\;\forall\varphi\in L^{2}\left(0,a\right),\forall x\in\left(0,a\right).\label{eq:mer-1}
\end{equation}

\begin{lem}
\label{lem:mer-1}Under the assumptions stated above, the Mercer operator
$T_{F}$ is trace class in $L^{2}\left(0,a\right)$; and if $F\left(0\right)=1$,
then 
\begin{equation}
trace\left(T_{F}\right)=a.\label{eq:mer-2}
\end{equation}
\end{lem}
\begin{proof}
This is an application of Mercer's theorem \cite{LP89,FR42,FM13}
to the integral operator $T_{F}$ in (\ref{eq:mer-1}). But we must
check that $F$, on $\left(-a,a\right)$, extends uniquely by limit
to a continuous p.d. function $F_{ex}$ on $\left[-a,a\right]$, the
closed interval. This is true, and easy to verify, see e.g. \cite{JPT14}.\end{proof}
\begin{cor}
\label{cor:mer1}Let $F$ and $\left(-a,a\right)$ be as in \lemref{mer-1}.
Then there is a sequence $\left(\lambda_{n}\right)_{n\in\mathbb{N}}$,
$\lambda_{n}>0$, s.t. $\sum_{n\in\mathbb{N}}\lambda_{n}=a$, and
a system of orthogonal functions $\left\{ \xi_{n}\right\} \subset L^{2}\left(0,a\right)\cap\mathscr{H}_{F}$
such that
\begin{equation}
F\left(x-y\right)=\sum_{n\in\mathbb{N}}\lambda_{n}\xi_{n}\left(x\right)\overline{\xi_{n}\left(y\right)},\mbox{ and}\label{eq:mer-3}
\end{equation}
\begin{equation}
\int_{0}^{a}\overline{\xi_{n}\left(x\right)}\xi_{m}\left(x\right)dx=\delta_{n,m},\; n,m\in\mathbb{N}.\label{eq:mer-4}
\end{equation}
\end{cor}
\begin{proof}
An application of Mercer's theorem \cite{LP89,FR42,FM13}.\end{proof}
\begin{cor}
\label{cor:merinn}For all $\psi,\varphi\in C_{c}^{\infty}\left(0,a\right)$,
we have 
\begin{equation}
\left\langle F_{\psi},F_{\varphi}\right\rangle _{\mathscr{H}_{F}}=\left\langle F_{\psi},T_{F}^{-1}F_{\varphi}\right\rangle _{2}.\label{eq:mnorm1}
\end{equation}
Consequently, 
\begin{equation}
\left\Vert h\right\Vert _{\mathscr{H}_{F}}=\Vert T_{F}^{-1/2}h\Vert_{2},\;\forall h\in\mathscr{H}_{F}.\label{eq:mnorm2}
\end{equation}
\end{cor}
\begin{proof}
Note 
\begin{align*}
\left\langle F_{\psi},T_{F}^{-1}F_{\varphi}\right\rangle _{2} & =\left\langle F_{\psi},T_{F}^{-1}T_{F}\varphi\right\rangle _{2}=\left\langle F_{\psi},\varphi\right\rangle _{2}\\
 & =\int_{0}^{a}\overline{\left(\int_{0}^{a}\psi\left(x\right)F\left(y-x\right)dx\right)}\,\varphi\left(y\right)dy\\
 & =\int_{0}^{a}\int_{0}^{a}\overline{\psi\left(x\right)}\varphi\left(y\right)F\left(x-y\right)dxdy=\left\langle F_{\psi},F_{\varphi}\right\rangle _{\mathscr{H}_{F}}.
\end{align*}
\end{proof}
\begin{cor}
Let $\left\{ \xi_{n}\right\} $ be the ONB in $L^{2}\left(0,a\right)$
as in \corref{mer1}; then $\left\{ \sqrt{\lambda_{n}}\xi_{n}\right\} $
is an ONB in $\mathscr{H}_{F}$. \end{cor}
\begin{proof}
The functions $\xi_{n}$ are in $\mathscr{H}_{F}$ by \thmref{HF}.
We check directly (\corref{merinn}) that 
\begin{align*}
\left\langle \sqrt{\lambda_{n}}\xi_{n},\sqrt{\lambda_{m}}\xi_{m}\right\rangle _{\mathscr{H}_{F}} & =\sqrt{\lambda_{n}\lambda_{m}}\left\langle \xi_{n},T^{-1}\xi_{m}\right\rangle _{2}\\
 & =\sqrt{\lambda_{n}\lambda_{m}}\lambda_{m}^{-1}\left\langle \xi_{n},\xi_{m}\right\rangle _{2}=\delta_{n,m}.
\end{align*}

\end{proof}

\section{\label{sec:exp}The Case of $F\left(x\right)=e^{-\left|x\right|}$,
$\left|x\right|<1$}

Our emphasis is von Neumann indices, and explicit formulas for partially
defined positive definite functions $F$, defined initially only on
a symmetric interval $\left(-a,a\right)$. It turns out that a number
of the features we explore are common to the following example: Let
$F\left(x\right):=e^{-\left|x\right|}$, defined for $\left|x\right|<1$.
The present section is devoted to this example.

\subsection{\label{sub:saext}The Selfadjoint Extensions $A_{\theta}\supset-iD_{F}$}

The notation ``$\supseteq$'' above refers to containment of operators,
or rather of the respective graphs of the two operators; see \cite{DS88b}.

\begin{lem}
\label{lem:dev1}Let $F\left(x\right)=e^{-\left|x\right|}$, $\left|x\right|<1$.
Set $F_{x}\left(y\right):=F\left(x-y\right)$, $\forall x,y\in\left(0,1\right)$;
and $F_{\varphi}\left(x\right)=\int_{0}^{1}\varphi\left(y\right)F\left(x-y\right)dy$,
$\forall\varphi\in C_{c}^{\infty}\left(0,1\right)$. Define $D_{F}\left(F_{\varphi}\right)=F_{\varphi'}$
on the dense subset 
\begin{equation}
dom\left(D_{F}\right)=\left\{ F_{\varphi}:\varphi\in C_{c}^{\infty}\left(0,1\right)\right\} \subset\mathscr{H}_{F}.\label{eq:D}
\end{equation}
Then the skew-Hermitian operator $D_{F}$ has deficiency indices $\left(1,1\right)$
in $\mathscr{H}_{F}$, where the defect vectors are 
\begin{align}
\xi_{+}\left(x\right) & =F_{0}\left(x\right)=e^{-x}\label{eq:edev1}\\
\xi_{-}\left(x\right) & =F_{1}\left(x\right)=e^{x-1};\label{eq:edev2}
\end{align}
moreover, 
\begin{equation}
\left\Vert \xi_{+}\right\Vert _{\mathscr{H}_{F}}=\left\Vert \xi_{+}\right\Vert _{\mathscr{H}_{F}}=1.\label{eq:edev3}
\end{equation}
\end{lem}
\begin{proof}
(Note if $\Omega$ is any bounded, open and connected domain in $\mathbb{R}^{n}$,
then a locally defined continuous p.d. function, $F:\Omega-\Omega:\rightarrow\mathbb{C}$,
extends uniquely to the boundary $\partial\Omega:=\overline{\Omega}\backslash\Omega$
by continuity \cite{JPT14}.)

In our current settings, $\Omega=\left(0,1\right)$, and $F_{x}\left(y\right):=F\left(x-y\right)$,
$\forall x,y\in\left(0,1\right)$. Thus, $F_{x}\left(y\right)$ extends
to all $x,y\in\left[0,1\right]$. In particular, 
\[
F_{0}\left(x\right)=e^{-x},\; F_{1}\left(x\right)=e^{x-1}
\]
are the two defect vectors, as shown in \corref{defg}. Moreover,
using the reproducing property, we have 
\begin{align*}
\left\Vert F_{0}\right\Vert _{\mathscr{H}_{F}}^{2} & =\left\langle F_{0},F_{0}\right\rangle _{\mathscr{H}_{F}}=F_{0}\left(0\right)=F\left(0\right)=1\\
\left\Vert F_{1}\right\Vert _{\mathscr{H}_{F}}^{2} & =\left\langle F_{1},F_{1}\right\rangle _{\mathscr{H}_{F}}=F_{1}\left(1\right)=F\left(0\right)=1
\end{align*}
and (\ref{eq:edev3}) follows. For more details, see \cite[lemma 2.10.14]{JPT14}.\end{proof}
\begin{lem}
\label{lem:dev}Let $F$ be any continuous p.d. function on $\left(-1,1\right)$.
Set 
\[
h\left(x\right)=\int_{0}^{1}\varphi\left(y\right)F\left(x-y\right)dy,\;\forall\varphi\in C_{c}^{\infty}\left(0,1\right);
\]
then 
\begin{align}
h\left(0\right) & =\int_{0}^{1}\varphi\left(y\right)F\left(-y\right)dy,\qquad h\left(1\right)=\int_{0}^{1}\varphi\left(y\right)F\left(1-y\right)dy\label{eq:dev-1-1}\\
h'\left(0\right) & =\int_{0}^{1}\varphi\left(y\right)F'\left(-y\right)dy,\quad\,\,\, h'\left(1\right)=\int_{0}^{1}\varphi\left(y\right)F'\left(1-y\right)dy;\label{eq:dev-1-2}
\end{align}
where the derivatives $F'$ in (\ref{eq:dev-1-1})-(\ref{eq:dev-1-2})
are in the sense of distribution.\end{lem}
\begin{proof}
Note that
\begin{align*}
h\left(x\right) & =\int_{0}^{x}\varphi\left(y\right)F\left(x-y\right)dy+\int_{x}^{1}\varphi\left(y\right)F\left(x-y\right)dy;\\
h'\left(x\right) & =\int_{0}^{x}\varphi\left(y\right)F'\left(x-y\right)dy+\int_{x}^{1}\varphi\left(y\right)F'\left(x-y\right)dy.
\end{align*}
and so (\ref{eq:dev-1-1})-(\ref{eq:dev-1-2}) follow.
\end{proof}
We now specialize to the function $F\left(x\right)=e^{-\left|x\right|}$
defined in $\left(-1,1\right)$.
\begin{cor}
\label{cor:dev1}For $F\left(x\right)=e^{-\left|x\right|}$, $\left|x\right|<1$,
set $h=T_{F}\varphi$, i.e., 
\[
h:=F_{\varphi}=\int_{0}^{1}\varphi\left(y\right)F\left(\cdot-y\right)dy,\;\forall\varphi\in C_{c}^{\infty}\left(0,1\right);
\]
then 
\begin{align}
h\left(0\right) & =\int_{0}^{1}\varphi\left(y\right)e^{-y}dy,\qquad h\left(1\right)=\int_{0}^{1}\varphi\left(y\right)e^{y-1}dy\label{eq:edev4}\\
h'\left(0\right) & =\int_{0}^{1}\varphi\left(y\right)e^{-y}dy,\quad\,\,\,\, h'\left(1\right)=-\int_{0}^{1}\varphi\left(y\right)e^{y-1}dy\label{eq:edev5}
\end{align}
In particular, 
\begin{align}
h\left(0\right)-h'\left(0\right) & =0\label{eq:dev6}\\
h\left(1\right)+h'\left(1\right) & =0.\label{eq:dev7}
\end{align}
\end{cor}
\begin{proof}
Immediately from \lemref{dev}. Specifically, 
\begin{align*}
h\left(x\right) & =e^{-x}\int_{0}^{x}\varphi\left(y\right)e^{y}dy+e^{x}\int_{x}^{1}\varphi\left(y\right)e^{-y}dy\\
h'\left(x\right) & =-e^{-x}\int_{0}^{x}\varphi\left(y\right)e^{y}dy+e^{x}\int_{x}^{1}\varphi\left(y\right)e^{-y}dy.
\end{align*}
Setting $x=0$ and $x=1$ gives the desired conclusions.\end{proof}
\begin{rem}
The space 
\[
\Big\{ h\in\mathscr{H}_{F}\:\big|\: h\left(0\right)-h'\left(0\right)=0,\; h\left(1\right)+h'\left(1\right)=0\Big\}
\]
is dense in $\mathscr{H}_{F}$. This is because it contains $\left\{ F_{\varphi}\:\big|\:\varphi\in C_{c}^{\infty}\left(0,1\right)\right\} $.
Note
\begin{align*}
F_{0}+F_{0}' & =-\delta_{0},\;\mbox{and}\\
F_{1}-F_{1}' & =-\delta_{1};
\end{align*}
however,  $\delta_{0},\delta_{1}\notin\mathscr{H}_{F}$.
\end{rem}
By von Neumann's theory \cite{DS88b} and \lemref{Dadj}, the family
of selfadjoint extensions of the Hermitian operator $-iD_{F}$ is
characterized by
\begin{gather}
\begin{split}A_{\theta}\left(h+c\left(e^{-x}+e^{i\theta}e^{x-1}\right)\right)=-i\, h'+i\, c\left(e^{-x}-e^{i\theta}e^{x-1}\right),\;\mbox{where}\\
dom\left(A_{\theta}\right):=\left\{ h+c\left(e^{-x}+e^{i\theta}e^{x-1}\right)\:\big|\: h\in dom\left(D_{F}\right),c\in\mathbb{C}\right\} .
\end{split}
\label{eq:saextA}
\end{gather}

\begin{rem}
In (\ref{eq:saextA}), $h\in dom\left(D_{F}\right)$ (see (\ref{eq:D})),
and by \corref{dev1}, $h$ satisfies the boundary conditions (\ref{eq:dev6})-(\ref{eq:dev7}).
Also, by \lemref{dev1}, $\xi_{+}=F_{0}=e^{-x}$, $\xi_{-}=F_{1}=e^{x-1}$,
and $\left\Vert \xi_{+}\right\Vert _{\mathscr{H}_{F}}=\left\Vert \xi_{-}\right\Vert _{\mathscr{H}_{F}}=1$.\end{rem}
\begin{prop}
\label{prop:dev1}Let $A_{\theta}$ be a selfadjoint extension of
$-iD$ as in (\ref{eq:saextA}). Then, 
\begin{equation}
\psi\left(1\right)+\psi'\left(1\right)=e^{i\theta}\left(\psi\left(0\right)-\psi'\left(0\right)\right),\;\forall\psi\in dom\left(A_{\theta}\right).\label{eq:dev-bd}
\end{equation}
\end{prop}
\begin{proof}
Any $\psi\in dom\left(A_{\theta}\right)$ has the decomposition 
\[
\psi\left(x\right)=h\left(x\right)+c\left(e^{-x}+e^{i\theta}e^{x-1}\right)
\]
where $h\in dom\left(D_{F}\right)$, and $c\in\mathbb{C}$. An application
of \corref{ptspec} gives 
\begin{align*}
\psi\left(1\right)+\psi'\left(1\right) & =\underset{=0\:\left(\text{by }\left(\ref{eq:dev7}\right)\right)}{\underbrace{h\left(1\right)+h'\left(1\right)}}+c\left(e^{-1}+e^{i\theta}\right)+c\left(-e^{-1}+e^{i\theta}\right)=2c\, e^{i\theta}\\
\psi\left(0\right)-\psi'\left(0\right) & =\underset{=0\;\left(\text{by }\left(\ref{eq:dev6}\right)\right)}{\underbrace{h\left(0\right)-h'\left(0\right)}}+c\left(1+e^{-1}e^{i\theta}\right)-c\left(-1+e^{-1}e^{i\theta}\right)=2c
\end{align*}
which is the assertion in (\ref{eq:dev-bd}).\end{proof}
\begin{cor}
\label{cor:spext}Let $A_{\theta}$ be a selfadjoint extension of
$-iD_{F}$ as in (\ref{eq:saextA}). Fix $\lambda\in\mathbb{R}$,
then $\lambda\in spec_{pt}\left(A_{\theta}\right)$ $\Longleftrightarrow$
$e_{\lambda}\left(x\right):=e^{i\lambda x}\in dom\left(A_{\theta}\right)$,
and $\lambda$ is a solution to the following equation: 
\begin{equation}
\lambda=\theta+\tan^{-1}\left(\frac{2\lambda}{\lambda^{2}-1}\right)+2n\pi,\; n\in\mathbb{Z}.\label{eq:dev9}
\end{equation}
\end{cor}
\begin{proof}
By assumption, $e^{i\lambda x}\in dom\left(A_{\theta}\right)$, so
$\exists h_{\lambda}\in dom\left(D_{F}\right)$, and $\exists c_{\lambda}\in\mathbb{C}$
s.t. 
\begin{equation}
e^{i\lambda x}=h_{\lambda}\left(x\right)+c_{\lambda}\left(e^{x}+e^{i\theta}e^{x-1}\right).\label{eq:dev8}
\end{equation}
Applying the boundary condition in \propref{dev1}, we have 
\[
e^{i\lambda}+i\lambda e^{i\lambda}=e^{i\theta}\left(1-i\lambda\right);\;\mbox{i.e.,}
\]
\begin{equation}
e^{i\lambda}=e^{i\theta}\frac{1-i\lambda}{1+i\lambda}=e^{i\theta}e^{i\arg\left(\frac{1-i\lambda}{1+i\lambda}\right)}\label{eq:dev8-1}
\end{equation}
where 
\[
\arg\left(\frac{1-i\lambda}{1+i\lambda}\right)=\tan^{-1}\left(\frac{2\lambda}{\lambda^{2}-1}\right)
\]
and (\ref{eq:dev9}) follows. For a discrete set of solutions, see
\figref{expExpSp}.
\end{proof}
\begin{figure}[H]
\includegraphics[scale=0.6]{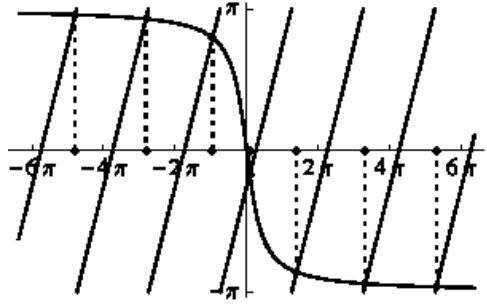}

\protect\caption{\label{fig:expExpSp}Fix $\theta=0.8$, $\Lambda_{\theta}=\left\{ \lambda_{n}\left(\theta\right)\right\} $
= intersections of two curves. }

\end{figure}

\begin{cor}
If $A_{\theta}\supset-iD_{F}$ is a selfadjoint extension in $\mathscr{H}_{F}$,
then
\begin{align*}
spect\left(A_{\theta}\right)= & \left\{ \lambda\in\mathbb{R}\:\big|\: e_{\lambda}\in\mathscr{H}_{F}\:\mbox{satisfying }\left(\ref{eq:dev-bd}\right)\right\} \\
= & \big\{\lambda\in\mathbb{R}\:\big|\: e_{\lambda}\in\mathscr{H}_{F},\: e_{\lambda}=h_{\lambda}+c_{\lambda}\left(e^{x}+e^{i\theta}e^{x-1}\right),\\
 & \quad h_{\lambda}\in dom\left(D_{F}\right),\; c_{\lambda}\in\mathbb{C}\big\}.
\end{align*}
\end{cor}
\begin{rem}
The corollary holds for \uline{all} continuous p.d. functions $F:\left(-a,a\right)\rightarrow\mathbb{C}$.\end{rem}
\begin{cor}
\label{cor:spdiscrete}All selfadjoint extensions $A_{\theta}\supset-iD_{F}$
have purely atomic spectrum; i.e., 
\begin{equation}
\Lambda_{\theta}:=spect\left(A_{\theta}\right)=\mbox{discrete subset in }\mathbb{R}.\label{eq:Aspect1}
\end{equation}
And for all $\lambda\in\Lambda_{\theta}$, 
\begin{equation}
ker\left(A_{\theta}-\lambda I_{\mathscr{H}_{F}}\right)=\mathbb{C}e_{\lambda},\;\mbox{where }e_{\lambda}\left(x\right)=e^{i\lambda x}\label{eq:Aspect2}
\end{equation}
i.e., all eigenvalues have multiplicity $1$. (The set $\Lambda_{\theta}$
will be denoted $\left\{ \lambda_{n}\left(\theta\right)\right\} _{n\in\mathbb{Z}}$
following Fig. \ref{fig:expExpSp}. )\end{cor}
\begin{proof}
This follows by solving eq. (\ref{eq:dev9}).\end{proof}
\begin{cor}
Let $A$ be a selfadjoint extension of $-iD_{F}$ as before. Suppose
$\lambda_{1},\lambda_{2}\in spec\left(A\right)$, $\lambda_{1}\neq\lambda_{2}$,
then $e_{\lambda_{i}}\in\mathscr{H}_{F}$, $i=1,2$; and $\left\langle e_{\lambda_{1}},e_{\lambda_{2}}\right\rangle _{\mathscr{H}_{F}}=0$. \end{cor}
\begin{proof}
Let $\lambda_{1},\lambda_{2}$ be as in the statement, then 
\[
\left(\lambda_{1}-\lambda_{2}\right)\left\langle e_{\lambda_{1}},e_{\lambda_{2}}\right\rangle _{\mathscr{H}_{F}}=\left\langle Ae_{\lambda_{1}},e_{\lambda_{2}}\right\rangle _{\mathscr{H}_{F}}-\left\langle e_{\lambda_{1}},Ae_{\lambda_{2}}\right\rangle _{\mathscr{H}_{F}}=0;
\]
so since $\lambda_{1}-\lambda_{2}\neq0$, we get $\left\langle e_{\lambda_{1}},e_{\lambda_{2}}\right\rangle _{\mathscr{H}_{F}}=0$. 
\end{proof}
For explicit computations regarding these points, see also Corollaries
\ref{cor:expinner}, \ref{cor:exporg}, and \ref{cor:Fext} below.

\subsection{The Operator $T_{F}$}

By $T_{F}$ we mean the Mercer operator (\ref{eq:mer-1}) in $L^{2}\left(0,1\right)$.
Let the initial positive definite function $F$ be as above. Below
we show that the corresponding integral operator $T_{F}$ in $L^{2}\left(0,1\right)$
is the sum of a \emph{Volterra operator} and a \emph{rank-one operator},
see (\ref{eq:sh1}).

Bringing in the Volterra operator $V$ (see Lemma \ref{lem:vot} below)
facilitates calculation of the spectrum for the rank-one perturbation
of $V$, $T_{F}=(\mbox{rank-one})+V$, as it allows us to take advantage
of the known filter of invariant subspace for $V$; see e.g., \cite{Ar85,BO05}.
\begin{lem}
\label{lem:vot}Let $F\left(x\right)=e^{-\left|x\right|}$, $\left|x\right|<1$,
and let $T_{F}:L^{2}\left(0,1\right)\rightarrow L^{2}\left(0,1\right)$
be the Mercer operator, i.e., $\left(T_{F}f\right)\left(x\right)=\int_{0}^{1}f\left(y\right)F\left(x-y\right)dy$.
Then, 
\begin{equation}
\left(T_{F}f\right)\left(x\right)=2\int_{0}^{x}\sinh\left(y-x\right)f\left(y\right)dy+e^{x}\int_{0}^{1}\varphi\left(y\right)e^{-y}dy,\;\forall f\in L^{2}\left(0,1\right);\label{eq:sh1}
\end{equation}
which is a rank-1 perturbation of the Volterra operator
\[
\left(Vf\right)\left(x\right)=2\int_{0}^{x}\sinh\left(y-x\right)f\left(y\right)dy,\;\forall f\in L^{2}\left(0,1\right).
\]
\end{lem}
\begin{proof}
Recall that $T_{F}^{-1}$ is a selfadjoint extension of the elliptic
operator 
\[
\tfrac{1}{2}\left(I-\left(\tfrac{d}{dx}\right)^{2}\right)\Big|_{C_{c}^{\infty}\left(0,1\right)}
\]
in $L^{2}\left(0,1\right)$. Thus for $h:=T_{F}f$, $f\in L^{2}\left(0,1\right)$,
we have 
\begin{equation}
h''=h-2f.\label{eq:sh}
\end{equation}
Now we solve the ODE in (\ref{eq:sh}) for $h$. 

For this, write $h=h_{p}+\mbox{homogeneous soln}$. 

Let $h_{p}:=2\int_{0}^{x}\sinh\left(y-x\right)f\left(y\right)dy$
be as in (\ref{eq:sh1}). We check directly that 
\begin{align*}
h'_{p}\left(x\right) & =-2\int_{0}^{x}\sinh\left(y-x\right)f\left(y\right)dy\\
h''_{p}\left(x\right) & =2\int_{0}^{x}\sinh\left(y-x\right)f\left(y\right)-2f\left(x\right)=h\left(x\right)-2f\left(x\right)
\end{align*}
i.e., $h''_{p}=h_{p}-2f$. So $h_{p}$ is a particular solution.

Therefore,
\begin{equation}
h\left(x\right)=\left(T_{F}f\right)\left(x\right)=\underset{\text{homogeneous soln}}{\underbrace{A\cosh\left(x\right)+B\sinh\left(x\right)}}+h_{p}\left(x\right)\label{eq:sh2}
\end{equation}
Applying the boundary condition $h_{p}\left(0\right)=h'_{p}\left(0\right)=0$
yields 
\[
A=B=\int_{0}^{1}e^{-y}f\left(y\right)dy;
\]
substituting into (\ref{eq:sh2}) gives 
\[
h\left(x\right)=\left(T_{F}f\right)\left(x\right)=e^{x}\int_{0}^{1}\varphi\left(y\right)e^{-y}dy+h_{p}\left(x\right).
\]
This is the desired statement in (\ref{eq:sh1}).

Below, we offer a more direct argument:

\begin{align*}
\left(T_{F}\varphi\right)\left(x\right) & =\int_{0}^{1}\varphi\left(y\right)e^{-\left|x-y\right|}dy\\
 & =\int_{0}^{x}\varphi\left(y\right)e^{y-x}dy+\int_{x}^{1}\varphi\left(y\right)e^{-\left(y-x\right)}dx\\
 & =\int_{0}^{x}\varphi\left(y\right)(\underset{2\sinh\left(y-x\right)}{\underbrace{e^{y-x}-e^{x-y}}}+e^{x-y})dy+\int_{x}^{1}\varphi\left(y\right)e^{-\left(y-x\right)}dx\\
 & =2\int_{0}^{x}\varphi\left(y\right)\sinh\left(y-x\right)dy+e^{x}\int_{0}^{x}\varphi\left(y\right)e^{-y}dy+e^{x}\int_{x}^{1}\varphi\left(y\right)e^{-y}dx\\
 & =2\int_{0}^{x}\varphi\left(y\right)\sinh\left(y-x\right)dy+e^{x}\int_{0}^{1}\varphi\left(y\right)e^{-y}dy.
\end{align*}

\end{proof}
We now show that $T_{F}$ is a Greens function for $\frac{1}{2}\left(I-\left(\frac{d}{dx}\right)^{2}\right)$.
\begin{lem}
\label{lem:Tf}For all $\varphi\in C_{c}^{\infty}\left(0,1\right)$,
set 
\[
f\left(x\right)=F_{\varphi}\left(x\right)=\left(T_{F}\varphi\right)\left(x\right)=\int_{0}^{1}\varphi\left(y\right)e^{-\left|x-y\right|}dy,\;0\leq x\leq1.
\]
Then
\begin{equation}
\varphi=T_{F}^{-1}f=\frac{1}{2}\left(f-\left(\tfrac{d}{dx}\right)^{2}f\right)\label{eq:u-0}
\end{equation}
with boundary condition on $f$ as follows:
\begin{align*}
f\left(0\right) & =f'\left(0\right)=\int_{0}^{1}e^{-y}\varphi\left(y\right)dy\\
f\left(1\right) & =-f'\left(1\right)=e^{-1}\int_{0}^{1}e^{y}\varphi\left(y\right)dy.
\end{align*}
\end{lem}
\begin{proof}[Proof of \lemref{Tf}]
From the definition of $f$, we have 
\begin{align*}
f\left(x\right) & =e^{-x}\int_{0}^{x}\varphi\left(y\right)e^{y}dy+e^{x}\int_{x}^{1}e^{-y}\varphi\left(y\right)dy\\
f'\left(x\right) & =-e^{-x}\int_{0}^{x}\varphi\left(y\right)e^{y}dy+e^{x}\int_{x}^{1}\varphi\left(y\right)e^{-y}dy\\
f''\left(x\right) & =f\left(x\right)-2\varphi\left(x\right);
\end{align*}
i.e., 
\[
\varphi=T_{F}^{-1}f=\frac{1}{2}\left(f-\left(\tfrac{d}{dx}\right)^{2}f\right).
\]
The boundary conditions on $f$ follow from \corref{dev1}, eqs. (\ref{eq:dev6})-(\ref{eq:dev7}).
\end{proof}
Let $F:\left(-1,1\right)\rightarrow\mathbb{C}$ be a continuous p.d.
function. Recall the Mercer operator $T_{F}:L^{2}\left(0,1\right)\rightarrow L^{2}\left(0,1\right)$,
by 
\[
\left(T_{F}\varphi\right)\left(x\right)=\int_{0}^{1}\varphi\left(y\right)F\left(x-y\right)dy,\;\forall\varphi\in C_{c}^{\infty}\left(0,1\right);
\]
which is positive, selfadjoint, and trace class. 

In the case $F\left(x\right)=e^{-\left|x\right|}$, $\left|x\right|<1,$
we see in (\ref{eq:u-0}) that 
\begin{equation}
T_{F}^{-1}\supset\tfrac{1}{2}\left(I-\left(\tfrac{d}{dx}\right)^{2}\right).\label{eq:u-1-1}
\end{equation}

\begin{prop}
Let 
\begin{gather}
\mathscr{D}:=\tfrac{1}{2}\left(I-\left(\tfrac{d}{dx}\right)^{2}\right),\;\mbox{where}\label{eq:u-3-1}\\
dom\left(\mathscr{D}\right):=\left\{ F_{\varphi}:\varphi\in C_{c}^{\infty}\left(0,1\right)\right\} \label{eq:u-3-2}
\end{gather}
acting in $L^{2}\left(0,1\right)$. 

Then $\mathscr{D}=T_{F}^{-1}$ (see (\ref{eq:u-1-1})); in particular,
$\mathscr{D}$ is selfadjoint and having purely atomic spectrum. Moreover,
$\lambda\in spect\left(\mathscr{D}\right)$ iff it satisfies the following
equation:
\begin{equation}
\tan k=\frac{2k}{k^{2}-1},\; k^{2}>1\label{eq:Dexp}
\end{equation}
where $k^{2}:=2\lambda-1$. See \figref{TFexp}.\end{prop}
\begin{proof}
Suppose $f,g\in dom\left(\mathscr{D}\right)$, then $f,g$ satisfy
(see (\ref{eq:dev6})-(\ref{eq:dev7})) 
\[
h\left(0\right)-h'\left(0\right)=h\left(1\right)+h'\left(1\right)=0;\;\mbox{and so}
\]
\begin{align*}
\left\langle f,g-g''\right\rangle _{2}-\left\langle f-f'',g\right\rangle _{2}= & \left[\overline{f'}g-\overline{f}g'\right]\left(0\right)-\left[\overline{f'}g-\overline{f}g'\right]\left(1\right)\\
= & \overline{f'\left(0\right)}g\left(0\right)-\overline{f\left(0\right)}g'\left(0\right)-\overline{f'\left(1\right)}g\left(1\right)+\overline{f\left(1\right)}g'\left(1\right)\\
= & \overline{f\left(0\right)}g\left(0\right)-\overline{f\left(0\right)}g\left(0\right)+\overline{f\left(1\right)}g\left(1\right)-\overline{f\left(1\right)}g\left(1\right)=0.
\end{align*}
It follows that $\mathscr{D}$ is selfadjoint, and $T_{F}^{-1}=\mathscr{D}$
by (\ref{eq:u-1-1}). Since $T_{F}$ is trace class, then $\mathscr{D}$
has discrete spectrum. 

Now we proceed to diagonalize $\mathscr{D}$. Suppose $h\in dom\left(\mathscr{D}\right)$
s.t., $\mathscr{D}h=\lambda h$, i.e., consider the eigenvalue problem
\begin{equation}
\begin{Bmatrix}h''=-k^{2}h\\
h\left(0\right)-h'\left(0\right)=0\\
h\left(1\right)+h'\left(1\right)=0
\end{Bmatrix}\label{eq:eigenDexp}
\end{equation}
where $k^{2}:=2\lambda-1$. By (\ref{lem:mer-1}), $trace\left(T_{F}\right)=1$.
Since $\mathscr{D}=T_{F}^{-1}$, it follows that $\lambda>1$; thus
$k^{2}=2\lambda-1>1$. 

Let $h\left(x\right)=Ae^{ikx}+Be^{-ikx}$, applying the boundary conditions
in (\ref{eq:eigenDexp}) we get 
\begin{align*}
A\left(1-ik\right)+B\left(1+ik\right) & =0\\
A\left(1+ik\right)e^{ik}+B\left(1-ik\right)e^{-ik} & =0
\end{align*}
Setting the determinant of the coefficient matrix to zero yields (\ref{eq:Dexp}). 
\end{proof}
\begin{figure}[H]
\includegraphics[scale=0.6]{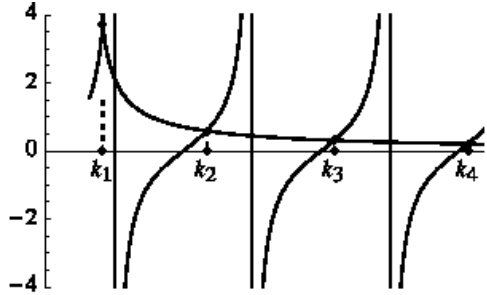}

\protect\caption{\label{fig:TFexp}Solutions to eq. (\ref{eq:Dexp}).}

\end{figure}

We now turn to a number of \emph{a priori} estimates valid for the
example $F\left(x\right)=e^{-\left|x\right|}$ in $\left|x\right|<1$. 
\begin{lem}
Let $F\left(x\right)=e^{-\left|x\right|}$, $x\in\left(-1,1\right)$.
A continuous function $h$ on $\left[0,1\right]$ is in $\mathscr{H}_{F}$
iff $h'\in L^{2}\left(0,1\right)$. \end{lem}
\begin{proof}
Let $\varphi\in C_{c}^{\infty}\left(0,1\right)$, then
\begin{align*}
\left|\int_{0}^{1}h\left(x\right)\varphi\left(x\right)dx\right|^{2} & =\left|\int_{0}^{1}h\left(x\right)T_{F}^{-1}T_{F}\varphi\left(x\right)dx\right|^{2}\\
 & =\Big|\int_{0}^{1}T_{F}^{-1/2}h\left(x\right)\underset{T_{F}^{1/2}}{\underbrace{T_{F}^{-1/2}T_{F}}}\varphi\left(x\right)dx\Big|^{2}\\
 & \leq\left\Vert T_{F}^{-1/2}h\right\Vert _{2}^{2}\left\Vert T_{F}^{1/2}\varphi\right\Vert _{2}^{2}\\
 & \leq\underset{=:C}{\underbrace{b\left(\left\Vert h\right\Vert _{2}^{2}+\left\Vert h'\right\Vert _{2}^{2}\right)}}\left\Vert T_{F}^{1/2}\varphi\right\Vert _{2}^{2}\\
 & =C\left\langle \varphi,T_{F}\varphi\right\rangle _{L^{2}\left(0,1\right)}=C\left\Vert F_{\varphi}\right\Vert _{\mathscr{H}_{F}}^{2}.
\end{align*}
\end{proof}
\begin{lem}
\label{lem:domD}Let $D_{F}\left(F_{\varphi}\right)=F_{\varphi'}$
as a skew-Hermitian operator in $\mathscr{H}_{F}$, with $dom\left(D_{F}\right)=\left\{ F_{\varphi}\:\big|\:\varphi\in C_{c}^{\infty}\left(0,1\right)\right\} $;
then 
\begin{equation}
\begin{split}dom\left(D_{F}^{*}\right) & =\left\{ h\in\mathscr{H}_{F}\:\big|\: h'\in\mathscr{H}_{F}\right\} \;\mbox{and}\\
D_{F}^{*}h & =-h'
\end{split}
\label{eq:u-3}
\end{equation}
where $h'$ refers to the weak derivative.\end{lem}
\begin{proof}
From the definition of $D_{F}^{*}$, the following are equivalent: 
\begin{enumerate}
\item $h\in dom\left(D_{F}^{*}\right)$. 
\item $h\in\mathscr{H}_{F}$, and $\exists\, C<\infty$, s.t.
\begin{align*}
\left|\left\langle h,D_{F}\left(F_{\varphi}\right)\right\rangle _{\mathscr{H}_{F}}\right|^{2} & =\left|\left\langle h,F_{\varphi'}\right\rangle _{\mathscr{H}_{F}}\right|^{2}\leq C\left\Vert F_{\varphi}\right\Vert _{\mathscr{H}_{F}}^{2}\\
 & =C\int_{0}^{1}\int_{0}^{1}\overline{\varphi\left(y\right)}\varphi\left(x\right)F\left(x-y\right)dxdy
\end{align*}
 for all $\varphi\in C_{c}^{\infty}\left(0,1\right)$.
\item $\exists\, b<\infty$, s.t.
\begin{equation}
\left|\int_{0}^{1}\overline{h\left(x\right)}\varphi'\left(x\right)dx\right|^{2}\leq C\left\langle \varphi,T_{F}\varphi\right\rangle _{2}\leq C\, b\int_{0}^{1}\left|\varphi\left(x\right)\right|^{2}dx\label{eq:u-4}
\end{equation}
where $T_{F}=$ Mercer operator, $\left(T_{F}\varphi\right)\left(x\right)=\int_{0}^{1}\varphi\left(y\right)F\left(x-y\right)dy$.
But $T_{F}:L^{2}\left(0,1\right)\rightarrow L^{2}\left(0,1\right)$
is bounded, selfadjoint, positive, and trace class, so $\exists\, b<\infty$
s.t. 
\begin{equation}
\left|\left\langle \varphi,T_{F}\varphi\right\rangle _{2}\right|^{2}\leq b\int_{0}^{1}\left|\varphi\left(x\right)\right|^{2}dx\label{eq:u-5}
\end{equation}
and (\ref{eq:u-4}) follows. Hence $h'\in L^{2}\left(0,1\right).$
\end{enumerate}

\begin{claim*}
$h'\in\mathscr{H}_{F}$ where $\left('\right)=\frac{d}{dx}=$ the
distributional (weak) derivative.
\begin{proof}
To prove this note that 
\begin{equation}
\left|\int_{0}^{1}\overline{h'\left(x\right)}\varphi\left(x\right)dx\right|^{2}\leq C\left\Vert F_{\varphi}\right\Vert _{\mathscr{H}_{F}}^{2}\label{eq:u-6}
\end{equation}
and since $h'\in L^{2}\left(0,1\right)$, it follows that $h$ is
continuous. So (\ref{eq:u-6}) $\Longrightarrow$$h'\in\mathscr{H}_{F}$.
\end{proof}
\end{claim*}

By definition of $D_{F}^{*}$ (relative to the $\mathscr{H}_{F}$
inner product) we have
\begin{equation}
\left\langle D_{F}^{*}h,F_{\varphi}\right\rangle _{\mathscr{H}_{F}}=\left\langle h,F_{\varphi'}\right\rangle _{\mathscr{H}_{F}},\;\forall\varphi\in C_{c}^{\infty}\left(0,1\right).\label{eq:u-7}
\end{equation}
Computation of the two sides in (\ref{eq:u-7}) yields:
\[
\left(\mbox{LHS}_{\left(\ref{eq:u-7}\right)}=\right)\int_{0}^{1}\left(D_{F}^{*}h\right)\left(x\right)\varphi\left(x\right)dx=\int_{0}^{1}h\left(x\right)\varphi'\left(x\right)dx\left(=\mbox{RHS}_{\left(\ref{eq:u-7}\right)}\right)
\]
$\forall\varphi\in C_{c}^{\infty}\left(0,1\right)$; which shows that
$D_{F}^{*}h=-h'$ where $h'$ is the weak derivative (in the sense
of Schwartz distributions.) 

\end{proof}
\begin{cor}
If $h\in dom\left(D_{F}^{*}\right)$, then $h$ and $h'\in\mathscr{H}_{F}$. \end{cor}
\begin{rem}
The result in \propref{dev1} is general, and applies to any continuous
p.d. function $F:\left(-a,a\right)\rightarrow\mathbb{C}$, s.t., $D^{\left(F\right)}$
has indices $\left(1,1\right)$. 

Note 
\begin{align}
dom\left(D_{F}\right) & \subset\left\{ h\in\mathscr{H}_{F}\:\big|\: h\left(0\right)=h'\left(0\right),\: h\left(1\right)=-h'\left(1\right)\right\} \label{eq:sat-1-1}\\
 & \Updownarrow\nonumber \\
D_{F}^{*}\left(F_{0}\right)=F_{0},\; & D_{F}^{*}\left(F_{1}\right)=-F_{1}
\end{align}
so, 
\begin{equation}
\begin{pmatrix}1 & -1\\
1 & 1
\end{pmatrix}\underset{\mbox{Wronskian}}{\underbrace{\begin{pmatrix}h\left(0\right) & h\left(1\right)\\
h'\left(0\right) & h'\left(1\right)
\end{pmatrix}}}=\begin{pmatrix}0 & h\left(1\right)-h'\left(1\right)\\
h\left(0\right)-h'\left(0\right) & 0
\end{pmatrix}\label{eq:sat-1-3}
\end{equation}
\end{rem}
\begin{thm}
Let $F:\left(-a,a\right)\rightarrow\mathbb{C}$ be any continuous
and p.d. function, and set 
\[
f\left(x\right)=F_{\varphi}\left(x\right)=\int_{0}^{a}\varphi\left(y\right)F\left(x-y\right)dy,\;\forall\varphi\in C_{c}^{\infty}\left(0,a\right);
\]
then
\begin{eqnarray*}
f\left(0\right)=f'\left(0\right),\;\mbox{for all }\varphi & \Longleftrightarrow & F_{0}=F'_{0}\\
f\left(a\right)=-f'\left(a\right),\;\mbox{for all }\varphi & \Longleftrightarrow & F_{a}=-F'_{a}
\end{eqnarray*}
where the derivatives are in the sense of distribution. \end{thm}
\begin{proof}
By definition, 
\begin{align*}
f\left(x\right) & =\int_{0}^{x}\varphi\left(y\right)F\left(x-y\right)dy+\int_{x}^{a}\varphi\left(y\right)F\left(x-y\right)dy\\
f'\left(x\right) & =\int_{0}^{x}\varphi\left(y\right)F'\left(x-y\right)dy+\int_{x}^{a}\varphi\left(y\right)F'\left(x-y\right)dy\\
 & \quad+\underset{=0}{\underbrace{\varphi\left(x\right)F\left(0\right)-\varphi\left(x\right)F\left(0\right)}}.
\end{align*}
So 
\begin{align*}
f\left(0\right) & =\int_{0}^{a}\varphi\left(y\right)F\left(0-y\right)dy\\
f'\left(0\right) & =\int_{0}^{a}\varphi\left(y\right)F'\left(0-y\right)dy=\int_{0}^{a}\varphi'\left(y\right)F\left(0-y\right)dy\\
f\left(a\right) & =\int_{0}^{a}\varphi\left(y\right)F\left(a-y\right)dy\\
f'\left(a\right) & =\int_{0}^{a}\varphi\left(y\right)F'\left(a-y\right)dy=\int_{0}^{a}\varphi'\left(y\right)F\left(a-y\right)dy
\end{align*}
and the theorem follows.
\end{proof}
Recall that $D_{F}^{*}\left(h\right)=-h',\;\forall h\in dom\left(D_{F}^{*}\right)$,
where $D_{F}$ is a skew Hermitian operator in $\mathscr{H}_{F}$.
Let $F_{0}$ and $F_{a}$ be the kernel functions at the endpoints:
\[
\left\langle F_{0},h\right\rangle _{\mathscr{H}_{F}}=h\left(0\right),\;\left\langle F_{a},h\right\rangle _{\mathscr{H}_{F}}=h\left(a\right).
\]

\begin{cor}
Let $F:\left(-a,a\right)\rightarrow\mathbb{C}$ be a continuous, p.d.
function, and let $D_{F}:F_{\varphi}\mapsto F_{\varphi'}$ be as usual;
then $f=F_{\varphi}$ satisfies the following boundary conditions
(kernels at endpoints):
\begin{eqnarray*}
f\left(0\right)=f'\left(0\right),\;\mbox{for all }\varphi & \Longleftrightarrow & D_{F}^{*}\left(F_{0}\right)=F_{0}\\
f\left(a\right)=-f'\left(a\right),\;\mbox{for all }\varphi & \Longleftrightarrow & D_{F}^{*}\left(F_{a}\right)=-F_{a}
\end{eqnarray*}
where $F_{0}$ and $F_{a}$ are the kernel of the RKHS $\mathscr{H}_{F}$. \end{cor}
\begin{lem}
\label{lem:defF}The following are equivalent:
\begin{enumerate}
\item \label{enu:F1}$F_{\varphi}'\left(0\right)=F_{\varphi}\left(0\right)$.
\item \label{enu:F2}$F_{0}\in dom\left(D_{F}^{*}\right)$ and $D_{F}^{*}\left(F_{0}\right)=F_{0}$.
\item \label{enu:F3}$F_{0}\left(x\right)=\mbox{const}\cdot e^{-x}$.
\end{enumerate}
\end{lem}
\begin{proof}
Note that 
\[
F_{\varphi}'\left(x\right)=\int_{0}^{a}\varphi'\left(x-y\right)F\left(y\right)dy,\;\forall\varphi\in C_{c}^{\infty}\left(0,a\right).
\]
Hence (\ref{enu:F1}) $\Longleftrightarrow$
\[
\left\langle F_{0},F_{\varphi'}\right\rangle _{\mathscr{H}_{F}}=\left\langle F_{0},F_{\varphi}\right\rangle _{\mathscr{H}_{F}}\Longleftrightarrow F_{0}\in dom\left(D_{F}^{*}\right)
\]
note $F_{\varphi'}=D_{F}\left(F_{\varphi}\right)$; and
\begin{align*}
\left\langle D_{F}^{*}\left(F_{0}\right),F_{\varphi}\right\rangle _{\mathscr{H}_{F}} & =\left\langle F_{0},F_{\varphi}\right\rangle _{\mathscr{H}_{F}},\;\forall\varphi\in C_{c}^{\infty}\left(0,a\right)\\
 & \Updownarrow\\
F_{0}\in dom\left(D_{F}^{*}\right),\; & \mbox{and}\; D_{F}^{*}\left(F_{0}\right)=F_{0},\;\mbox{which is \ensuremath{\left(\ref{enu:F2}\right)}.}
\end{align*}
But we always have (\ref{enu:F2})$\Leftrightarrow$(\ref{enu:F3}).\end{proof}
\begin{rem}
If $F\left(x\right)=e^{-\left|x\right|}$, $\left|x\right|<1$, we
have $F_{0}=e^{-x}$ and $D_{F}^{*}F_{0}=F_{0}$. But for $F\left(x\right)=1-\left|x\right|$,
$\left|x\right|<\frac{1}{2}$; $F_{0}=1-x$, $0\leq x\leq\frac{1}{2}$,
and $D_{F}^{*}\left(F_{0}\right)=1=$ constant function 1. See figures
\ref{fig:exp}-\ref{fig:abs}.
\end{rem}
\begin{figure}[H]
\includegraphics[scale=0.35]{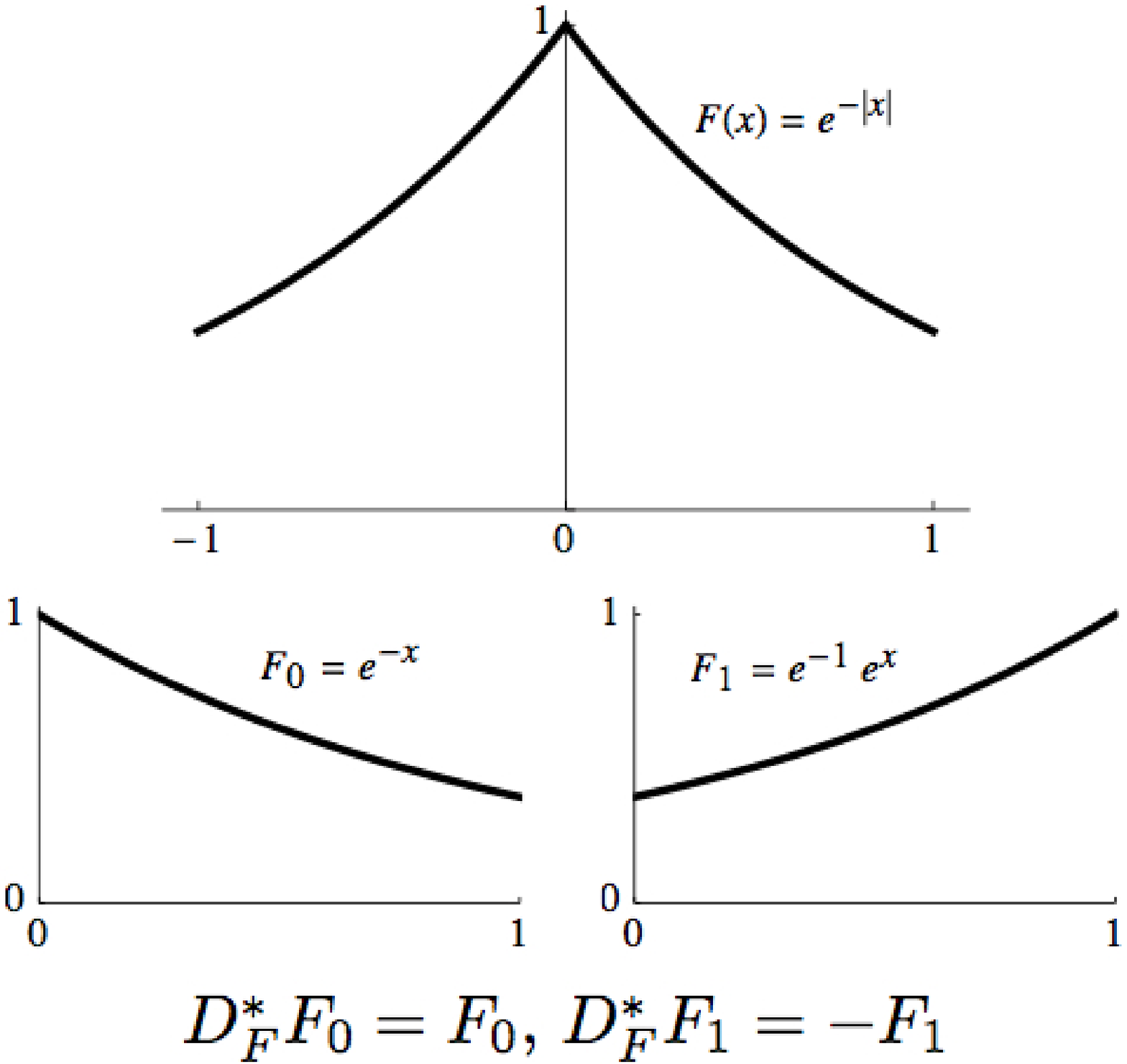}

\protect\caption{\label{fig:exp}$F\left(x\right)=e^{-\left|x\right|}$, $x\in\left(-1,1\right)$}

\end{figure}
\begin{figure}
\includegraphics[scale=0.35]{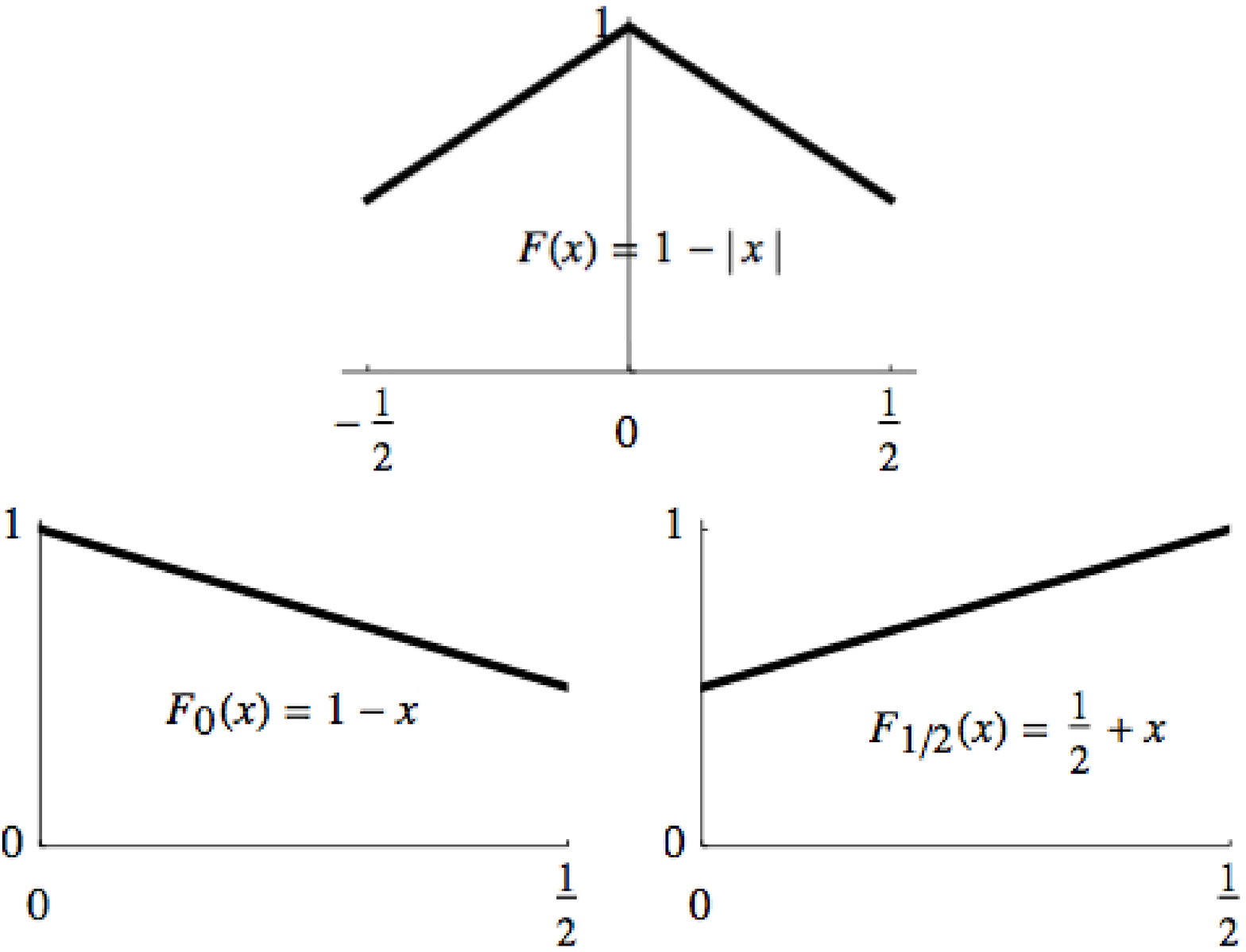}

\protect\caption{\label{fig:abs}$F\left(x\right)=1-\left|x\right|$, $x\in\left(-\frac{1}{2},\frac{1}{2}\right)$}
\end{figure}

For the conclusions in the examples below, compare with \lemref{defF}.
\begin{example}
\label{ex:absm}Let $F\left(x\right)=1-\left|x\right|$, $\left|x\right|<\frac{1}{2}$.
Set 
\begin{align*}
f\left(x\right) & =F_{\varphi}\left(x\right)=\int_{0}^{\frac{1}{2}}\varphi\left(y\right)F\left(x-y\right)dy\\
 & =\int_{0}^{x}\varphi\left(y\right)\left(1-x+y\right)dy+\int_{x}^{\frac{1}{2}}\varphi\left(y\right)\left(1-y+x\right)dy.
\end{align*}
Then, 
\begin{align*}
f'\left(x\right) & =-\int_{0}^{x}\varphi\left(y\right)dy+\int_{x}^{\frac{1}{2}}\varphi\left(y\right)dy\\
f''\left(x\right) & =-2\varphi\left(x\right);
\end{align*}
so 
\[
\varphi=T_{F}^{-1}f=-\tfrac{1}{2}\left(\tfrac{d}{dx}\right)^{2}f.
\]
For the boundary conditions, we have
\begin{align*}
f\left(0\right) & =\int_{0}^{\frac{1}{2}}\varphi\left(y\right)\left(1-y\right)dy,\\
f'\left(0\right) & =\int_{0}^{\frac{1}{2}}\varphi\left(y\right)dy,\\
f\left(\tfrac{1}{2}\right) & =\int_{0}^{\frac{1}{2}}\varphi\left(y\right)\left(\tfrac{1}{2}+y\right)dy,\\
f'\left(\tfrac{1}{2}\right) & =-\int_{0}^{\frac{1}{2}}\varphi\left(y\right)dy\neq\pm f\left(\tfrac{1}{2}\right),
\end{align*}
where $f\left(0\right)-f'\left(0\right)=-\int_{0}^{\frac{1}{2}}y\varphi\left(y\right)dy\neq0$.
\end{example}

\subsection{The Spectra of the s.a. extensions $A_{\theta}\supset-iD_{F}$}

Let $F\left(x\right)=e^{-\left|x\right|}$, $\left|x\right|<1$. Define
$D_{F}$$\left(F_{\varphi}\right)=F_{\varphi'}$ as before, where
\begin{align*}
F_{\varphi}\left(x\right) & =\int_{0}^{1}\varphi\left(y\right)F\left(x-y\right)dy\\
 & =\int_{0}^{1}\varphi\left(y\right)e^{-\left|x-y\right|}dy,\;\forall\varphi\in C_{c}^{\infty}\left(0,1\right).
\end{align*}
And let $\mathscr{H}_{F}$ be the RKHS of $F$. 
\begin{lem}
\label{lem:spInner}For all $\varphi\in C_{c}^{\infty}\left(0,1\right)$,
and all $h,h''\in\mathscr{H}_{F}$, we have 
\begin{equation}
\left\langle F_{\varphi},h\right\rangle _{\mathscr{H}_{F}}=\left\langle F_{\varphi},\tfrac{1}{2}\left(h-h''\right)\right\rangle _{2}-\tfrac{1}{2}\left[W\right]_{0}^{1}\label{eq:sp1}
\end{equation}
where 
\begin{equation}
W=\det\begin{bmatrix}h & F_{\varphi}\\
h' & F_{\varphi'}
\end{bmatrix}.\label{eq:sp1-2}
\end{equation}
Setting $l:=F_{\varphi}$, we have 
\begin{equation}
\left[W\right]_{0}^{1}=-\overline{l}\left(1\right)\left(h\left(1\right)+h'\left(1\right)\right)-\overline{l}\left(0\right)\left(h\left(0\right)-h'\left(0\right)\right).\label{eq:sp1-3}
\end{equation}
\end{lem}
\begin{proof}
Note 
\begin{align*}
\left\langle F_{\varphi},h\right\rangle _{\mathscr{H}_{F}} & =\int_{0}^{1}\varphi\left(x\right)h\left(x\right)dx\quad(\mbox{reproducing property})\\
 & =\left\langle \tfrac{1}{2}\left(I-\left(\tfrac{d}{dx}\right)^{2}\right)F_{\varphi},h\right\rangle _{2}\\
 & =\left\langle F_{\varphi},\tfrac{1}{2}\left(h-h''\right)\right\rangle _{2}-\tfrac{1}{2}\left[W\right]_{0}^{1}.
\end{align*}
Set $l:=F_{\varphi}\in\mathscr{H}_{F}$, $\varphi\in C_{c}^{\infty}\left(0,1\right)$.
Recall the boundary condition in \corref{dev1}:
\[
l\left(0\right)-l'\left(0\right)=l\left(1\right)+l'\left(1\right)=0.
\]
Then 
\begin{align*}
\left[W\right]_{0}^{1} & =\left(\overline{l'}h-\overline{l}h'\right)\left(1\right)-\left(\overline{l'}h-\overline{l}h'\right)\left(0\right)\\
 & =-\overline{l}\left(1\right)h\left(1\right)-\overline{l}\left(1\right)h'\left(1\right)-\overline{l}\left(0\right)h\left(0\right)+\overline{l}\left(0\right)h'\left(0\right)\\
 & =-\overline{l}\left(1\right)\left(h\left(1\right)+h'\left(1\right)\right)-\overline{l}\left(0\right)\left(h\left(0\right)-h'\left(0\right)\right)
\end{align*}
which is (\ref{eq:sp1-3}).\end{proof}
\begin{cor}
\label{cor:elambda}$e_{\lambda}\in\mathscr{H}_{F}$, $\forall\lambda\in\mathbb{R}$.\end{cor}
\begin{proof}
By \thmref{HF}, we need the following estimate: $\exists C<\infty$
s.t. 
\begin{equation}
\left|\int_{0}^{1}\varphi\left(x\right)e_{\lambda}\left(x\right)dx\right|^{2}\leq C\left\Vert F_{\varphi}\right\Vert _{\mathscr{H}_{F}}^{2}.\label{eq:sp1-5}
\end{equation}
But
\begin{align*}
\int_{0}^{1}\varphi\left(x\right)e_{\lambda}\left(x\right)dx & =\left\langle \tfrac{1}{2}\left(I-\left(\tfrac{d}{dx}\right)^{2}\right)F_{\varphi},e_{\lambda}\right\rangle _{2}\\
 & =\left\langle F_{\varphi},\tfrac{1}{2}\left(e_{\lambda}-e_{\lambda}''\right)\right\rangle _{2}-\frac{1}{2}\left[W\right]_{0}^{1}\\
 & =\tfrac{1}{2}\left(1+\lambda^{2}\right)\left\langle F_{\varphi},e_{\lambda}\right\rangle _{2}-\tfrac{1}{2}\left(-l\left(1\right)\left(1+i\lambda\right)e^{i\lambda}-l\left(0\right)\left(1-i\lambda\right)\right);
\end{align*}
see (\ref{eq:sp1})-(\ref{eq:sp1-3}). Here, $l:=F_{\varphi}$. 

It suffices to show

(i) $\exists C_{1}<\infty$ s.t. 
\[
\left|l\left(0\right)\right|^{2}\mbox{ and }\left|l\left(1\right)\right|^{2}\leq C_{1}\left\Vert F_{\varphi}\right\Vert _{\mathscr{H}_{F}}^{2}.
\]

(ii) $\exists C_{2}<\infty$ s.t. 
\[
\left|\left\langle F_{\varphi},e_{\lambda}\right\rangle _{2}\right|^{2}\leq C_{2}\left\Vert F_{\varphi}\right\Vert _{\mathscr{H}_{F}}^{2}.
\]

For (i), note that 
\begin{align*}
\left|l\left(0\right)\right| & =\left|\left\langle F_{0},l\right\rangle _{\mathscr{H}_{F}}\right|\leq\left\Vert F_{0}\right\Vert _{\mathscr{H}_{F}}\left\Vert l\right\Vert _{\mathscr{H}_{F}}=\left\Vert F_{0}\right\Vert _{\mathscr{H}_{F}}\left\Vert F_{\varphi}\right\Vert _{\mathscr{H}_{F}}\\
\left|l\left(1\right)\right| & =\left|\left\langle F_{1},l\right\rangle _{\mathscr{H}_{F}}\right|\leq\left\Vert F_{1}\right\Vert _{\mathscr{H}_{F}}\left\Vert l\right\Vert _{\mathscr{H}_{F}}=\left\Vert F_{1}\right\Vert _{\mathscr{H}_{F}}\left\Vert F_{\varphi}\right\Vert _{\mathscr{H}_{F}}
\end{align*}
and we have 
\begin{align*}
\left\Vert F_{0}\right\Vert _{\mathscr{H}_{F}} & =\left\Vert F_{1}\right\Vert _{\mathscr{H}_{F}}=1\\
\left\Vert l\right\Vert _{\mathscr{H}_{F}}^{2} & =\left\Vert F_{\varphi}\right\Vert _{\mathscr{H}_{F}}^{2}=\left\Vert T_{F}\varphi\right\Vert _{2}^{2}\leq\lambda_{1}^{2}\left\Vert \varphi\right\Vert _{2}^{2}<\infty
\end{align*}
where $\lambda_{1}$ is the top eigenvalue of the Mercer operator
$T_{F}$ (\lemref{mer-1}). 

For (ii), 
\begin{eqnarray*}
\left|\left\langle F_{\varphi},e_{\lambda}\right\rangle _{2}\right|^{2} & = & \left|\left\langle T_{F}\varphi,e_{\lambda}\right\rangle _{2}\right|^{2}\\
 & = & \left|\left\langle T_{F}^{1/2}\varphi,T_{F}^{1/2}e_{\lambda}\right\rangle _{2}\right|^{2}\\
 & \leq & \left\Vert T_{F}^{1/2}\varphi\right\Vert _{2}^{2}\left\Vert T_{F}^{1/2}e_{\lambda}\right\Vert _{2}^{2}\;\left(\mbox{by Cauchy-Schwarz}\right)\\
 & = & \left\langle \varphi,T_{F}\varphi\right\rangle _{2}\left\Vert T_{F}^{1/2}e_{\lambda}\right\Vert _{2}^{2}\\
 & \leq & \left\Vert F_{\varphi}\right\Vert _{\mathscr{H}_{F}}^{2}\left\Vert e_{\lambda}\right\Vert _{2}^{2}=\left\Vert F_{\varphi}\right\Vert _{\mathscr{H}_{F}}^{2};
\end{eqnarray*}
where we used the fact that $\left\Vert T_{F}^{1/2}e_{\lambda}\right\Vert _{2}^{2}\leq\lambda_{1}\left\Vert e_{\lambda}\right\Vert _{2}^{2}\leq1$,
since $\lambda_{1}<1$ = the right endpoint of the interval $\left[0,1\right]$
(see \lemref{mer-1}), and $\left\Vert e_{\lambda}\right\Vert _{2}=1$.

Therefore, the corollary follows.\end{proof}
\begin{cor}
\label{cor:HFinner2}For all $\lambda\in\mathbb{R}$, and all $F_{\varphi}$,
$\varphi\in C_{c}^{\infty}\left(0,1\right)$, we have 
\begin{align}
\left\langle F_{\varphi},e_{\lambda}\right\rangle _{\mathscr{H}_{F}} & =\tfrac{1}{2}\left(1+\lambda^{2}\right)\left\langle F_{\varphi},e_{\lambda}\right\rangle _{2}\label{eq:sp1-4}\\
 & \quad+\tfrac{1}{2}\left(\overline{l}\left(1\right)\left(1+i\lambda\right)e^{i\lambda}+\overline{l}\left(0\right)\left(1-i\lambda\right)\right).\nonumber 
\end{align}
\end{cor}
\begin{proof}
By \lemref{spInner}, 
\[
\left\langle F_{\varphi},e_{\lambda}\right\rangle _{\mathscr{H}_{F}}=\left\langle F_{\varphi},\tfrac{1}{2}\left(e_{\lambda}-e_{\lambda}''\right)\right\rangle _{2}-\tfrac{1}{2}\left[W\right]_{0}^{1}.
\]
where 
\[
\tfrac{1}{2}\left(e_{\lambda}-e_{\lambda}''\right)=\tfrac{1}{2}\left(1+\lambda^{2}\right)e_{\lambda};\;\mbox{and}
\]
\[
\left[W\right]_{0}^{1}\overset{\left(\ref{eq:sp1-3}\right)}{=}-\overline{l}\left(1\right)\left(1+i\lambda\right)e^{i\lambda}-\overline{l}\left(0\right)\left(1-i\lambda\right),\; l:=F_{\varphi}.
\]
\end{proof}
\begin{lem}
For all $F_{\varphi}$, $\varphi\in C_{c}^{\infty}\left(0,1\right)$,
and all $\lambda\in\mathbb{R}$, 
\begin{equation}
\left\langle F_{\varphi},e_{\lambda}\right\rangle _{\mathscr{H}_{F}}=\left\langle \varphi,e_{\lambda}\right\rangle _{2}.\label{eq:sp2}
\end{equation}
In particular, set $\lambda=0$, we get 
\begin{align*}
\left\langle F_{\varphi},\mathbf{1}\right\rangle _{\mathscr{H}_{F}} & =\int_{0}^{1}\varphi\left(x\right)dx=\frac{1}{2}\int_{0}^{1}\left(F_{\varphi}-F_{\varphi}''\right)\left(x\right)dx\\
 & =\frac{1}{2}\left(\left\langle F_{\varphi},\mathbf{1}\right\rangle _{2}-\left\langle F_{\varphi}'',\mathbf{1}\right\rangle _{2}\right)\\
 & \leq C\left\Vert F_{\varphi}\right\Vert _{\mathscr{H}}
\end{align*}
\end{lem}
\begin{proof}
Eq. (\ref{eq:sp2}) follows from basic fact of the Mercer operator.
See \secref{mercer} for details. It suffices to note the following
estimate:
\begin{align*}
\int_{0}^{1}F''_{\varphi}\left(x\right)dx & =F'{}_{\varphi}\left(1\right)-F'_{\varphi}\left(0\right)\\
 & =-e^{-1}\int_{0}^{1}e^{y}\varphi\left(y\right)dy-\int_{0}^{1}e^{-y}\varphi\left(y\right)dy\\
 & =-F_{\varphi}\left(1\right)-F_{\varphi}\left(0\right)\leq2\left\Vert F_{\varphi}\right\Vert _{\mathscr{H}}.
\end{align*}
\end{proof}
\begin{cor}
\label{cor:elambda1}For all $\lambda\in\mathbb{R}$, 
\begin{equation}
\left\langle e_{\lambda},e_{\lambda}\right\rangle _{\mathscr{H}_{F}}=\frac{\lambda^{2}+3}{2}.\label{eq:sp3}
\end{equation}
\end{cor}
\begin{proof}
By \corref{HFinner2}, we see that 
\begin{align}
\left\langle F_{\varphi},e_{\lambda}\right\rangle _{\mathscr{H}_{F}} & =\frac{1}{2}\left(1+\lambda^{2}\right)\left\langle F_{\varphi},e_{\lambda}\right\rangle _{2}\nonumber \\
 & +\frac{1}{2}\left(\overline{l}\left(1\right)\left(1+i\lambda\right)e^{i\lambda}+\overline{l}\left(0\right)\left(1-i\lambda\right)\right);\; l:=F_{\varphi}.\label{eq:sp4-1}
\end{align}
Since $\left\{ F_{\varphi}:\varphi\in C_{c}^{\infty}\left(0,1\right)\right\} $
is dense in $\mathscr{H}_{F}$, $\exists F_{\varphi_{n}}\rightarrow e_{\lambda}$
in $\mathscr{H}_{F}$, so that
\begin{align*}
\left\langle F_{\varphi_{n}},e_{\lambda}\right\rangle _{\mathscr{H}_{F}}\rightarrow & \left\langle e_{\lambda},e_{\lambda}\right\rangle _{\mathscr{H}_{F}}\\
= & \frac{1}{2}\left(1+\lambda^{2}\right)+\frac{1}{2}\left(e^{-i\lambda}\left(1+i\lambda\right)e^{i\lambda}+\left(1-i\lambda\right)\right)\\
= & \frac{1}{2}\left(1+\lambda^{2}\right)+1=\frac{\lambda^{2}+3}{2}.
\end{align*}

The approximation is justified since all the terms in the RHS of (\ref{eq:sp4-1})
satisfy the estimate $\left|\cdots\right|^{2}\leq C\left\Vert F_{\varphi}\right\Vert _{\mathscr{H}_{F}}^{2}$.
See the proof of \corref{elambda} for details.
\end{proof}
Note \lemref{spInner} is equivalent to the following:
\begin{cor}
\label{cor:expinner}For all $h\in\mathscr{H}_{F}$, and all $k\in dom\left(T_{F}^{-1}\right)=\left\{ F_{\varphi}:\varphi\in C_{c}^{\infty}\left(0,1\right)\right\} $,
we have 
\begin{equation}
\left\langle h,k\right\rangle _{\mathscr{H}}=\frac{1}{2}\left(\left\langle h,k\right\rangle _{0}+\left\langle h',k'\right\rangle _{0}\right)+\frac{1}{2}\left(\overline{h\left(0\right)}k\left(0\right)+\overline{h\left(1\right)}k\left(1\right)\right)\label{eq:HFinner1-1}
\end{equation}
and eq. (\ref{eq:HFinner1-1}) extends to all $k\in\mathscr{H}_{F}$,
since $dom\left(T_{F}^{-1}\right)$ is dense in $\mathscr{H}_{F}$.\end{cor}
\begin{example}
Take $h=k=e_{\lambda}$, $\lambda\in\mathbb{R}$, then (\ref{eq:HFinner1-1})
gives 
\[
\left\langle e_{\lambda},e_{\lambda}\right\rangle _{\mathscr{H}}=\frac{1}{2}\left(1+\lambda^{2}\right)+\frac{1}{2}\left(1+1\right)=\frac{\lambda^{2}+3}{2}
\]
as in (\ref{eq:sp3}).\end{example}
\begin{cor}
\label{cor:exporg}Let $A_{\theta}\supset-iD$ be any selfadjoint
extension in $\mathscr{H}_{F}$. If $\lambda,\mu\in spect\left(A_{\theta}\right)$,
s.t. $\lambda\neq\mu$, then $\left\langle e_{\lambda},e_{\mu}\right\rangle _{\mathscr{H}_{F}}=0$. \end{cor}
\begin{proof}
It follows from (\ref{eq:HFinner1-1}) that 
\begin{align}
2\left\langle e_{\lambda},e_{\mu}\right\rangle _{\mathscr{H}} & =\left\langle e_{\lambda},e_{\mu}\right\rangle _{0}+\lambda\mu\left\langle e_{\lambda},e_{\mu}\right\rangle _{0}+\left(1+e^{i\left(\mu-\lambda\right)}\right)\nonumber \\
 & =\left(1+\lambda\mu\right)\left\langle e_{\lambda},e_{\mu}\right\rangle _{0}+\left(1+e^{i\left(\mu-\lambda\right)}\right)\nonumber \\
 & =\left(1+\lambda\mu\right)\frac{e^{i\left(\mu-\lambda\right)}-1}{i\left(\mu-\lambda\right)}+\left(1+e^{i\left(\mu-\lambda\right)}\right)\label{eq:HFinner-1-2}
\end{align}
By \corref{spext}, eq. (\ref{eq:dev8-1}), we have 
\[
e^{i\lambda}=\frac{1-i\lambda}{1+i\lambda}e^{i\theta},\quad e^{i\mu}=\frac{1-i\mu}{1+i\mu}e^{i\theta}
\]
and so 
\[
e^{i\left(\mu-\lambda\right)}=\frac{\left(1-i\mu\right)\left(1+i\lambda\right)}{\left(1+i\mu\right)\left(1-i\lambda\right)}.
\]
Substitute this into (\ref{eq:HFinner-1-2}) yields 
\[
2\left\langle e_{\lambda},e_{\mu}\right\rangle _{\mathscr{H}}=\frac{-2\left(1+\lambda\mu\right)}{\left(1+i\mu\right)\left(1-i\lambda\right)}+\frac{2\left(1+\lambda\mu\right)}{\left(1+i\mu\right)\left(1-i\lambda\right)}=0.
\]
\end{proof}
\begin{cor}
\label{cor:Fext}Let $F\left(x\right)=e^{-\left|x\right|}$, $\left|x\right|<1$.
Let $D_{F}\left(F_{\varphi}\right)=F_{\varphi'}$, $\forall\varphi\in C_{c}^{\infty}\left(0,1\right)$,
and $A_{\theta}\supset-iD_{F}$ be a selfadjoint extension  in $\mathscr{H}_{F}$.
Set $e_{\lambda}\left(x\right)=e^{i\lambda x}$, and 
\begin{equation}
\Lambda_{\theta}:=spect\left(A_{\theta}\right)\left(=\mbox{discrete subset in }\mathbb{R}\mbox{ by Cor. }\ref{cor:spdiscrete}\right)\label{eq:sp5}
\end{equation}
Then 
\begin{equation}
\widetilde{F}_{\theta}\left(x\right)=\sum_{\lambda\in\Lambda_{\theta}}\frac{2}{\lambda^{2}+3}e_{\lambda}\left(x\right),\;\forall x\in\mathbb{R}\label{eq:sp6}
\end{equation}
is a continuous p.d. extension of $F$ to the real line. Note that
both sides in eq. (\ref{eq:sp6}) depend on the choice of $\theta$.
\end{cor}
The type 1 extensions are indexed by $\theta\in[0,2\pi)$ where $\Lambda_{\theta}$
is given in (\ref{eq:sp5}), see also (\ref{eq:dev9}) in \corref{spext}.
\begin{cor}[Sampling property of the set $\Lambda_{\theta}$ ]
Let $F\left(x\right)=e^{-\left|x\right|}$ in $\left|x\right|<1$,
$\mathscr{H}_{F}$, $\theta$, and $\Lambda_{\theta}$ be as above.
Let $T_{F}$ be the corresponding Mercer operator. Then for all $\varphi\in L^{2}\left(0,1\right)$,
we have 
\[
\left(T_{F}\varphi\right)\left(x\right)=2\sum_{\lambda\in\Lambda_{\theta}}\frac{\widehat{\varphi}\left(\lambda\right)}{\lambda^{2}+3}e^{i\lambda x},\;\mbox{for all }x\in\left(0,1\right).
\]
\end{cor}
\begin{proof}
This is immediate from \corref{Fext}.\end{proof}
\begin{rem}
Note that the system $\left\{ e_{\lambda}\:|\:\lambda\in\Lambda_{\theta}\right\} $
is orthogonal in $\mathscr{H}_{F}$, but \uline{not} in $L^{2}\left(0,1\right)$.\end{rem}
\begin{proof}
We see in \subref{saext} that $A_{\theta}$ has pure atomic spectrum.
By (\ref{cor:elambda1}), the set $\left\{ \sqrt{\frac{2}{\lambda^{2}+3}}e_{\lambda}:\lambda\in\Lambda_{\theta}\right\} $
is an ONB in $\mathscr{H}_{F}$. Hence, for $F=F_{0}=e^{-\left|x\right|}$,
we have the corresponding p.d. extension: 
\begin{align}
F_{\theta}\left(x\right) & =\sum_{\lambda\in\Lambda_{\theta}}\frac{1}{\left\Vert e_{\lambda}\right\Vert _{\mathscr{H}_{F}}^{2}}\left\langle e_{\lambda},F\right\rangle _{\mathscr{H}_{F}}e_{\lambda}\left(x\right)\nonumber \\
 & =\sum_{\lambda\in\Lambda_{\theta}}\frac{2}{\lambda^{2}+3}e_{\lambda}\left(x\right),\;\forall x\in\left[0,1\right].\label{eq:sp4-2}
\end{align}
where $\left\langle e_{\lambda},F\right\rangle _{\mathscr{H}_{F}}=\overline{e_{\lambda}\left(0\right)}=1$
by the reproducing property. But the RHS of (\ref{eq:sp4-2}) extends
to $\mathbb{R}$. See \figref{expExt}.
\end{proof}
\begin{figure}[H]
\includegraphics[scale=0.6]{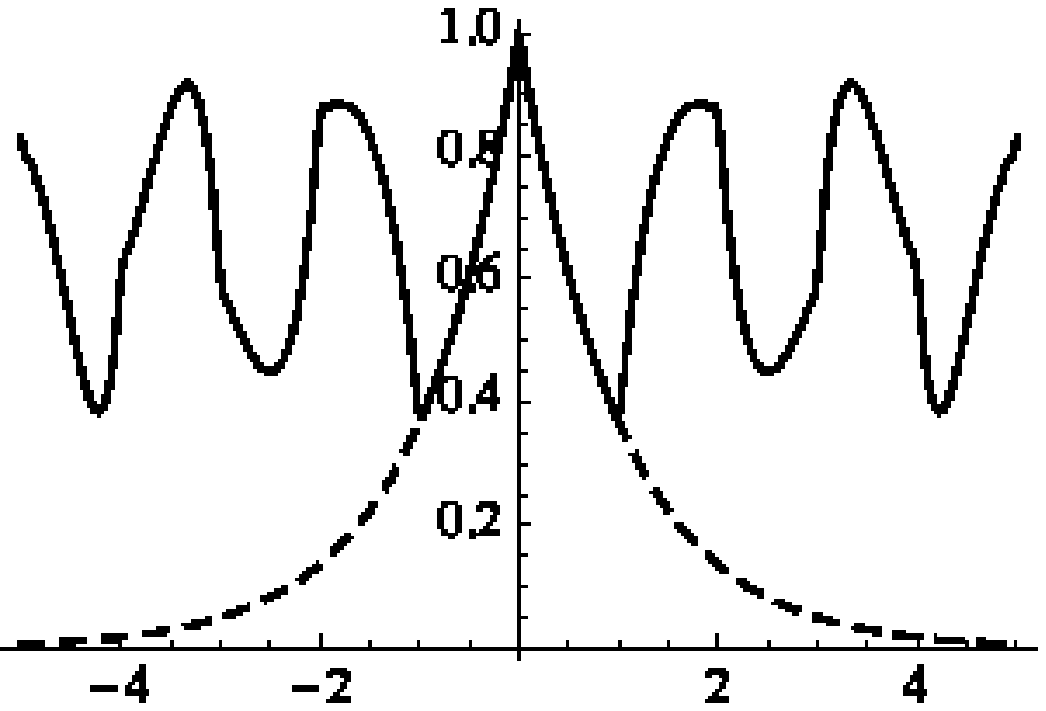}

\protect\caption{\label{fig:expExt}$\theta=0$. A type 1 continuous p.d. extension
of $F\left(x\right)=e^{-\left|x\right|}\big|_{\left[-1,1\right]}$
in $\mathscr{H}_{F}$.}

\end{figure}

\begin{cor}
\label{cor:exponb}Let $F\left(x\right)=e^{-\left|x\right|}$ in $\left(-1,1\right)$,
and let $\mathscr{H}_{F}$ be the RKHS. Let $\theta\in[0,2\pi)$,
and let $\Lambda_{\theta}$ be as above; then $\left\{ \sqrt{\frac{2}{\lambda^{2}+3}}e_{\lambda}\:|\:\lambda\in\Lambda_{\theta}\right\} $
is an ONB in $\mathscr{H}_{F}$.\end{cor}
\begin{lem}
\label{lem:lcg-Bochner}Let $G$ be a locally compact abelian group.
There is a bijective correspondence between all continuous p.d. extensions
$\tilde{F}$ to $G$ of the given p.d. function $F$ on $\Omega-\Omega$,
on the one hand; and all Borel probability measures $\mu$ on $\widehat{G}$,
on the other, i.e., all $\mu\in\mathscr{M}(\widehat{G})$ s.t.
\begin{equation}
F\left(x\right)=\widehat{\mu}\left(x\right),\:\forall x\in\Omega-\Omega,\;\mbox{where}\label{eq:lcg-bochner}
\end{equation}
\[
\widehat{\mu}\left(x\right)=\int_{\widehat{G}}\lambda\left(x\right)d\mu\left(\lambda\right)=\int_{\widehat{G}}\left\langle \lambda,x\right\rangle d\mu\left(\lambda\right),\:\forall x\in G.
\]
Moreover, 
\[
V:\mathscr{H}_{F}\ni F_{\varphi}\longmapsto\widehat{\varphi}\in L^{2}\left(\mu\right)
\]
extends by density to an isometry from $\mathscr{H}_{F}$ to $L^{2}\left(\mu\right)$,
and 
\[
V^{*}f=\left(fd\mu\right)^{\vee}\in\mathscr{H}_{F},\;\mbox{where}
\]
\[
\left(fd\mu\right)^{\vee}:=\int_{\widehat{G}}\left\langle \lambda,x\right\rangle f\left(\lambda\right)d\mu\left(\lambda\right),\;\forall f\in L^{2}\left(\mu\right).
\]
\end{lem}
\begin{proof}
This is an immediate application of Bochner's characterization of
the continuous positive definite functions on locally compact abelian
groups. See \cite{JPT14}.\end{proof}
\begin{defn}
Set 
\[
Ext\left(F\right)=\left\{ \mu\in\mathscr{M}(\widehat{G})\:\Big|\: s.t.\:(\ref{eq:lcg-bochner})\mbox{ holds}\right\} .
\]
(See \secref{types} for type 1 and 2 extensions.)\end{defn}
\begin{cor}
In the case $F=e^{-\left|x\right|}$, $\left|x\right|<1$, let $A_{\theta}\supset-iD$
be any one of the s.a. extensions. Let $\widetilde{F}_{\theta}\left(x\right)$
be the continuous p.d. extension in (\ref{eq:sp6}). 
\begin{enumerate}
\item Then
\[
\widetilde{F}_{\theta}\left(x\right)=\int_{-\infty}^{\infty}e^{i\lambda x}d\mu_{\theta}\left(\lambda\right).
\]

\item And
\[
\mu_{\theta}\left(\cdot\right):=\sum_{\lambda\in\Lambda_{\theta}}\frac{2\delta_{\lambda}}{\lambda^{2}+3},\:\mbox{where }\mbox{\ensuremath{\delta_{\lambda}}= Dirac point mass at \ensuremath{\lambda}.}
\]

\item And $\mu_{\theta}$ is an extreme point in $Ext(F)$.
\end{enumerate}
\end{cor}
\begin{proof}
Follows from the argument in Corollaries \ref{cor:Fext}, \ref{cor:exponb},
and Lemma \ref{lem:lcg-Bochner}.\end{proof}
\begin{thm}
\label{thm:ez}Let $\mathscr{H}_{F}$ be the RKHS for $F\left(x\right)=e^{-\left|x\right|}$,
$x\in\left(-1,1\right)$. Then, for all complex numbers $z$, the
function $e_{z}\left(x\right)=e^{zx}$ is in $\mathscr{H}_{F}$.\end{thm}
\begin{proof}
First, if $z$ is purely imaginary, the function $e_{z}$ is in $\mathscr{H}_{F}$
by Coroll. \ref{cor:elambda}. From the lemmas in \secref{exp}, we
know the defect vectors for the associated skew Hermitian operator
$D_{F}$. The operator $D_{F}$ has indices $\left(1,1\right)$ in
$\mathscr{H}_{F}$. By von Neumann \cite{DS88b}, we know that $e_{z}$
solves $D_{F}^{*}e_{z}=z\, e_{z}$ whenever $z$ has non-zero real
part. Since the indices are $\left(1,1\right)$, the solution for
each $z$ must be in $\mathscr{H}_{F}$. 
\end{proof}
Below, we consider the converse of \corref{Fmu} for the special case
where $F\left(x\right)=e^{-\left|x\right|}$, $\left|x\right|<1$. 
\begin{lem}
\label{lem:expmu}Let $F\left(x\right)=e^{-\left|x\right|}$, $\left|x\right|<1$,
and let $\mathscr{H}_{F}$ be the corresponding RKHS. Then, for all
$\lambda\in\mathbb{R}$, we have 
\begin{gather}
e_{\lambda}\left(x\right)=F_{\mu_{\lambda}}\left(x\right)=\int_{0}^{1}F\left(x-y\right)d\mu_{\lambda}\left(y\right),\;\forall x\in\left(0,1\right);\;\mbox{where for }\lambda\:\mbox{fixed}\label{eq:expmu1}\\
d\mu_{\lambda}\left(y\right):=\frac{1}{2}\left(1+\lambda^{2}\right)e_{\lambda}\left(y\right)dy+\frac{1}{2}\left(\left(1-i\lambda\right)\delta_{0}+\left(1+i\lambda\right)e^{-i\lambda}\delta_{1}\right).\label{eq:expmu2}
\end{gather}
\end{lem}
\begin{proof}
Note by \corref{elambda}, $e_{\lambda}\in\mathscr{H}_{F}$ for all
$\lambda\in\mathbb{R}$. 

Let $A\supset-iD$ be any of the s.a. extensions, with $spect\left(A\right)=\left\{ \lambda_{n}\:\big|\: n\in\mathbb{Z}\right\} $;
and let $\left\{ e_{\lambda_{n}}\:\big|\: n\in\mathbb{Z}\right\} $
be the corresponding eigenfunctions. Set $e_{n}\left(x\right):=e_{\lambda_{n}}\left(x\right)=e^{i\lambda_{n}x}$,
$x\in\left(0,1\right)$. Thus, $\left\{ e_{n}/\left\Vert e_{n}\right\Vert _{\mathscr{H}_{F}}\right\} $
is an ONB in $\mathscr{H}_{F}$. 

Using the reproducing property, we have 
\begin{equation}
F\left(x\right)=\sum_{n}\frac{1}{\left\Vert e_{n}\right\Vert _{\mathscr{H}}^{2}}\left\langle e_{n},F_{0}\right\rangle e_{n}\left(x\right)=\sum_{n\in\mathbb{Z}}\frac{1}{\left\Vert e_{n}\right\Vert _{\mathscr{H}}^{2}}e_{n}\left(x\right),\;\forall x\in\left(0,1\right).\label{eq:Fexpand}
\end{equation}
Note in (\ref{eq:Fexpand}), we used $\left\langle e_{n},F_{0}\right\rangle _{\mathscr{H}_{F}}=\overline{e_{n}\left(0\right)}=1.$

Now, let $d\mu_{\lambda}$ be as in the statement of the lemma. Then,
\begin{eqnarray*}
\int_{0}^{1}F\left(x-y\right)d\mu_{\lambda}\left(y\right) & \overset{\text{by }\left(\ref{eq:Fexpand}\right)}{=} & \sum_{n}\frac{1}{\left\Vert e_{n}\right\Vert _{\mathscr{H}}^{2}}\int_{0}^{1}\overline{e_{n}\left(y\right)}d\mu\left(y\right)\\
 & = & \sum_{n}\frac{1}{\left\Vert e_{n}\right\Vert _{\mathscr{H}}^{2}}e_{n}\left(x\right)\Bigg[\frac{1}{2}\left(1+\lambda^{2}\right)\int_{0}^{1}\overline{e_{n}\left(y\right)}e_{\lambda}\left(y\right)dy\\
 &  & +\frac{1}{2}\left(\left(1-i\lambda\right)\overline{e_{n}\left(0\right)}+\left(1+i\lambda\right)e^{-i\lambda}\overline{e_{n}\left(1\right)}\right)\Bigg]\\
 & = & \sum_{n}\frac{1}{\left\Vert e_{n}\right\Vert _{\mathscr{H}}^{2}}e_{n}\left(x\right)\Bigg[\frac{1}{2}\left\langle e_{n},e_{\lambda}-e_{\lambda}''\right\rangle _{2}\\
 &  & +\frac{1}{2}\left(\left(1-i\lambda\right)\overline{e_{n}\left(0\right)}+\left(1+i\lambda\right)e^{-i\lambda}\overline{e_{n}\left(1\right)}\right)\Bigg]\\
 & \overset{\text{Cor.}\ref{cor:HFinner2}}{=} & \sum_{n}\frac{1}{\left\Vert e_{n}\right\Vert _{\mathscr{H}}^{2}}e_{n}\left(x\right)\left\langle e_{n},e_{\lambda}\right\rangle _{\mathscr{H}_{F}}\\
 & = & e_{\lambda}\left(x\right)
\end{eqnarray*}
thus (\ref{eq:expmu1})-(\ref{eq:expmu2}) hold.\end{proof}
\begin{cor}
\label{cor:expmu}For $F\left(x\right)=e^{-\left|x\right|}$, $\left|x\right|<1$,
every $h\in\mathscr{H}_{F}$ takes the form $h=F_{\mu_{h}}$, where
$\mu_{h}$ is a complex Borel measure in $\left[0,1\right]$. Specifically,
\begin{gather}
h\left(x\right)=\int_{0}^{1}e^{-\left|x-y\right|}d\mu_{h}\left(y\right),\;\mbox{where }\label{eq:expmu3}\\
d\mu_{h}\left(y\right):=\sum_{n}\frac{1}{\left\Vert e_{n}\right\Vert _{\mathscr{H}}^{2}}\left\langle e_{n},h\right\rangle _{\mathscr{H}}d\mu_{n}\left(y\right).\label{eq:expmu4}
\end{gather}
\end{cor}
\begin{rem}
Note the measure $d\mu_{h}$ in the corollary does not necessarily
have finite total variation. From (\ref{eq:expmu4}), we conclude
that 
\[
\left|\mu_{h}\right|<\infty\Longleftrightarrow\sum_{n}\frac{1}{\left\Vert e_{n}\right\Vert _{\mathscr{H}}^{2}}\left|\left\langle e_{n},h\right\rangle _{\mathscr{H}}\right|\left|\mu_{n}\right|<\infty;
\]
where $\left|\mu_{n}\right|$ = the total variation of the measure
$\mu_{n}$, see (\ref{eq:expmu2}).\end{rem}
\begin{proof}[Proof of \corref{expmu} ]
By (\ref{lem:expmu}), for all $e_{n}$, $n\in\mathbb{Z}$, $\exists\mu_{n}$
s.t. 
\[
e_{n}\left(x\right)=\int_{0}^{1}F\left(x-y\right)d\mu_{n}\left(y\right);\;\mbox{where}
\]
\begin{equation}
d\mu_{n}\left(y\right)=\frac{1}{2}\left(1+\lambda_{n}^{2}\right)e_{\lambda_{n}}\left(y\right)dy+\frac{1}{2}\left(\left(1-i\lambda_{n}\right)\delta_{0}+\left(1+i\lambda_{n}\right)e^{-i\lambda_{n}}\delta_{1}\right)\label{eq:mun}
\end{equation}
for all $n\in\mathbb{Z}$. 

It follows that, for all $h\in\mathscr{H}_{F}$, we have 
\begin{align*}
h\left(x\right) & =\sum_{n}\frac{1}{\left\Vert e_{n}\right\Vert _{\mathscr{H}}^{2}}\left\langle e_{n},h\right\rangle _{\mathscr{H}}e_{n}\left(x\right)\\
 & =\sum_{n}\frac{1}{\left\Vert e_{n}\right\Vert _{\mathscr{H}}^{2}}\left\langle e_{n},h\right\rangle _{\mathscr{H}}\int_{0}^{1}F\left(x-y\right)d\mu_{n}\\
 & \overset{\left(\text{Fubini}\right)}{=}\int_{0}^{1}F\left(x-y\right)\sum_{n}\frac{1}{\left\Vert e_{n}\right\Vert _{\mathscr{H}}^{2}}\left\langle e_{n},h\right\rangle _{\mathscr{H}}d\mu_{n}\left(y\right);
\end{align*}
i.e., we set $d\mu_{h}\left(y\right)$ as in (\ref{eq:expmu4}). \end{proof}
\begin{rem}
The total variation of the measure $\mu_{n}$ in $\left[0,1\right]$
is $\frac{1}{2}\left(1+\lambda_{n}^{2}\right)+\sqrt{1+\lambda_{n}^{2}}$,
$n\in\mathbb{Z}$.
\end{rem}

\subsection{\label{sub:type2ext}Examples of Type 2 Extensions}

For $F=e^{-\left|x\right|}$, $x\in\left(-1,1\right)$, consider the
following family of extensions:
\[
G_{r}\left(x\right):=\begin{cases}
e^{-1}e^{r\left(1-x\right)} & x\geq1\\
e^{-\left|x\right|} & -1<x<1\\
e^{-1}e^{r\left(1+x\right)} & x\leq-1
\end{cases}
\]
For $r\in\left[0,1\right]$, this family of extensions are type 2.
See \figref{gr}.

The Fourier transform of $G_{r}$ is 
\begin{align*}
\widehat{G_{r}}\left(\lambda\right) & =\int_{-\infty}^{\infty}G_{r}\left(x\right)e^{-i\lambda x}dx\\
 & =\frac{2e\cdot\left(\lambda^{2}+r^{2}\right)+2(r-1)\left(\cos(\lambda)\left(\lambda^{2}-r\right)+\lambda(r+1)\sin(\lambda)\right)}{e\cdot\left(\lambda^{2}+1\right)\left(\lambda^{2}+r^{2}\right)};
\end{align*}
and so
\[
G_{r}\left(x\right)=\int_{-\infty}^{\infty}e^{i\lambda x}\frac{1}{2\pi}\widehat{G_{r}}\left(\lambda\right)d\lambda.
\]
Set 
\begin{align*}
\widehat{g_{r}}\left(\lambda\right):= & \frac{1}{2\pi}\widehat{G_{r}}\left(\lambda\right)
\end{align*}
then we get 
\[
e^{-\left|x\right|}=\int_{-\infty}^{\infty}e^{i\lambda x}\widehat{g_{r}}\left(\lambda\right)d\lambda,\quad\forall x\in\left(-1,1\right).
\]
Indeed, 
\[
d\mu_{r}\left(\lambda\right):=\widehat{g_{r}}\left(\lambda\right)d\lambda,\; r\in\left[0,1\right]
\]
is a family of probability measures that extends $e^{-\left|x\right|}\Big|_{\left(-1,1\right)}$. 

\begin{figure}[H]
\includegraphics[scale=0.45]{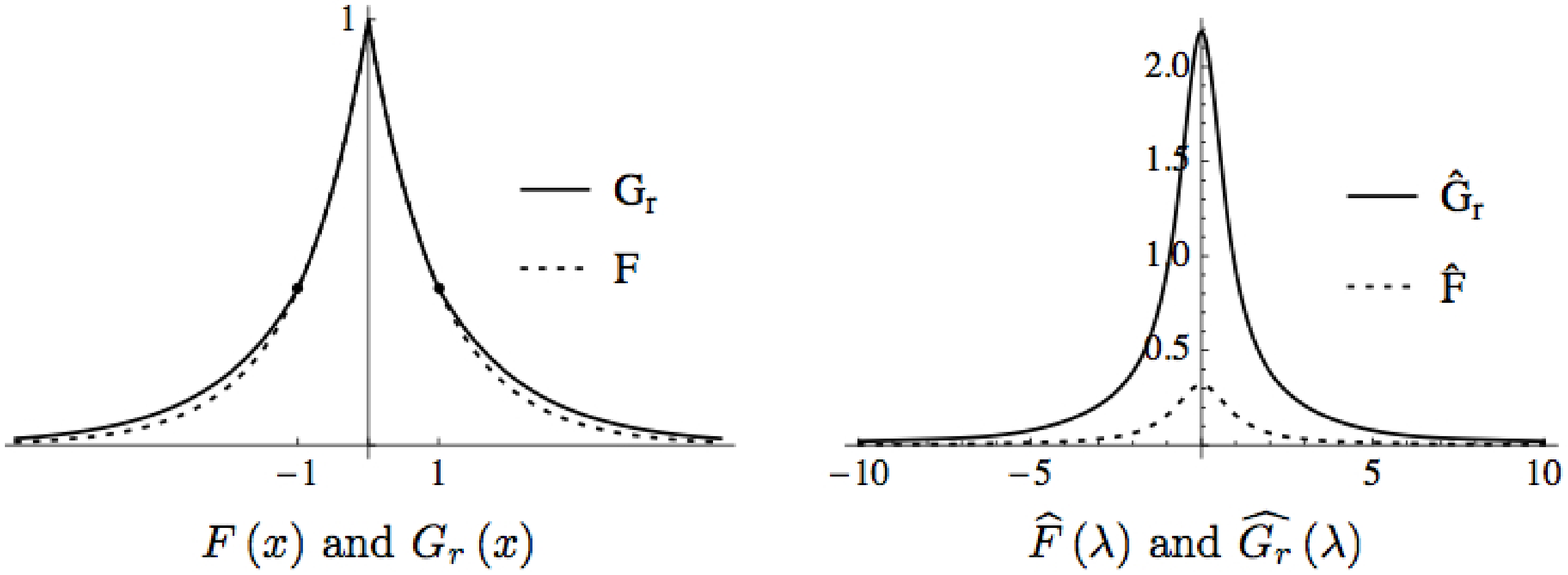}

\protect\caption{\label{fig:gr}$G_{r}$ extensions (type 2), $r=0.8$}
\end{figure}

\begin{example}
Considering $F\left(x\right)=e^{-\left|x\right|}$ for $\left|x\right|<1$;
then the trivial p.d. extension $\widetilde{F}\left(x\right):=e^{-\left|x\right|}$,
defined for \uline{all} $x\in\mathbb{R}$, is a type 2 extension.
\end{example}

\subsection{The Unitary Groups}

Starting with a pair $\left(\Omega,F\right)$, where $F$ is a prescribed
continuous positive definite function defined on $\Omega=\left(-a,a\right)$,
we study the family of selfadjoint extensions $\left\{ A_{\theta}\right\} $
of the canonical Hermitian operator $D^{\left(F\right)}$ in $\mathscr{H}_{F}$.
We show that the corresponding family of unitary one-parameter group
$U^{\left(F,\theta\right)}\left(t\right)$, $t\in\mathbb{R}$, acting
in $\mathscr{H}_{F}$ and generated by $A_{\theta}$ has an explicit
translation representation. It takes an explicit form and involving
systems of boundary conditions, a system defined on functions $f$
in $\mathscr{H}_{F}$, and involving both $f$ and the distribution
derivative $f'$ at the two endpoints $x=0$, and $x=a$.

We see from \propref{dev1} the unitary one-parameter group $\left\{ U^{\theta}\left(t\right)\right\} _{t\in\mathbb{R}}$
generated by $A_{\theta}$ must be translation acting on continuous
functions $h$ on $\mathbb{R}$ modulo
\begin{equation}
h\left(1\right)+h'\left(1\right)=e^{i\theta}\left(h\left(0\right)-h'\left(0\right)\right).\label{eq:u-2-1}
\end{equation}
But we must therefore look at extensions from $\left[0,1\right]$
to $\mathbb{R}$ of the function $h$ s.t. $h:\mathbb{R}\rightarrow\mathbb{C}$,
$h\big|_{\left[0,1\right]}\in\mathscr{H}_{F}$ and satisfies (\ref{eq:u-2-1}).
See \figref{ind0}.

\begin{figure}[H]
\includegraphics[scale=0.45]{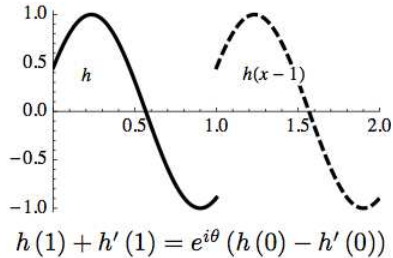}

\protect\caption{\label{fig:ind0}Extensions preserving the boundary condition (\ref{eq:u-2-1}). }

\end{figure}

The extended function, also denoted $h$, satisfies 
\begin{equation}
h\left(x+1\right)+h'\left(x+1\right)=e^{i\theta}\left(h\left(x\right)-h'\left(x\right)\right),\;\forall x\in\mathbb{R}\label{eq:u-2-2}
\end{equation}
and further $h,h'\in\mathscr{H}_{F}$; see \lemref{domD}. And the
functions in (\ref{eq:u-2-2}) are invariant under $U^{\theta}\left(t\right):h\mapsto h\left(\cdot+t\right)$,
$\forall t\in\mathbb{R}$. See \figref{indreprep}. Note this is an
induced representation, see \cite{Ors79}.

\begin{figure}[H]
\includegraphics[scale=0.45]{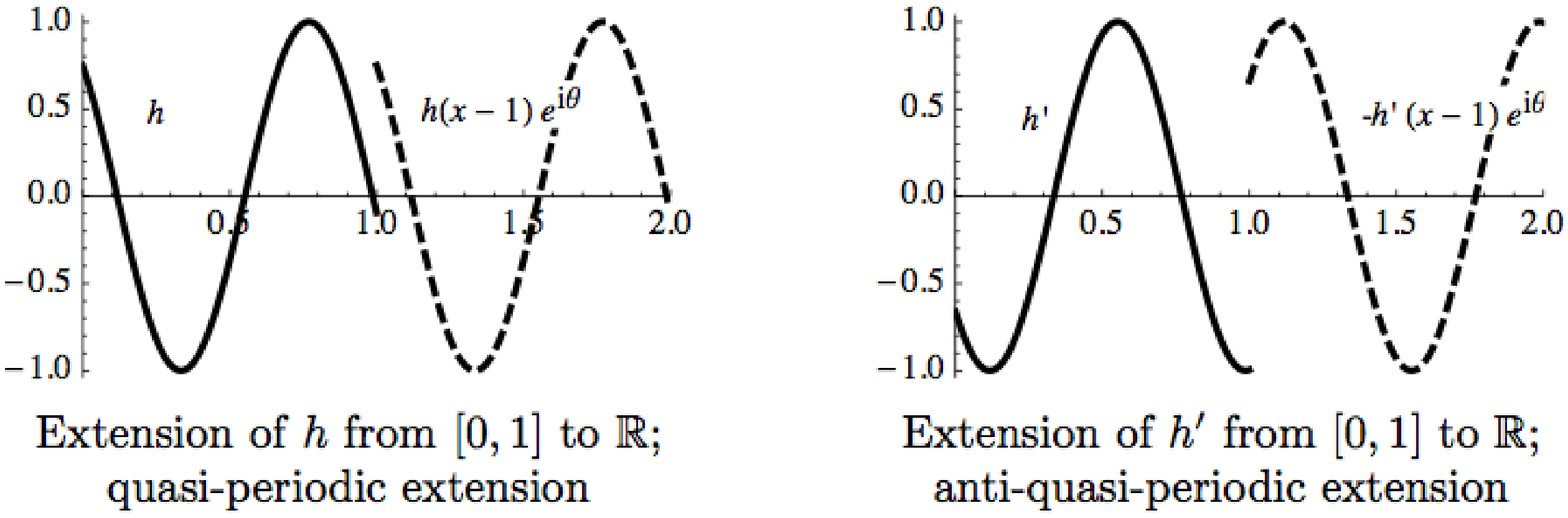}

\protect\caption{\label{fig:indreprep}$U^{\theta}\left(t\right):h\mapsto h\left(x+t\right)$,
$\forall t\in\mathbb{R}$, modulo condition (\ref{eq:u-2-1}). }
\end{figure}

\begin{example}
\label{ex:ddx}It is useful to look at the classical example, $\frac{d}{dx}$
acting on $L^{2}\left(0,1\right)$ with domain $C_{c}^{\infty}\left(0,1\right)$.
The selfadjoint extensions are parameterized by $e^{i\theta}$. Let
$U^{\theta}\left(t\right)=e^{itA_{\theta}}$, $t\in\mathbb{R}$, as
before. 

Note the following representation 
\[
\pi:C\left(\mathbb{R}/\mathbb{Z}\right)\longrightarrow\mathscr{H}^{\left(\theta\right)},\quad\pi\left(\varphi\right)f:=\varphi f
\]
satisfies the covariance relation 
\[
U^{\left(\theta\right)}\left(t\right)\pi\left(\varphi\right)U^{\left(\theta\right)}\left(-t\right)=\pi\left(\varphi\left(\cdot+t\right)\right),\;\forall t\in\mathbb{R}.
\]
Therefore, by the theorem of imprimitivity, there exists a unitary
representation $L$ of $\mathbb{R}/\mathbb{Z}$ on $\mathscr{H}_{F}$,
such that $U^{\left(\theta\right)}\simeq ind_{\mathbb{Z}}^{\mathbb{R}}\left(L\right)$.

Now, let $L\left(n\right):=e^{in\theta}$, $n\in\mathbb{Z}$, be a
one-dimensional representation of $\mathbb{Z}$. The unitary one-parameter
group $\left\{ T^{\left(\theta\right)}\left(t\right)\right\} _{t\in\mathbb{R}}$
is induced from $L$, i.e., 
\[
T^{\left(\theta\right)}=ind_{\mathbb{Z}}^{\mathbb{R}}\left(e^{i\theta}\right)
\]
acting on 
\begin{gather*}
\mathscr{H}_{ind}^{\left(\theta\right)}=\left\{ f:\mathbb{R}\rightarrow\mathbb{C},\; f\left(x+1\right)=e^{i\theta}f\left(x\right)\right\} ,\;\mbox{where}\\
\left\Vert f\right\Vert _{\mathscr{H}_{ind}^{\left(\theta\right)}}^{2}=\int_{0}^{1}\left|f\left(x\right)\right|^{2}dx=\int_{\mathbb{R}/\mathbb{Z}}\left|f\left(x\right)\right|^{2}dx,\;\forall f\in\mathscr{H}_{ind}^{\left(\theta\right)};
\end{gather*}
such that 
\[
(T^{\left(\theta\right)}\left(t\right))f\left(x\right)=f\left(x+t\right),\;\forall t\in\mathbb{R}.
\]

\end{example}
Let 
\[
W^{\left(\theta\right)}:\mathscr{H}_{ind}^{\left(\theta\right)}\longrightarrow L^{2}\left(0,1\right),\quad(W^{\left(\theta\right)}f)\left(x\right)=f\left(x\mod\mathbb{Z}\right)
\]
then
\[
U^{\left(\theta\right)}\left(t\right)=W^{\left(\theta\right)}U_{ind}^{\left(\theta\right)}\left(t\right)W^{\left(\theta\right)*}
\]
i.e., we have the following commutative diagram:
\[
\xymatrix{\mathscr{H}_{ind}^{\left(\theta\right)}\ar[r]^{W^{\left(\theta\right)}}\ar[d]_{T_{t}^{\left(\theta\right)}=U_{ind}^{\left(\theta\right)}\left(t\right)} & L^{2}\left(0,1\right)\ar[d]^{U^{\left(\theta\right)}\left(t\right)=e^{itA_{\theta}}}\\
\mathscr{H}_{ind}^{\left(\theta\right)}\ar[r]_{W^{\left(\theta\right)}} & L^{2}\left(0,1\right)
}
\]

Return to $F\left(x\right)=e^{-\left|x\right|}$, $\left|x\right|<1$.

Consider the following two subspaces of $\mathscr{H}_{F}$, 
\begin{align*}
\mathscr{D}_{0} & :=dom\left(D_{F}\right)=\left\{ F_{\varphi}:\varphi\in C_{c}^{\infty}\left(0,1\right)\right\} \\
\mathscr{D}_{\theta} & :=dom\left(A_{\theta}\right)=\left\{ f+c\left(e^{-x}+e^{i\theta}e^{x-1}\right):f\in\mathscr{D}_{0},c\in\mathbb{C}\right\} \\
 & \:=\left\{ h\in D_{F}^{*}:h\left(1\right)+h'\left(1\right)=e^{i\theta}\left(h\left(0\right)-h'\left(0\right)\right)\right\} 
\end{align*}
and note $\mathscr{D}_{0}\subset\mathscr{D}_{\theta}$.
\begin{defn}
Let $L$ be a subspace in $\mathscr{H}_{F}$ (e.g., $L=\mathscr{D}_{0}$
or $\mathscr{D}_{\theta}$), and let 
\[
Mul\left(L\right)=\left\{ f\:\big|\: fh\in L,\:\forall h\in L\right\} 
\]
where $fh$ is point-wise product, i.e., $\left(fh\right)\left(x\right)=f\left(x\right)h\left(x\right)$.
(Actually $f$ must be a multiplier of $\mathscr{H}_{F}$ as well,
i.e., we must have $f$ satisfy that $fh\in\mathscr{H}_{F}$, $\forall h\in\mathscr{H}_{F}$.)\end{defn}
\begin{thm}
We have
\begin{align}
Mul\left(\mathscr{D}_{0}\right) & =\left\{ f\: s.t.\: f'\left(0\right)=f'\left(1\right)=0\right\} ;\;\mbox{and}\label{eq:mul1}\\
Mul\left(\mathscr{D}_{\theta}\right) & =\left\{ f\in Mul\left(\mathscr{D}_{0}\right)\; s.t.\: f\left(0\right)=f\left(1\right)=0\right\} .\label{eq:mul2}
\end{align}

\end{thm}

\begin{proof}
~

1. Assume $h\in\mathscr{D}_{0}$, i.e., $h\left(0\right)-h'\left(0\right)=h\left(1\right)+h'\left(1\right)=0$.
Then, 
\[
\left\{ \begin{split}\left(\varphi h\right)\left(1\right)+\left(\varphi h\right)'\left(1\right) & =0\\
\left(\varphi h\right)\left(0\right)-\left(\varphi h\right)'\left(0\right) & =0
\end{split}
\right\} \Longleftrightarrow\left\{ \begin{split}\varphi'\left(1\right)h\left(1\right) & =0\\
\varphi'\left(0\right)h\left(0\right) & =0
\end{split}
\right\} ,\;\forall h\in\mathscr{D}_{0};
\]
this gives (\ref{eq:mul1}).

2. Suppose $\varphi\in Mul\left(\mathscr{D}_{0}\right)$, then $\varphi\in Mul\left(\mathscr{D}_{\theta}\right)$
$\Longleftrightarrow$ $\varphi\left(0\right)=\varphi\left(1\right)$.
In fact, by assumption, 
\[
\left(\varphi h\right)'\left(b\right)=\varphi\left(b\right)h'\left(b\right),\;\forall b\in\left\{ 0,1\right\} 
\]
and $\varphi h\in Mul\left(\mathscr{D}_{\theta}\right)$ $\Longleftrightarrow$
\begin{eqnarray*}
\left(\varphi h\right)'\left(1\right)+\left(\varphi h\right)'\left(1\right) & = & e^{i\theta}\left(\left(\varphi h\right)\left(x\right)-\left(\varphi h\right)'\left(0\right)\right),\;\forall h\in Mul\left(\mathscr{D}_{\theta}\right)\\
 & \Updownarrow\\
\varphi\left(1\right)\left[h\left(1\right)+h'\left(1\right)\right] & = & e^{i\theta}\varphi\left(0\right)\left[h\left(0\right)-h'\left(0\right)\right],\;\forall h\in\in Mul\left(\mathscr{D}_{\theta}\right)\\
 & \Updownarrow\\
\varphi\left(1\right) & = & \varphi\left(0\right)
\end{eqnarray*}

\begin{claim*}
$\mathscr{A}=\left\{ \varphi:\varphi'\left(0\right)=\varphi'\left(1\right)=0\right\} $
(i.e., $Mul\left(\mathscr{D}_{0}\right)$) is an algebra. 

In fact, $\left(\varphi\psi\right)'\left(b\right)=\varphi'\left(b\right)\psi\left(b\right)+\varphi\left(b\right)\psi'\left(b\right)=0$,
if $\varphi$ and $\psi$ are in $\mathscr{A}$, $b\in\left\{ 0,1\right\} $.
So $\mathscr{A}$ is an algebra. 
\end{claim*}
\end{proof}
\begin{cor}
If $\varphi\in Mul\left(\mathscr{D}_{\theta}\right)$ then 
\[
U_{\theta}\left(t\right)M\left(\varphi\right)U_{\theta}\left(-t\right)=M\left(\varphi\left(\cdot+t\right)\right),\;\forall t\in\mathbb{R}
\]
where $\varphi\in Mul$ i.e., $h\longrightarrow\varphi h$ is bounded
in $\mathscr{H}_{F}$.
\end{cor}

\begin{cor}
If $\varphi\in\mathscr{A}_{1}\cap\mathscr{A}_{2}$, then the action
of $M\left(\varphi\right)$ on the boundary values is as follows
\[
\begin{pmatrix}h\left(0\right)-h'\left(0\right)\\
h\left(1\right)+h'\left(1\right)
\end{pmatrix}\longmapsto\varphi\left(0\right)\begin{pmatrix}h\left(0\right)-h'\left(0\right)\\
h\left(1\right)+h'\left(1\right)
\end{pmatrix}
\]
(Recall $\varphi\left(0\right)=\varphi\left(1\right)$.)
\end{cor}

\subsection{Harmonic ONBs in the RKHS $\mathscr{H}_{F}$: Complex Exponentials}

Below we give some explicit results on orthonormal bases consisting
of complex exponentials. For this purpose, let $\Lambda_{\theta}=\mbox{spectrum}\left(A_{\theta}\right)$
in our example $F\left(x\right)=e^{-\left|x\right|}$, $\left|x\right|<1$,
i.e., (see \corref{spext}) 
\begin{equation}
\Lambda_{\theta}=\left\{ e^{i\lambda x}\:\Big|\:\lambda\in\mathbb{R},\mbox{ s.t. }\lambda=\theta+\tan^{-1}\left(\frac{2\lambda}{\lambda^{2}-1}\right)+2n\pi,\; n\in\mathbb{Z}\right\} .\label{eq:hb1}
\end{equation}

This covers, \emph{mutatis mutandis}, the spectral picture for other
examples of positive definite functions $F$, specified only on a
fixed finite interval. And in each such case, if the indices of $D_{F}$
are $\left(1,1\right)$, we then get, for every fixed value of $\theta$,
existence of systems of orthonormal bases $\left\{ e_{\lambda}\:|\:\lambda\in\Lambda_{\theta}\right\} $
in the corresponding RKHS $\mathscr{H}_{F}$.

\textbf{Conclusion:} Corollary \ref{cor:exponb}. While in classical
Fourier analysis, the Fourier frequencies typically are sampled on
a suitable arithmetic progression of points on the real line, this
is not the case for our present case of harmonic bases in $\mathscr{H}_{F}$.

\section{Elliptic Positive Definite Functions $F$}

The study of locally defined positive definite continuous functions
$F$ (i.e., defined on some fixed bounded connected subset in $\mathbb{R}^{d}$)
entails boundary value problems for, and extensions of, constant coefficient,
strictly elliptic partial differential operators. This is explored
in the present section. To state these connections, we will need a
few basic facts about Sobolev spaces, which are first reviewed briefly.

\subsection{A characterization of $\mathscr{H}_{F}$ for elliptic $F$ in terms
of the first Sobolev space of $\left(0,a\right)$}
\begin{defn}
Let $F:\left(-a,a\right)\rightarrow\mathbb{C}$ be a continuous positive
definite function defined on a finite interval $\left(-a,a\right)$,
$a>0$ fixed. We say that $F$ is \textbf{\emph{\uline{elliptic}}}
iff the first Sobolev space $H_{1}\left(0,a\right)$ is contained
in $\mathscr{H}_{F}$; i.e., if every continuous function on $\left[0,a\right]$
such that $h'\in L^{2}\left(0,a\right)$ is in $\mathscr{H}_{F}$. 

Equivalently, there exists a second order elliptic operator $\mbox{\ensuremath{\mathscr{D}}}$,
$T_{F}^{-1}\supset\mathscr{D}$, and a finite constant $C>0$ s.t.
\begin{equation}
\left\langle h,T_{F}^{-1}h\right\rangle _{L^{2}\left(0,a\right)}\leq C\int_{0}^{a}\left(\left|h\left(x\right)\right|^{2}+\left|h'\left(x\right)\right|^{2}\right)dx,\;\forall h\in dom(T_{F}^{-1/2}).\label{eq:e-1}
\end{equation}
\end{defn}
\begin{rem}
$T_{F}^{-1}$ is one of the selfadjoint extensions of $\mathscr{D}\big|_{C_{c}^{\infty}\left(0,a\right)}$,
as a Hermitian operator in $L^{2}\left(0,a\right)$.
\end{rem}
The following examples of $F$ are elliptic. 
\begin{example}
$F\left(x\right)=e^{-\left|x\right|}$, $-1<x<1$. A direct computation
shows 
\[
\left(T_{F}\varphi\right)''=T_{F}\varphi-2\varphi\Longleftrightarrow\varphi=\tfrac{1}{2}\left(1-\left(\tfrac{d}{dx}\right)^{2}\right)T_{F}\varphi\Longrightarrow T_{F}^{-1}\supset\underset{=:\mathscr{D}}{\underbrace{\tfrac{1}{2}\left(1-\left(\tfrac{d}{dx}\right)^{2}\right)}}
\]
For details, see \lemref{Tf}.
\end{example}

\begin{example}
$F\left(x\right)=1-\left|x\right|$, $-\frac{1}{2}<x<\frac{1}{2}$.
In this case, we have 
\[
\left(T_{F}\varphi\right)''=-2\varphi\Longleftrightarrow\varphi=-\tfrac{1}{2}\left(\tfrac{d}{dx}\right)^{2}T_{F}\varphi\Longrightarrow T_{F}^{-1}\supset\underset{=:\mathscr{D}}{\underbrace{-\tfrac{1}{2}\left(\tfrac{d}{dx}\right)^{2}}}
\]
See \exref{absm}.
\end{example}
In the following, we give a spectral representation.
\begin{rem}
The operators 
\[
\mathscr{D}_{1}:=\tfrac{1}{2}\left(1-\left(\tfrac{d}{dx}\right)^{2}\right),\;\mathscr{D}_{2}:=-\tfrac{1}{2}\left(\tfrac{d}{dx}\right)^{2}
\]
are not selfadjoint in $L^{2}\left(0,a\right)$, $a=1$ for $\mathscr{D}_{1}$,
and $a=\frac{1}{2}$ for $\mathscr{D}_{2}$; but they have selfadjoint
extensions. $T_{F}^{-1}$ is one of these selfadjoint extensions.
Also note that $T_{F}$ is positive definite, hence $T_{F}^{-1}$
is also positive definite.\end{rem}
\begin{proof}[Proof of $T_{F}$ p.d. $\Longrightarrow$ $T_{F}^{-1}$ p.d]
 On $L^{2}\left(0,a\right)\ominus\ker\left(T_{F}\right)$, we have
\begin{equation}
T_{F}=\sum_{k=1}^{\infty}\lambda_{k}P_{k}
\end{equation}
where $\left\{ P_{k}\right\} _{k\in\mathbb{N}}$ are the spectral
projections of $T_{F}$, and $\lambda_{k}>0$, for all $k\in\mathbb{N}$,
and 
\begin{equation}
\sum_{k}\lambda_{k}=\mbox{trace}\left(T_{F}\right)=a.
\end{equation}
So 
\begin{equation}
T_{F}^{-1}=\sum_{k}\lambda_{k}^{-1}P_{k}
\end{equation}
is well-defined, and 
\[
\left\langle f,T_{F}^{-1}f\right\rangle _{2}=\sum_{k}\lambda_{k}^{-1}\left\Vert P_{k}f\right\Vert _{2}^{2}\geq0,\;\forall f\in dom\left(T_{F}^{-1}\right).
\]
Note 
\begin{equation}
dom\left(T_{F}^{-1}\right)=\left\{ f:\sum_{k}\lambda_{k}^{-2}\left\Vert P_{k}f\right\Vert _{2}^{2}<\infty\right\} 
\end{equation}
by the spectral theorem.\end{proof}
\begin{thm}
If $F:\left(-a,a\right)\rightarrow\mathbb{C}$ is continuous, positive
definite, and elliptic, then 
\begin{equation}
\left\{ h:\int_{0}^{1}\left(\left|h\left(x\right)\right|^{2}+\left|h'\left(x\right)\right|^{2}\right)dx<\infty\right\} \subseteq\mathscr{H}_{F}.
\end{equation}
 \end{thm}
\begin{proof}
We must prove that if 
\begin{equation}
\int_{0}^{a}\left(\left|h\left(x\right)\right|^{2}+\left|h'\left(x\right)\right|^{2}\right)dx<\infty,
\end{equation}
then $\exists C<\infty$ s.t. 
\begin{equation}
\left|\int_{0}^{a}h\left(x\right)\varphi\left(x\right)dx\right|^{2}\leq C\left\Vert F_{\varphi}\right\Vert _{\mathscr{H}_{F}}^{2},\;\forall\varphi\in C_{c}^{\infty}\left(0,a\right).
\end{equation}
But we have that 
\begin{equation}
F_{\varphi}=T_{F}\varphi,\;\forall\varphi\in C_{c}^{\infty}\left(0,a\right)\subset L^{2}\left(0,a\right)
\end{equation}
where $T_{F}$ is the Mercer operator. Now, 
\[
\left\langle h,\varphi\right\rangle _{2}=\left\langle h,T_{F}^{-1}T_{F}\varphi\right\rangle _{2}
\]
and if $h,h'\in L^{2}\left(0,a\right)$, then $\exists C<\infty$
s.t.
\[
\left\Vert T_{F}^{-1/2}h\right\Vert _{2}^{2}\underset{\left(\ref{eq:e-1}\right)}{\leq}C\int_{0}^{a}\left(\left|h\left(x\right)\right|^{2}+\left|h'\left(x\right)\right|^{2}\right)dx,\;\forall h\in\underset{\text{Soblev}_{1}}{\underbrace{H_{1}\left(0,a\right)}}
\]
and so 
\begin{eqnarray*}
\left|\left\langle h,\varphi\right\rangle _{2}\right|^{2} & = & \left|\left\langle h,T_{F}^{-1}T_{F}\varphi\right\rangle _{2}\right|^{2}\\
 & = & \left|\left\langle T_{F}^{-1/2}h,T_{F}^{1/2}\varphi\right\rangle _{2}\right|^{2}\\
 & \leq & \left\Vert T_{F}^{-1/2}h\right\Vert _{2}^{2}\left\Vert T_{F}^{1/2}\varphi\right\Vert _{2}^{2}\\
 & \leq & C\left(\int_{0}^{a}\left|h\right|^{2}+\left|h'\right|^{2}\right)\left\langle \varphi,T_{F}\varphi\right\rangle _{2}\\
 & = & \mbox{const}\left\Vert F_{\varphi}\right\Vert _{\mathscr{H}_{F}}^{2}.
\end{eqnarray*}

\end{proof}

\subsection{If $F$ is elliptic then the operator $D_{F}$ has indices $\left(1,1\right)$}
\begin{cor}
Let $F:\left(-a,a\right)\rightarrow\mathbb{C}$ be an elliptic continuous
p.d. function and set $D_{F}\left(F_{\varphi}\right)=F_{\varphi'}$,
$\varphi\in C_{c}^{\infty}\left(0,a\right)$; then $D_{F}$ has deficiency
indices $\left(1,1\right)$ as a skew Hermitian operator in the RKHS
$\mathscr{H}_{F}$.
\end{cor}
Below, we give a more direct argument for why $D_{F}$ in $\mathscr{H}_{F}$
has deficiency indices $\left(1,1\right)$ in the case where $F\left(x\right)=1-\left|x\right|$,
$-\frac{1}{2}<x<\frac{1}{2}$. 

Recall $\mathscr{H}_{F}$ consists of continuous function on $\left[0,\frac{1}{2}\right]$,
so in particular, certain restrictions of continuous functions of
$\mathbb{R}$. 
\begin{lem}
\label{lem:abs}Let $F\left(x\right)=1-\left|x\right|$, $\left|x\right|<\frac{1}{2}$,
then a continuous function $h$ on $\left[0,\frac{1}{2}\right]$ is
in $\mathscr{H}_{F}$ iff $h'\in L^{2}\left(0,\frac{1}{2}\right)$
where $h'=$ the distributional derivative.\end{lem}
\begin{proof}
Let $T_{F}$ be the Mercer operator, i.e., $T_{F}:L^{2}\left(0,\frac{1}{2}\right)\rightarrow L^{2}\left(0,\frac{1}{2}\right)$,
and 
\[
\left(T_{F}\varphi\right)\left(x\right)=\int_{0}^{\frac{1}{2}}\varphi\left(y\right)F\left(x-y\right)dy,\;\forall\varphi\in C_{c}^{\infty}\left(0,\tfrac{1}{2}\right).
\]
Recall $T_{F}$ is bounded, positive definite, selfadjoint, and trace
class, $trace\left(T_{F}\right)=\frac{1}{2}$. Hence $T_{F}^{1/2}$
and $T_{F}^{-1}$ are well-defined, but $T_{F}^{-1}$ is unbounded.
We showed, in \exref{absm}, that $T_{F}^{-1}$ is a selfadjoint extension
of $-\frac{1}{2}\left(\frac{d}{dx}\right)^{2}$ with dense domain
$C_{c}^{\infty}\left(0,\frac{1}{2}\right)$ in $L^{2}\left(0,\frac{1}{2}\right)$. 

Now assume $h'\in L^{2}\left(0,\frac{1}{2}\right)$. Then for all
$\varphi\in C_{c}^{\infty}\left(0,\frac{1}{2}\right)$, we have 
\begin{align*}
\left|\int_{0}^{\frac{1}{2}}\overline{h\left(x\right)}\varphi\left(x\right)dx\right|^{2} & =\left|\left\langle h,T_{F}^{-1}T_{F}\varphi\right\rangle _{L^{2}\left(0,\frac{1}{2}\right)}\right|^{2}\\
 & =\Big|\langle T_{F}^{-1/2}h,\underset{T_{F}^{1/2}}{\underbrace{T_{F}^{-1/2}T_{F}}\varphi}\rangle_{L^{2}\left(0,\frac{1}{2}\right)}\Big|^{2}\\
 & \leq\underset{=:C}{\underbrace{\frac{1}{2}\left\Vert h'\right\Vert _{L^{2}\left(0,\frac{1}{2}\right)}^{2}}}\left\Vert T_{F}^{1/2}\varphi\right\Vert _{L^{2}\left(0,\frac{1}{2}\right)}^{2}\\
 & =C\left\langle \varphi,T_{F}\varphi\right\rangle _{L^{2}\left(0,\frac{1}{2}\right)}=C\left\Vert T_{F}\varphi\right\Vert _{\mathscr{H}_{F}}^{2}.
\end{align*}
As a result (application of Riesz' theorem to $\mathscr{H}_{F}$),
we conclude that the function $h$ represents a unique element in
$\mathscr{H}_{F}$; i.e., we proved that
\[
h'\in L^{2}\left(0,\tfrac{1}{2}\right)\Longrightarrow h\in\mathscr{H}_{F}.
\]
\end{proof}
\begin{cor}
Let $F\left(x\right)=1-\left|x\right|$, $\left|x\right|<\frac{1}{2}$.
The two functions $e^{\pm x}\big|_{\left[0,\frac{1}{2}\right]}$ are
in $\mathscr{H}_{F}$. In particular, $D_{F}$ has deficiency $\left(1,1\right)$. \end{cor}
\begin{proof}
An application of \lemref{abs}.\end{proof}
\begin{rem}
The proof of \lemref{abs} shows that the converse implication holds
as well: If a continuous function $h$ on $\left[0,\frac{1}{2}\right]$
is in $\mathscr{H}_{F}$, then it follows that its distributional
derivative $h'$ is in $L^{2}\left(0,\frac{1}{2}\right)$. 

The argument applies to every positive definite continuous function
$F$ such that the corresponding Mercer operator $T_{F}$ has $T_{F}^{-1}$
be an extension of a second order elliptic differential operator.

BELOW WE PROVE THE FOLLOWING ASSERTIONS:
\end{rem}

\subsection{If $F$ is elliptic then its distribution derivative has a Dirac
discontinuity at $x=0$}

\subsection{If $F$ is elliptic then there is an associated system of linear
conditions for the kernel functions for $\mathscr{H}_{F}$ at the
two endpoints $0$ and $a$}

\subsection{Two examples $F=e^{-\left|x\right|}$, and $F=1-\left|x\right|$,
on the respective intervals, are elliptic}

\subsection{Translation representation for the unitary one-parameter groups $U(t)$
in $\mathscr{H}_{F}$ }
\begin{cor}
If $F:\left(-a,a\right)\rightarrow\mathbb{C}$ is an elliptic continuous
positive definite function, $0<a<\infty$; then the dense domain $\left\{ F_{\varphi}\right\} _{\varphi\in C_{c}^{\infty}\left(0,a\right)}$,
i.e, $dom\left(D_{F}\right)\subset\mathscr{H}_{F}$, is given by a
rank-2 subspace of $\mathbb{C}^{4}$ in the form 
\[
\left(h\left(0\right),h'\left(0\right),h\left(a\right),h'\left(a\right)\right).
\]
The unitary one-parameter group $\left\{ U^{\left(\theta\right)}\left(t\right)\right\} _{t\in\mathbb{R}}$
in $\mathscr{H}_{F}$ generated by skew-adjoint extensions of $D_{F}$
(in $\mathscr{H}_{F}$) are determined by a pair of boundary conditions:
\[
h\in\mathscr{H}_{F}\longleftrightarrow\left\{ h,h'\right\} \subset L^{2}\left(0,a\right)\oplus L^{2}\left(0,a\right).
\]
When $\theta$ is fixed, then there is a closed subspace $\mathscr{L}^{\left(\theta\right)}$
in $L^{2}\left(0,a\right)\oplus L^{2}\left(0,a\right)$ such that
\[
h\longmapsto U^{\left(\theta\right)}\left(t\right)h\;\mbox{in }\mathscr{H}_{F}
\]
is equivalent to a translation representation 
\[
\left(T_{t}^{\theta},S_{t}^{\theta}\right)\::\:\begin{Bmatrix}h\longmapsto T_{t}^{\left(\theta\right)}h\left(x\right)=h\left(x+t\right)\\
h'\longmapsto S_{t}^{\left(\theta\right)}h'\left(x\right)=h'\left(x+t\right)
\end{Bmatrix}\;\mbox{in }\mathscr{L}^{\left(\theta\right)},\;\forall s,t\in\mathbb{R}.
\]
 \end{cor}
\begin{example}[Greens-Gauss-Stokes]
 Recall the elliptic operator $-\frac{1}{2}\left(\frac{d}{dx}\right)^{2}$
in $L^{2}\left(0,\frac{1}{2}\right)$ is defined initially on $C_{c}^{\infty}\left(0,\frac{1}{2}\right)$;
it is densely defined and Hermitian, but not selfadjoint. To get selfadjoint
extensions, we look for dense subspace $\mathscr{L}\subset L^{2}\left(0,\frac{1}{2}\right)$
s.t.
\[
\int_{0}^{\frac{1}{2}}h''\left(x\right)k\left(x\right)\, dx-\int_{0}^{\frac{1}{2}}h\left(x\right)k''\left(x\right)\, dx=\left[h'k-hk'\right]_{bd}\equiv0,\;\forall h,k\in\mathscr{L};
\]
where $\left[F\right]_{bd}:=F\left[0\right]-F\left[1/2\right]$, applied
to $F=h'k-hk'$, $\forall h,k\in\mathscr{L}$. 

Now fix $F=1-\left|x\right|$, $\left|x\right|<\frac{1}{2}$, consider
the Mercer operator
\[
\left(T_{F}\varphi\right)\left(x\right)=\int_{0}^{\frac{1}{2}}\varphi\left(y\right)F\left(x-y\right)dy=\int_{0}^{\frac{1}{2}}\varphi\left(y\right)\left(1-\left|x-y\right|\right)dy
\]
and set $\mathscr{L}=\mathscr{L}_{F}=\left\{ T_{F}\varphi\:\big|\:\varphi\in C_{c}^{\infty}\left(0,\frac{1}{2}\right)\right\} $.
We may restrict to real-valued functions since $F$ is real-valued.
Set $h=T_{F}\varphi$, and $k=T_{F}\psi$. Then,
\begin{align*}
h\left(0\right) & =\int_{0}^{\frac{1}{2}}\left(1-y\right)\varphi\left(y\right)dy & h'\left(0\right) & =\int_{0}^{\frac{1}{2}}\varphi\left(y\right)dy\\
h\left(\tfrac{1}{2}\right) & =\int_{0}^{\frac{1}{2}}\left(\tfrac{1}{2}+y\right)\varphi\left(y\right)dy & h'\left(\tfrac{1}{2}\right) & =-\int_{0}^{\frac{1}{2}}\varphi\left(y\right)dy
\end{align*}
 And
\[
\left[h'k-hk'\right]_{bd}=\tfrac{3}{2}\left(\int_{0}^{\frac{1}{2}}\varphi\int_{0}^{\frac{1}{2}}\psi-\int_{0}^{\frac{1}{2}}\varphi\int_{0}^{\frac{1}{2}}\psi\right)=0.
\]
Hence $T_{F}^{-1}$ is the unique selfadjoint extension of $-\frac{1}{2}\left(\frac{d}{dx}\right)^{2}\big|_{C_{c}^{\infty}\left(0,\frac{1}{2}\right)}$
in $L^{2}\left(0,\frac{1}{2}\right)$ corresponding to the boundary
conditions:
\begin{equation}
\begin{split}\begin{cases}
h'\left(0\right)+h'\left(\tfrac{1}{2}\right)=0\\
h\left(0\right)+h\left(\tfrac{1}{2}\right)=\tfrac{3}{2}h'\left(0\right)
\end{cases}\end{split}
\label{eq:absbd}
\end{equation}
Suppose $T_{F}h=\lambda h$, then $-\frac{1}{2}\left(\lambda h\right)''=h$,
and $h$ satisfies (\ref{eq:absbd}). Let $k^{2}=-\tfrac{2}{\lambda}$,
we have $h=Ae^{ikx}+B^{-ikx}$, such that 
\begin{align*}
\left(1+e^{ik/2}\right)A-\left(1+e^{-ik/2}\right)B & =0\\
A\left(1+e^{ik/2}-\tfrac{3}{2}ik\right)+\left(1+e^{-ik/2}+\tfrac{3}{2}ik\right)B & =0
\end{align*}
Setting the determinant of the coefficient matrix to zero, we then
get
\begin{equation}
\tan\left(k/2\right)=\frac{4}{3k}.\label{eq:abssp}
\end{equation}

\end{example}
Recall that $P$ is a positive polynomial if $P\left(\xi\right)=\overline{Q\left(\xi\right)}Q\left(\xi\right)$,
$\forall\xi\in\mathbb{R}$. (In one variable, every $P\geq0$ is a
square; in higher dimensions, sums of squares.) 
\begin{lem}
Fix $a>0$. Let $P$ be a positive polynomial (in one variable), and
set 
\[
\mathscr{D}^{\left(P\right)}=P\left(\tfrac{1}{i}\tfrac{d}{dx}\right)\;\mbox{on }C_{c}^{\infty}\left(0,a\right)\subset L^{2}\left(0,a\right)
\]
and pick a fundamental solution $F=F^{\left(P\right)}$, so 
\[
\mathscr{D}^{\left(P\right)}T_{F^{\left(P\right)}}\varphi=\varphi,\;\forall\varphi\in C_{c}^{\infty}\left(0,a\right).
\]
Then, $F^{\left(P\right)}$ is positive definite, and is extendible. \end{lem}
\begin{example}
Let 
\begin{align*}
\mathscr{D}_{2} & =-\tfrac{1}{2}\left(\tfrac{d}{dx}\right)^{2}\big|_{C_{c}^{\infty}\left(0,\frac{1}{2}\right)}\mbox{ in }L^{2}\left(0,\tfrac{1}{2}\right)\\
\mathscr{D}_{3} & =\tfrac{1}{2}\left(I-\left(\tfrac{d}{dx}\right)^{2}\right)\big|_{C_{c}^{\infty}\left(0,1\right)}\mbox{ in }L^{2}\left(0,1\right)
\end{align*}
with the corresponding fundamental solutions
\begin{align*}
F_{2}\left(x\right) & =1-\left|x\right|\\
F_{3}\left(x\right) & =e^{-\left|x\right|};
\end{align*}
i.e., 
\[
\left\{ \begin{split}\mathscr{D}_{2}F_{2}\left(\cdot-x\right)=\delta_{x}\\
\mathscr{D}_{3}F_{3}\left(\cdot-x\right)=\delta_{x}
\end{split}
\right\} \Longleftrightarrow\left\{ \begin{split}\mathscr{D}_{2}T_{F_{2}}\varphi=\varphi,\;\forall\varphi\in C_{c}^{\infty}\left(0,\tfrac{1}{2}\right)\\
\mathscr{D}_{3}T_{F_{3}}\varphi=\varphi,\;\forall\varphi\in C_{c}^{\infty}\left(0,1\right)
\end{split}
\right\} 
\]
\end{example}
\begin{cor}
Let $F=F^{\left(P\right)}$, $P$ is positive, then
\[
T_{F^{\left(P\right)}}:L^{2}\left(0,a\right)\rightarrow L^{2}\left(0,a\right),\;\mbox{by}
\]
\[
\left(T_{F^{\left(P\right)}}\varphi\right)\left(x\right)=\int_{0}^{a}\varphi\left(y\right)F^{\left(P\right)}\left(x-y\right)dy
\]
is positive definite, selfadjoint, and trace class. Moreover, 
\[
T_{F^{\left(P\right)}}^{-1}\supset P\left(\tfrac{1}{i}\tfrac{d}{dx}\right)\big|_{C_{c}^{\infty}\left(0,a\right)}.
\]
Setting 
\[
h=T_{F^{\left(P\right)}}\varphi,\;\mbox{and }k=T_{F^{\left(P\right)}}\psi,
\]
then 
\[
W=\overline{Q\left(\tfrac{1}{i}\tfrac{d}{dx}\right)h}\, k-\overline{h}\, Q\left(\tfrac{1}{i}\tfrac{d}{dx}\right)k
\]
satisfying 
\[
W\left(0\right)=W\left(a\right)
\]
by Greens-Gauss-Stokes principle.\end{cor}
\begin{example}
Let $F_{1}\left(x\right)=1-\left|x\right|$ in the interval $\left(-\frac{1}{2},\frac{1}{2}\right)$,
and let $F_{1}'$ and $F_{1}''$ be the corresponding distribution
derivatives. Let $H$ bet the Heaviside function
\[
H\left(x\right)=\begin{cases}
0 & \mbox{if }x<0\\
1 & \mbox{if }x\geq0
\end{cases}
\]
then
\begin{equation}
\begin{cases}
F_{1}'=1-2H,\;\mbox{and} & \mbox{}\\
F_{1}''=-2\delta_{0}.
\end{cases}\label{eq:absdev}
\end{equation}
\end{example}
\begin{proof}
This is immediate from elementary Schwartz to distribution calculus. \end{proof}
\begin{example}
Let $F_{2}\left(x\right)=e^{-\left|x\right|}$ in the interval $\left(-1,1\right)$;
then for the respective distributional derivatives $F_{2}'$ and $F_{2}''$
we have
\begin{equation}
\begin{cases}
F_{2}'\left(x\right)=\left(1-2H\left(x\right)\right)e^{-\left|x\right|},\:\mbox{and}\\
F_{2}''\left(x\right)=F_{2}\left(x\right)-2\delta\left(x-0\right).
\end{cases}\label{eq:expdev}
\end{equation}
\end{example}
\begin{prop}
Let $F:\left(-a,a\right)\rightarrow\mathbb{C}$ be continuous and
positive definite. Suppose $F$ is elliptic with $F=F^{\left(P\right)}$
and 
\begin{equation}
P\left(\tfrac{1}{i}\tfrac{d}{dx}\right)T_{F}\varphi=\varphi,\;\forall\varphi\in C_{c}^{\infty}\left(0,a\right);\label{eq:e2-1}
\end{equation}
then 
\begin{equation}
\left(P\left(\tfrac{1}{i}\tfrac{d}{dx}\right)F\right)\left(x\right)=\delta\left(x-0\right);\label{eq:e2-2}
\end{equation}
in particular $F'$ is discontinuous at $x=0$. \end{prop}
\begin{proof}
Given $F$ as above, let $P=P\left(\tfrac{1}{i}\tfrac{d}{dx}\right)$
be the corresponding elliptic PDO. It has an even-degree, say $2m$,
leading term $\pm\left(\frac{d}{dx}\right)^{2m}$. Since $T_{F}\varphi=\varphi\ast F$,
$\forall\varphi\in C_{c}^{\infty}\left(0,a\right)$, we get 
\[
\varphi\ast\delta_{0}=\varphi=\varphi\ast P\left(\tfrac{1}{i}\tfrac{d}{dx}\right)F;
\]
and therefore 
\[
\delta_{0}=P\left(\tfrac{1}{i}\tfrac{d}{dx}\right)F=\cdots\pm F^{\left(2m\right)}
\]
from which the assertion follows.\end{proof}
\begin{rem}
\textbf{\uline{Higher dimensions.}} The extension from $\mathbb{R}$
to $\mathbb{R}^{n}$ when $n>1$ is subtle for a number of reasons:

(i) Rather than just a single Hermitian (symmetric) operator $D^{\left(F\right)}:F_{\varphi}\longmapsto\frac{1}{i}F_{\varphi'}$
in $\mathscr{H}_{F}$, we will need to consider partial derivatives
$\frac{\partial}{\partial x_{k}}$, $k=1,\ldots,n$, where $n>1$;
so a system of unbounded operators:
\[
D^{\left(F\right)}:F_{\varphi}\longmapsto F_{\frac{\partial\varphi}{\partial x_{k}}},
\]
defined for $\varphi\in C_{c}^{\infty}\left(\Omega\right)$, where
$\Omega\subset\mathbb{R}^{n}$ is a given open domain. 

(ii) Even for the example $n=1$, of $F\left(x\right)=e^{-\left|x\right|}$
in $\left|x\right|<1$; passing to high dimensions yields difficulties
as follows: If $x=\left(x_{1},\ldots,x_{n}\right)$, $\lambda=\left(\lambda_{1},\ldots,\lambda_{n}\right)$,
and $\left|x\right|=\left(\sum_{k=1}^{n}x_{k}^{2}\right)^{1/2}$;
then 
\[
e^{-\left|x\right|}=\mathscr{F}_{\lambda}\left(2^{n}\pi^{\frac{n-1}{2}}\Gamma\left(\frac{n+1}{2}\right)\frac{1}{\left(1+\left|\lambda\right|^{2}\right)^{\frac{n+1}{2}}}\right)\left(x\right)
\]
where $\mathscr{F}$ denotes Fourier transform in $\mathbb{R}^{n}$,
i.e.,
\[
e^{-\left|x\right|}=\int_{\mathbb{R}^{n}}2^{n}\pi^{\frac{n-1}{2}}\Gamma\left(\frac{n+1}{2}\right)\frac{e^{i\lambda x}}{\left(1+\left|\lambda\right|^{2}\right)^{\frac{n+1}{2}}}d\lambda_{1}\cdots d\lambda_{n}.
\]

(iii) If $F=e^{-\left|x\right|}$ is on open neighborhood $\Omega\subset\mathbb{R}^{n}$,
and $\triangle=\sum_{k=1}^{n}\left(\frac{\partial}{\partial x_{k}}\right)^{2}$
is the Laplacian, it follows from (ii) that 
\[
\frac{1}{2^{n}}\left(1-\triangle\right)^{\frac{n+1}{2}}F=\delta_{0}
\]
taken in the sense of Schwartz distribution; equivalently
\[
\frac{1}{2^{n}}\left(I-\triangle\right)^{\frac{n+1}{2}}T_{F}\varphi=\varphi,\;\forall\varphi\in C_{c}^{\infty}\left(\Omega\right),
\]
and where $T_{F}$ is an $n$-dimension Mercer operator
\[
\left(T_{F}\varphi\right)\left(x\right)=\int_{\Omega}\varphi\left(y\right)F\left(x-y\right)dy,\;\forall\varphi\in C_{c}^{\infty}\left(\Omega\right).
\]

Note as before, $\mathscr{H}_{F}$ is the RKHS obtained by completion
of 
\[
\left\Vert T_{F}\varphi\right\Vert _{\mathscr{H}_{F}}^{2}:=\int_{\Omega}\int_{\Omega}\overline{\varphi\left(x\right)}\varphi\left(y\right)F\left(x-y\right)dxdy.
\]
\end{rem}
\begin{thm}
\label{thm:nd}Let $\Omega\subset\mathbb{R}^{n}$ be open and connected
and let $F:\Omega-\Omega\rightarrow\mathbb{C}$ be positive definite.
Suppose there is a $\mu\in\mathscr{M}\left(\mathbb{R}^{n}\right)$,
$\mu\in Ext\left(F\right)$ such that $\mu$ has compact support in
$\mathbb{R}^{n}$; then all the $n$ operators
\begin{equation}
D_{k}^{\left(F\right)}:F_{\varphi}\longmapsto\frac{1}{i}F_{\frac{\partial\varphi}{\partial x_{k}}},\; k=1,\ldots,n\label{eq:nd1}
\end{equation}
are bounded in $\mathscr{H}_{F}$. \end{thm}
\begin{proof}
We need the following:
\begin{lem}
\label{lem:nd1}Let $F$, $\Omega$, $\mathscr{H}_{F}$, and $\mu\in Ext\left(F\right)$
be as stated in the theorem; then 
\begin{equation}
\left\Vert F_{\varphi}\right\Vert _{\mathscr{H}_{F}}^{2}=\int_{\mathbb{R}^{n}}\left|\widehat{\varphi}\left(\lambda\right)\right|^{2}d\mu\left(\lambda\right)\label{eq:nd2}
\end{equation}
where $\lambda=\left(\lambda_{1},\ldots,\lambda_{n}\right)\in\mathbb{R}^{n}$,
and $\widehat{\varphi}=$ the $\mathbb{R}^{n}$-Fourier transform,
and $\varphi\in C_{c}\left(\Omega\right)$. \end{lem}
\begin{proof}
Let $F$, $\mathscr{H}_{F}$, $\mu$, and $\varphi\in C_{c}\left(\Omega\right)$
be as stated in the lemma, then 
\begin{eqnarray*}
\left(\mbox{LHS}\right)_{\left(\ref{eq:nd2}\right)} & = & \int_{\Omega}\int_{\Omega}\overline{\varphi\left(x\right)}\varphi\left(y\right)F\left(x-y\right)dxdy\\
 & = & \int_{\Omega}\int_{\Omega}\overline{\varphi\left(x\right)}\varphi\left(y\right)\int_{\mathbb{R}^{n}}e^{i\left(x-y\right)\lambda}d\mu\left(\lambda\right)dxdy\;\left(\mbox{since }\mu\in Ext\left(F\right)\right)\\
 & \underset{\left(\text{by Fubini}\right)}{=} & \int_{\mathbb{R}^{n}}\left|\int_{\Omega}\varphi\left(x\right)e^{-ix\lambda}dx\right|^{2}d\mu\left(\lambda\right)\\
 & = & \int_{\mathbb{R}^{n}}\left|\widehat{\varphi}\left(\lambda\right)\right|^{2}d\mu\left(\lambda\right)=\left(\mbox{RHS}\right)_{\left(\ref{eq:nd2}\right)}
\end{eqnarray*}

\end{proof}
\begin{flushleft}
\emph{\uline{Proof of \mbox{\thmref{nd}} continued}} 
\par\end{flushleft}

Now let $\varphi\in C_{c}^{\infty}\left(\Omega\right)$, then by the
lemma, we have
\begin{align*}
\left|\left\langle F_{\varphi},D_{k}^{\left(F\right)}F_{\varphi}\right\rangle _{\mathscr{H}_{F}}\right| & =\left|\int_{\mathbb{R}^{n}}\lambda_{k}\left|\widehat{\varphi}\left(\lambda\right)\right|^{2}d\mu\left(\lambda\right)\right|\\
 & \leq\mbox{diam}\left(\mbox{suppt}\left(\mu\right)\right)\int_{\mathbb{R}^{n}}\left|\widehat{\varphi}\left(\lambda\right)\right|^{2}d\mu\left(\lambda\right)\\
 & =\mbox{diam}\left(\mbox{suppt}\left(\mu\right)\right)\left\Vert F_{\varphi}\right\Vert _{\mathscr{H}_{F}}^{2}\;(\mbox{by lemma }\ref{lem:nd1})
\end{align*}
where 
\begin{align*}
 & \mbox{diam}\left(\mbox{suppt}\left(\mu\right)\right)\\
= & \inf\left\{ k\in\mathbb{R}_{+}\:\big|\:\mbox{suppt\ensuremath{\left(\mu\right)\subset\left\{  \lambda:\left|\lambda\right|\leq k\right\} } }\right\} <\infty.
\end{align*}
\end{proof}
\begin{cor}
The functions 
\[
F_{k}\left(x\right)=\left(\frac{\sin\pi x}{\pi x}\right)^{k},\; k=1,2,\ldots
\]
used in the theory of B-splines, and in Shannon-sampling have this
property: That is, if $F_{k}$ is defined on an interval $\left(-a,a\right)$,
then the corresponding Hermitian symmetric operators $D^{\left(F_{k}\right)}$
in $\mathscr{H}_{F}$ are all bounded, in fact
\begin{equation}
\left\Vert D^{\left(F_{k}\right)}\left(u\right)\right\Vert _{\mathscr{H}_{F}}\leq\frac{k}{2}\left\Vert u\right\Vert _{\mathscr{H}_{F}}\label{eq:sd3}
\end{equation}
holds for all $u\in\mathscr{H}_{F}$. \end{cor}
\begin{proof}
Fix $k$, then $D^{\left(F_{k}\right)}$ is defined on its dense domain
by 
\[
D^{\left(F_{k}\right)}\left(T_{F}\varphi\right)=\frac{1}{i}T_{F_{k}}\varphi',\;\forall\varphi\in C_{c}^{\infty}\left(0,a\right).
\]
The bound $k/2$ on the RHS in (\ref{eq:sd3}) arises as follows:
Let $B=\chi_{\left(-\frac{1}{2},\frac{1}{2}\right)}$, then 
\[
F_{k}\left(x\right)=\left(\frac{\sin\pi x}{\pi x}\right)^{k}=\underset{\mbox{k times}}{\underbrace{\left(B\ast\cdots\ast B\right)}}^{\wedge}.
\]
\end{proof}
\begin{rem}
Note that the operator bound $k/2$ on the RHS in (\ref{eq:sd3})
is independent on the size of the interval $\left(-a,a\right)$ where
$F_{k}$ is specified.
\end{rem}

\section{Extension from finite or countably infinite subsets}

In this section, we consider extension of positive definite functions
defined, initially, only on some fixed finite or countably infinite
subset of $\mathbb{R}$, or of $\mathbb{R}^{n}$.

An understanding of this problem is of use in a variety of optimization
problems, for example, in the determination of suitable families of
splines (from numerical analysis) involving a choice of a penalty
term.

In sampling theory, and in the theory of splines it is of interest
to extend positive definite functions $F$ defined on discrete subsets
of $\mathbb{R}$, finite, or infinite. 

Below we consider this problem: Let $S\subset\mathbb{R}$ be a countably
discrete subset and let 
\begin{equation}
F:S-S\longrightarrow\mathbb{C}\label{eq:s1}
\end{equation}
be a positive definite function, defined on $S-S=\left\{ x-y\:\big|\: x,y\in S\right\} $,
i.e., we have for all finite summation, $\left\{ c_{k}\right\} \subset\mathbb{C}^{\mbox{finite}}$,
$\left\{ s_{k}\right\} \subset S$, 
\begin{equation}
\sum_{j}\sum_{k}\overline{c_{j}}c_{k}F\left(s_{j}-s_{k}\right)\geq0,\label{eq:s2}
\end{equation}
and let $\mathscr{H}_{F}$ denote the corresponding reproducing kernel
Hilbert space (of functions on $S$) w.r.t. 
\begin{equation}
\left\langle F\left(\cdot-s\right),F\left(\cdot-t\right)\right\rangle _{\mathscr{H}_{F}}=F\left(s-t\right),\;\forall s,t\in S.\label{eq:s3}
\end{equation}

\begin{thm}
Let $S$ and $F$ be as above, i.e., $F$ is a fixed positive definite
function defined on $S-S$, and for $s\in S$, set $F\left(\cdot-s\right)=:F_{s}$.
Then $F$ has a continuous positive definite extension to all of $\mathbb{R}$
if and only if there is a finite Borel measure $\mu$ on $\mathbb{R}$
such that the assignment $W=W_{\left(\mu\right)}$, $W\left(F_{s}\right)=e^{is\cdot}$,
$s\in S$, extends to an isometric linear operator
\begin{equation}
\widetilde{W}:\mathscr{H}_{F}\longrightarrow L^{2}\left(\mathbb{R},\mu\right).\label{eq:s4}
\end{equation}
\end{thm}
\begin{proof}
\textbf{Step 1.} Suppose $F$ has a continuous positive definite extension,
say $G$, to $\mathbb{R}$, as 
\begin{equation}
F\left(x\right)=G\left(x\right),\;\forall s\in S-S;\label{eq:s5}
\end{equation}
and let $\mu$ be the corresponding Borel measure (which exists by
Bochner's theorem) such that
\[
G\left(x\right)=G_{\mu}\left(x\right)=\int_{\mathbb{R}}e^{-ix\lambda}d\mu\left(\lambda\right),\; x\in\mathbb{R}.
\]
Now we use (\ref{eq:s3}) to compute the $\mathscr{H}_{F}$-norm of
\begin{equation}
\sum_{k}c_{k}F_{s_{k}}\label{eq:s6}
\end{equation}
for any finite sum, $\left\{ c_{k}\right\} \subset\mathbb{C}$, $\left\{ s_{k}\right\} \subset S$,
where the summation in (\ref{eq:s6}) is assumed finite. We have
\begin{eqnarray*}
\left\Vert \sum_{k}c_{k}F_{s_{k}}\right\Vert _{\mathscr{H}_{F}}^{2} & = & \sum_{j}\sum_{k}\overline{c_{j}}c_{k}F\left(s_{j}-s_{k}\right)\\
 & \underset{\left(\text{by }\ref{eq:s5}\right)}{=} & \sum_{j}\sum_{k}\overline{c_{j}}c_{k}G_{\mu}\left(s_{j}-s_{k}\right)\\
 & \underset{\left(\text{by }\ref{eq:s6}\right)}{=} & \sum_{j}\sum_{k}\overline{c_{j}}c_{k}\int_{\mathbb{R}}e^{i\left(s_{j}-s_{k}\right)\lambda}d\mu\left(\lambda\right)\\
 & = & \int_{\mathbb{R}}\left|\sum_{k}c_{k}e^{is_{k}\lambda}\right|^{2}d\mu\left(\lambda\right).
\end{eqnarray*}
Hence, introducing $W=W_{\mu}$ as in (\ref{eq:s4}), 
\begin{equation}
W_{\mu}\left(F_{s}\right):=e^{is\cdot}\label{eq:s7}
\end{equation}
as a function on $\mathbb{R}$, it follows that $W_{\mu}$ extends
by linearity and norm-closure $\mathscr{H}_{F}\longrightarrow L^{2}\left(\mathbb{R},\mu\right)$
to yield a well-defined isometry from $\mathscr{H}_{F}$ into $L^{2}\left(\mathbb{R},\mu\right)$
with the stated properties.

\textbf{Step 2.} The converse implication is an immediate consequence
of the argument above: If an isometry $W_{\mu}$ exists as in the
statement of the theorem, see (\ref{eq:s7}), then a measure $\mu$
exists having the stated properties. Then set $G=G_{\mu}$, $L=$
the positive definite function from (\ref{eq:s6}). We claim that
(\ref{eq:s5}) holds, i.e., that 
\begin{equation}
F\left(s_{1}-s_{2}\right)=G_{\mu}\left(s_{1}-s_{2}\right),\:\forall s_{1},s_{2}\in S.\label{eq:s8}
\end{equation}
Let $c_{1},c_{2}\in\mathbb{C}$, then 
\begin{equation}
\sum_{j}\sum_{k}\overline{c_{j}}c_{k}F\left(s_{j}-s_{k}\right)=\int_{\mathbb{R}}\left|\sum_{j}c_{j}e^{is_{j}\lambda}\right|^{2}d\mu\left(\lambda\right)\label{eq:s9}
\end{equation}
holds.

But both sides in (\ref{eq:s9}) is a quadratic from on $\mathbb{C}^{2}$.
Since the respective quadratic forms on the two sides in (\ref{eq:s9})
agree, it follows that the extensions must agree, i.e., 
\begin{align*}
F\left(s_{j}-s_{k}\right) & =\int_{\mathbb{R}}e^{-is_{j}\lambda}e^{is_{k}\lambda}d\mu\left(\lambda\right)\\
 & =\int_{\mathbb{R}}e^{-i\left(s_{j}-s_{k}\right)\lambda}d\mu\left(\lambda\right)=G_{\mu}\left(s_{j}-s_{k}\right)
\end{align*}
which is the desired conclusion (\ref{eq:s8}).
\end{proof}

\section{Some ONBs in $\mathscr{H}_{F}$}

In this section we give a \textquotedblleft divided-differences\textquotedblright{}
algorithm for constructing particular orthonormal bases (ONB) in the
RKHSs $\mathscr{H}_{F}$ built from a fixed continuous p.d. function
$F$ in a bounded interval $J$. Our ONB-algorithm is based on operations
on a dyadic multiresolution in the interval $J$. Our applications
include B-splines.

Fix $a>0$. Let $F:\left(-a,a\right)\rightarrow\mathbb{C}$ be a continuous
positive definite (p.d.) function, and assume $F\left(0\right)=1$.
We construct an orthonormal basis in the corresponding reproducing
kernel Hilbert space (RKHS) $\mathscr{H}_{F}$. To simplify the notations,
we assume $F$ is real-valued. 

Recall that, by continuity, $F$ extends uniquely to the endpoints
$x=\pm a$. Set 
\begin{equation}
F_{x}\left(y\right)=F\left(x-y\right),\;\forall x,y\in\left[0,a\right].\label{eq:gs-def}
\end{equation}

\begin{prop}
Fix $F:\left(-a,a\right)\rightarrow\mathbb{C}$ p.d., continuous.
Let $\mathscr{H}_{F}$ be the corresponding RKHS. If $S\subset\left[0,a\right]$
is a dense subset, then $\left\{ F_{s}\right\} _{s\in S}$ is total,
i.e., the $\mathscr{H}_{F}$-closed subspace of $\left\{ F_{s}\right\} _{s\in S}$
is $\mathscr{H}_{F}$. \end{prop}
\begin{proof}
We must show that if $h\in\mathscr{H}_{F}$ and $\left\langle F_{s},h\right\rangle =0$,
$\forall s\in S$ $\Longrightarrow$ $h=0$. 

Note $\left\langle F_{s},h\right\rangle =h\left(s\right)=0$, $\forall s\in S$.
Since $h$ is uniformly continuous on $\left[0,a\right]$ by general
theory, it follows that $h=0$ on $\left[0,a\right]$ point-wise.
Then, 
\[
0=\left\langle h,F_{\varphi}\right\rangle _{\mathscr{H}_{F}}=\int_{0}^{1}\overline{h\left(x\right)}\varphi\left(x\right)dx=0;
\]
i.e., $h\perp\left\{ F_{\varphi}\right\} _{\varphi\in C_{c}^{\infty}\left(0,a\right)}$
$\Longrightarrow$ $h=0$ in $\mathscr{H}_{F}$.\end{proof}
\begin{lem}
\label{lem:gm1}Let $F:\left(-a,a\right)\rightarrow\mathbb{R}$ be
a continuous p.d. function, s.t. $F\left(0\right)=1$. Define $F_{x}\left(y\right):=F\left(x-y\right)$,
for all $x,y\in\left[0,a\right]$. Let 
\begin{align}
RA_{2}^{+} & :=\left\{ 0,1,\frac{1}{2},\frac{1}{4},\frac{3}{4},\frac{1}{8},\frac{3}{8},\frac{5}{8},\frac{7}{8},\frac{1}{16}\ldots,\frac{k}{2^{n}},\ldots\right\} \;\mbox{where}\label{eq:RA}\\
S\left(n\right) & :=\left\{ k:k=1,3,\dots,\;\mbox{\ensuremath{\left(odd\right)}}<2^{n}\right\} ,\; n\geq1.\label{eq:onbsn}
\end{align}
Suppose the set $\left\{ F_{x}\:\big|\: x\in aRA_{2}^{+}\right\} $
is linearly independent, then we have an ONB as follows: 
\begin{align*}
h_{0} & :=F_{0}\\
h_{1} & :=\frac{1}{\sqrt{1-F^{2}\left(a\right)}}\left(F_{a}-F\left(a\right)F_{0}\right)\\
h_{n,k} & :=\sqrt{\frac{1+F\left(\frac{a}{2^{n-1}}\right)}{1+F\left(\frac{a}{2^{n-1}}\right)-2F^{2}\left(\frac{a}{2^{n}}\right)}}\left(F_{\frac{k\, a}{2^{n}}}-\frac{F\left(\frac{a}{2^{n}}\right)}{1+F\left(\frac{a}{2^{n-1}}\right)}\left(F_{\frac{\left(k-1\right)a}{2^{n}}}+F_{\frac{\left(k+1\right)a}{2^{n}}}\right)\right)
\end{align*}
for all $k\in S\left(n\right)$ (eq. (\ref{eq:RA})), and $n=1,2,\ldots$;
see \tabref{abs2}.\end{lem}
\begin{rem}
In the examples we use in the present paper the assumption of linear
independence is satisfied, i.e., the set $\left\{ F_{x}\:\big|\: x\in aRA_{2}^{+}\right\} $
is linearly independent in $\mathscr{H}_{F}$.\end{rem}
\begin{proof}[Proof of \lemref{gm1}]
Applying the Gram-Schmidt process to (\ref{eq:RA}) yields the desired
ONB. Note $F_{0}$ is a unit vector in $\mathscr{H}_{F}$, since $\left\langle F_{0},F_{0}\right\rangle _{\mathscr{H}_{F}}=F\left(0-0\right)=F\left(0\right)=1$. \end{proof}
\begin{rem}
\label{rem:diagmt}Let $RA_{2}^{+}$ and $S\left(n\right)$ be as
in \lemref{gm1}. The Gram-Schmidt process represents a transformation
\[
\left\{ F_{x}\:\big|\: x\in aRA_{2}^{+}\right\} \longmapsto h_{x},\; x=\frac{ak}{2^{n}}
\]
by an $\infty\times\infty$ bounded matrix of \emph{tridiagonal} form:

\[
\left[
\begin{array}{c|ccccccc
cccc} 
  & 0 & 1 & \frac{1}{2} & \frac{1}{4} & \frac{3}{4} 
 & \frac{1}{8} & \frac{3}{8} & \frac{5}{8} & \frac{7}{8}
 & \cdots \\ 
\hline
0 &* &* &* &* & &* & & \\
1 &  &* &* & &* & & & &* \\
\frac{1}{2} & & &* &* &* & &* &* \\
\frac{1}{4} & & & &* & &* &* & \\
\frac{3}{4} & & & & &* & & &* &*\\
\frac{1}{8} & & &\makebox(0,0){\text{\huge0}} & & &* & & \\
\frac{3}{8} & & & & & & &* & \\
\frac{5}{8} & & & & & & & &* \\
\frac{7}{8} & & & & & & & & &*\\
\vdots & & & & & & & & & &\ddots \\

\end{array}
\right]
\]\\

Hermitian \emph{tridiagonal} matrices are called Jacobi matrices.
Jacobi matrices define Hermitian operators in $l^{2}$, which automatically
must have indices $(0,0)$ or $(1,1)$. For background on Jacobi-matrices,
banded matrices, and more general \textquotedblleft sparse\textquotedblright{}
infinite by infinite matrices, we refer to \cite{Akh65}. For their
use in physics, see \cite{Jor77}.
\end{rem}

\begin{rem}
Our algorithm for the ONB in $\mathscr{H}_{F}$ uses an analogue of
\textquotedblleft divided differences\textquotedblright{} from numerical
analysis \cite{Gau13}, as well as the Gram-Schmidt algorithm used
in the theory of orthogonal polynomials \cite{Akh65}, and more generally
orthogonal functions. There, one expresses the Gram-Schmidt algorithm
with the use of suitable tri-diagonal infinite by infinite matrices,
so a band of numbers down the infinite diagonal, and zeroes off the
band.

Our present matrix is analogous (see Remark \ref{rem:diagmt}): It
is sparse, but with a slightly different sparsely-pattern.

For other uses of tri-diagonal infinite by infinite matrices in classical
moment problems, see \cite{Akh65}. \end{rem}
\begin{cor}
\label{cor:gsonb1}Any $f\in\mathscr{H}_{F}$ can be expanded as 
\[
f\left(x\right)=c_{0}h_{0}\left(x\right)+c_{1}h_{1}\left(x\right)+\sum_{n=1}^{\infty}\sum_{k\in S\left(n\right)}c_{n,k}h_{n,k}\left(x\right),\;\forall x\in\left[0,a\right];\mbox{ where}
\]
\begin{align*}
c_{0} & =f\left(0\right)\\
c_{1} & =\frac{1}{\sqrt{1-F^{2}\left(a\right)}}\left(f\left(a\right)-F\left(a\right)f\left(0\right)\right)\\
c_{n,k} & =\sqrt{\frac{1+F\left(\frac{a}{2^{n-1}}\right)}{1+F\left(\frac{a}{2^{n-1}}\right)-2F^{2}\left(\frac{a}{2^{n}}\right)}}\Bigg[f\left(\frac{k\, a}{2^{n}}\right)-\\
 & \quad\frac{F\left(\frac{a}{2^{n}}\right)}{1+F\left(\frac{a}{2^{n-1}}\right)}\left(f\left(\frac{\left(k-1\right)a}{2^{n}}\right)+f\left(\frac{\left(k+1\right)a}{2^{n}}\right)\right)\Bigg].
\end{align*}
Moreover, 
\begin{equation}
\left\Vert f\right\Vert _{\mathscr{H}_{F}}^{2}=\left|c_{0}\right|^{2}+\left|c_{1}\right|^{2}+\sum_{n=1}^{\infty}\sum_{k\in s\left(n\right)}\left|c_{n,k}\right|^{2}.\label{eq:gs3}
\end{equation}
Note, $S\left(n\right)$, $n=2,3,\ldots$, is as in \lemref{gm1},
eq. (\ref{eq:onbsn}).\end{cor}
\begin{proof}
This follows from the reproducing property in $\mathscr{H}_{F}$.
Also note $\mathscr{H}_{F}$ consists of continuous functions on $\left[0,a\right]$.\end{proof}
\begin{cor}
\label{cor:gsonb2}Let $f$ be any continuous on $\left[0,a\right]$,
then: 
\begin{equation}
f\in\mathscr{H}_{F}\Longleftrightarrow\left|c_{0}\right|^{2}+\left|c_{1}\right|^{2}+\sum_{n=1}^{\infty}\sum_{k\in s\left(n\right)}\left|c_{n,k}\right|^{2}<\infty;\label{eq:gs4}
\end{equation}
where the coefficients (depend on $f$) are given in \corref{gsonb1}. \end{cor}
\begin{proof}
Immediate.\end{proof}
\begin{example}
\label{ex:abs}Consider $F\left(x\right)=1-\left|x\right|$, $x\in\left(-\frac{1}{2},\frac{1}{2}\right)$,
and let $\mathscr{H}_{F}$ be the corresponding RKHS. Following the
construction in (\ref{lem:gm1}), we get the ONB in $\mathscr{H}_{F}$;
see \tabref{abs}. Figure \ref{fig:onbAbs} below shows the first
5 functions in \tabref{abs}.

\begin{figure}[H]
\includegraphics[scale=0.6]{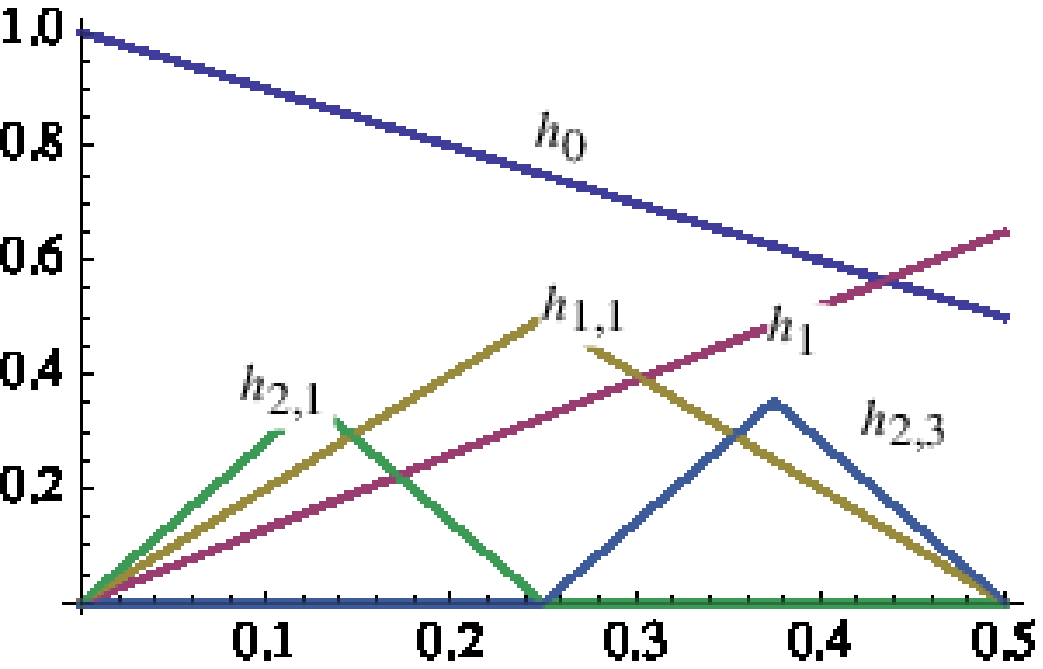}

\protect\caption{\label{fig:onbAbs}$F\left(x\right)=1-\left|x\right|$, $\left|x\right|<\frac{1}{2}$.}
\end{figure}
\end{example}
\begin{cor}
Let $F\left(x\right)=1-\left|x\right|$, $x\in\left(-\frac{1}{2},\frac{1}{2}\right)$.
Let $D_{F}\left(F_{\varphi}\right):=F_{\varphi'}$, defined on $\left\{ F_{\varphi}:\varphi\in C_{c}^{\infty}\left(0,\frac{1}{2}\right)\right\} $,
as a skew Hermitian operator in $\mathscr{H}_{F}$. Then $D_{F}$
has deficiency indices $\left(1,1\right)$ in $\mathscr{H}_{F}$. \end{cor}
\begin{proof}
Expand the defect vectors $e^{\pm x}$ in the ONB from \lemref{gm1},
we check numerically the series (for both functions) in (\ref{eq:gs4})
is convergent. In fact, by (\ref{eq:gs3}), we have 
\begin{align*}
\left\Vert e^{x}\right\Vert _{\mathscr{H}_{F}}^{2} & =2.76598\\
\left\Vert e^{-x}\right\Vert _{\mathscr{H}_{F}}^{2} & =1.01755.
\end{align*}
Therefore $e^{\pm x}$ are in $\mathscr{H}_{F}$, and $D_{F}$ has
indices $\left(1,1\right)$. See also \thmref{moment} below.\end{proof}
\begin{example}
For $F\left(x\right)=e^{-\left|x\right|}$, $x\in\left(-1,1\right)$,
\lemref{gm1} yields an ONB in $\mathscr{H}_{F}$; see \tabref{exp}.
Figure \ref{fig:onbExp} blow shows the first 5 functions in \tabref{exp}.
\end{example}
\begin{figure}[H]
\includegraphics[scale=0.6]{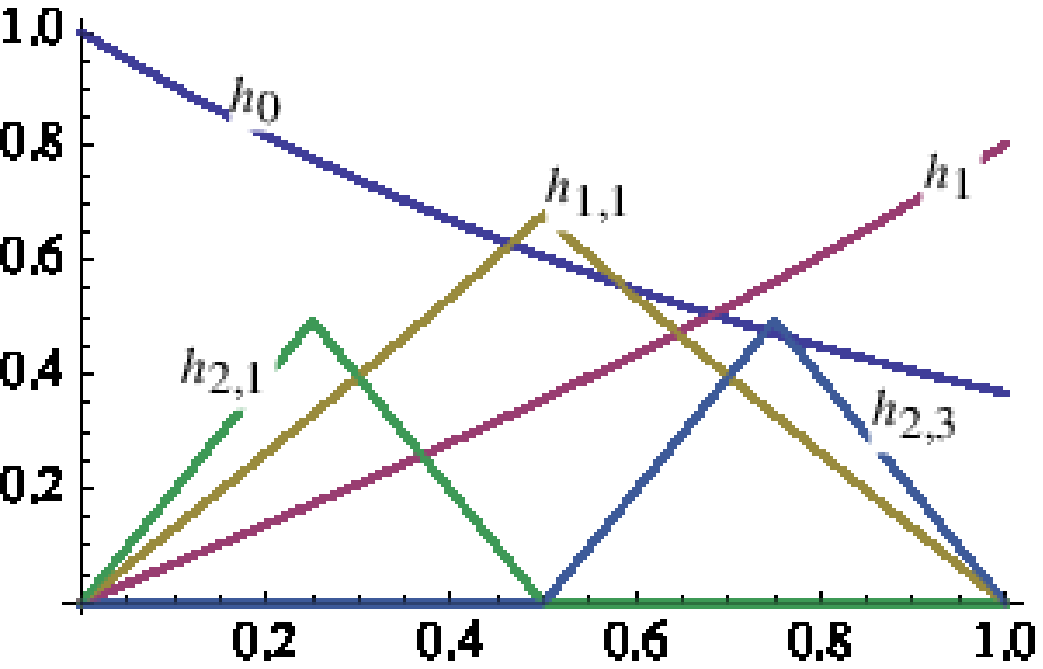}

\protect\caption{\label{fig:onbExp}$F\left(x\right)=e^{-\left|x\right|}$, $\left|x\right|<1$ }
\end{figure}

\begin{cor}
\label{cor:gsexp}Let $F=e^{-\left|x\right|}$, $\left|x\right|<1$,
and $\mathscr{H}_{F}$ be the corresponding RKHS. Define $D_{F}\left(F_{\varphi}\right):=F_{\varphi'}$,
for all $\varphi\in C_{c}^{\infty}\left(0,1\right)$ as before, so
$D_{F}$ is skew Hermitian in $\mathscr{H}_{F}$. Then 
\[
\left\Vert e^{-x}\right\Vert _{\mathscr{H}}=\left\Vert e^{x-1}\right\Vert _{\mathscr{H}}=1.
\]
\end{cor}
\begin{proof}
Expand the functions $e^{-x}$ and $e^{x-1}$ in the ONB in \tabref{exp},
and we check numerically the series in (\ref{eq:gs4}) (for both functions)
converges to $1$.\end{proof}
\begin{rem}
In the case $F\left(x\right)=e^{-\left|x\right|}$, $x\in\left(-1,1\right)$,
the defect vectors are $F_{0}$ and $F_{1}$, i.e., the kernels themselves;
and so they both have norm 1 in $\mathscr{H}_{F}$. In fact, by the
reproducing property, 
\[
\left\Vert F_{s}\right\Vert _{\mathscr{H}_{F}}^{2}=\left\langle F_{s},F_{s}\right\rangle _{\mathscr{H}_{F}}=F\left(s-s\right)=F\left(0\right)=1,\;\forall x\in\left[0,1\right].
\]
The purpose of \corref{gsexp} is to offer a numerical evidence and
to illustrate the use of the ONB from \tabref{exp}.\end{rem}
\begin{lem}
Let $F:\left(-a,a\right)\rightarrow\mathbb{C}$ be a continuous p.d.
function, and assume $F\left(0\right)=1$. Let $S$ be any countable
subset of $\left[0,a\right]$, such that 
\[
\Lambda_{S}:=\left\{ F_{s}\:\big|\: s\in S\right\} 
\]
is linearly independent. Set 
\[
L:=\overline{span\,\Lambda_{S}}\subset\mathscr{H}_{F},
\]
and let $P_{L}$ be the projection from $\mathscr{H}_{F}$ onto $L$.
Then, for all $f\in\mathscr{H}_{F}$, 
\begin{equation}
f\left(s\right)=\left(P_{L}f\right)\left(s\right),\;\;\forall s\in S\label{eq:gs-6}
\end{equation}
i.e., $f$ and $P_{L}f$ coincide at the lattice points in $S$.\end{lem}
\begin{proof}
Pick $s_{0}\in S$. Set $h_{0}:=F_{s_{0}}$. Note $\left\Vert h_{0}\right\Vert _{\mathscr{H}_{F}}^{2}=\left\langle F_{s_{0}},F_{s_{0}}\right\rangle =F\left(s_{0}-s_{0}\right)=F\left(0\right)=1$. 

Apply Gram-Schmidt process to $\Lambda_{S}$ yields an ONB in $L$
as $\left\{ h_{k}\:\big|\: k=0,1,\ldots\right\} $. Then, 
\[
P_{L}f=\sum_{k\geq0}\left\langle h_{k},f\right\rangle _{\mathscr{H}_{F}}h_{k}=f\left(s_{0}\right)\underset{=F_{s_{0}}}{\underbrace{h_{0}}}+\sum_{k\geq1}\left\langle h_{k},f\right\rangle _{\mathscr{H}_{F}}h_{k},\;\mbox{and so}
\]
\[
\left(P_{L}f\right)\left(s_{0}\right)=\left\langle F_{s_{0}},P_{L}f\right\rangle _{\mathscr{H}_{F}}=\left\langle h_{s_{0}},P_{L}f\right\rangle _{\mathscr{H}_{F}}=f\left(s_{0}\right),
\]
which is the assertion in (\ref{eq:gs-6}).
\end{proof}
\textbf{B-Splines.} Example \ref{ex:abs} follows from the general
construction of B-splines. For background on box-splines, we refer
to \cite{MU03}. Below we illustration its connection to our extension
problem of locally defined p.d. functions.

Set $F_{1}:=\chi_{\left[-\frac{1}{2},\frac{1}{2}\right]}$, i.e.,
the indicator function on the interval $\left[-\frac{1}{2},\frac{1}{2}\right]$.
Taking Fourier transform, we see that 
\begin{align*}
F_{1}\left(x\right) & =\int_{-\infty}^{\infty}e^{i\lambda x}d\mu_{1}\left(\lambda\right),\;\mbox{where}\\
d\mu_{1}\left(\lambda\right) & :=\frac{\sin\pi\lambda}{\pi\lambda}d\lambda.
\end{align*}
For all $n\in\mathbb{Z}_{+}$, let $F_{n}:=F_{1}\ast\cdots\ast F_{1}$
be the $n$-fold convolution of $F_{1}$, and so 
\[
d\mu_{n}\left(\lambda\right)=\left(\frac{\sin\pi\lambda}{\pi\lambda}\right)^{n}d\lambda.
\]
Note that 
\[
d\mu_{2k}\left(\lambda\right)=\left(\frac{\sin\pi\lambda}{\pi\lambda}\right)^{2k}d\lambda
\]
is a positive measure in $\mathbb{R}$, so by Bochner's theorem $F_{2k}$
is positive definite. See \figref{Bspline} below.

\begin{figure}[H]
\includegraphics[scale=0.6]{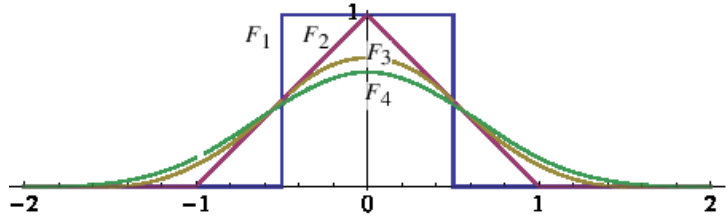}

\protect\caption{\label{fig:Bspline}$F_{1},F_{2},F_{3},F_{4}$}

\end{figure}

The functions in \figref{Bspline} are defined as 
\begin{align*}
F_{1}\left(x\right) & =\chi_{\left[-\frac{1}{2},\frac{1}{2}\right]}\left(x\right)\\
F_{2}\left(x\right) & =\left(1-\left|x\right|\right)_{+}=\max\left(0,1-\left|x\right|\right)\\
F_{3}\left(x\right) & =\begin{cases}
\frac{1}{4}\left(3-4\left|x\right|^{2}\right) & \left|x\right|<\frac{1}{2}\\
\frac{1}{8}\left(4\left|x\right|^{2}-12\left|x\right|+9\right) & \frac{1}{2}<\left|x\right|<\frac{3}{2}\\
0 & \left|x\right|\geq\frac{3}{2}
\end{cases}\\
F_{4}\left(x\right) & =\begin{cases}
\frac{1}{6}\left(3\left|x\right|^{3}-6\left|x\right|^{2}+4\right) & \left|x\right|<1\\
\frac{1}{6}\left(-\left|x\right|^{3}+6\left|x\right|^{2}-12\left|x\right|+8\right) & 1\leq\left|x\right|<2\\
0 & \left|x\right|\geq2
\end{cases}
\end{align*}

\begin{lem}
Let $k\in\mathbb{Z}_{+}$, then the support of $F_{2k}$ is in $\left[-k,k\right]$. \end{lem}
\begin{proof}
Note $F_{2}\left(x\right)=\left(1-\left|x\right|\right)_{+}=\max\left(0,1-\left|x\right|\right)$,
$x\in\mathbb{R}$, and $suppt\left(F_{2}\right)\subset\left[-1,1\right]$.
The lemma follows from the fact that $suppt\left(G\ast H\right)\subseteq suppt\left(G\right)+suppt\left(H\right)$,
for all $G,H\in C_{c}\left(\mathbb{R}\right)$. 
\end{proof}
Note the restriction of $F_{2}$ to $\left[-\frac{1}{2},\frac{1}{2}\right]$
is the p.d. function in \exref{abs}. The skew Hermitian operator
$D_{F}$ has deficiency indices $\left(1,1\right)$ in the corresponding
RKHS $\mathscr{H}_{F_{2}}$. 

However, the truncation of $F_{4}=F_{2}\ast F_{2}$ to $\left[-\frac{1}{2},\frac{1}{2}\right]$
has indices $\left(0,0\right)$. See \tabref{meas} below.

The following theorem applies to $D_{F}$ in the cases stated above: 
\begin{thm}
\label{thm:moment}Let $F:\left(-a,a\right)\rightarrow\mathbb{C}$
be a continuous p.d. function, and let $\mu\in Ext\left(F\right)$.
Define $D_{F}\left(F_{\varphi}\right)=F_{\varphi'}$ on $\left\{ F_{\varphi}:\varphi\in C_{c}^{\infty}\left(0,a\right)\right\} $,
as a skew-Hermitian operator acting in the RKHS $\mathscr{H}_{F}$.
Then
\end{thm}
\begin{equation}
\int_{\mathbb{R}}\lambda^{2}d\mu\left(\lambda\right)=\infty\Longleftrightarrow D_{F}\mbox{ has deficiency indices }\left(1,1\right).\label{eq:moment}
\end{equation}

\begin{proof}
See \cite{JPT14,Jor81} for details.
\end{proof}
\begin{table}[H]
\renewcommand{\arraystretch}{2}

\begin{tabular}{|c|c|c|c|}
\hline 
p.d. function & measure & condition (\ref{eq:moment}) & indices\tabularnewline
\hline 
$e^{-\left|x\right|}$, $\left|x\right|<1$ & $\frac{d\lambda}{\pi\left(1+\lambda^{2}\right)}$ & $\int_{\mathbb{R}}\left|\lambda\right|^{2}\frac{d\lambda}{\pi\left(1+\lambda^{2}\right)}=\infty$ & $\left(1,1\right)$\tabularnewline
\hline 
$1-\left|x\right|$, $\left|x\right|<\frac{1}{2}$ & $\left(\frac{\sin\pi\lambda}{\pi\lambda}\right)^{2}d\lambda$ & $\int_{\mathbb{R}}\left|\lambda\right|^{2}\left(\frac{\sin\pi\lambda}{\pi\lambda}\right)^{2}d\lambda=\infty$ & $\left(1,1\right)$\tabularnewline
\hline 
$F_{4}=F_{2}\ast F_{2}$ & $\left(\frac{\sin\pi\lambda}{\pi\lambda}\right)^{4}d\lambda$ & $\int_{\mathbb{R}}\left|\lambda\right|^{2}\left(\frac{\sin\pi\lambda}{\pi\lambda}\right)^{4}d\lambda<\infty$ & $\left(0,0\right)$\tabularnewline
\hline 
\end{tabular}

\renewcommand{\arraystretch}{1}

\protect\caption{\label{tab:meas}Application of Theorem \ref{thm:moment}.}
\end{table}

\begin{acknowledgement*}
The co-authors thank the following for enlightening discussions: Professors
Sergii Bezuglyi, Dorin Dutkay, Paul Muhly, Myung-Sin Song, Wayne Polyzou,
Gestur Olafsson, Robert Niedzialomski, and members in the Math Physics
seminar at the University of Iowa.
\end{acknowledgement*}
\clearpage{}

\section*{Appendix}

\begin{table}[H]
\renewcommand{\arraystretch}{2}

\begin{tabular}{|l|c|c|}
\hline 
orthogonal basis in $\mathscr{H}_{F}$ & $\left\Vert \cdot\right\Vert _{\mathscr{H}_{F}}^{2}$ & ONB\tabularnewline
\hline 
$F_{0}$ & $1$ & $h_{0}$\tabularnewline
\hline 
$F_{a}-F\left(a\right)F_{0}$ & $1-F^{2}\left(a\right)$ & $h_{1}$\tabularnewline
\hline 
$F_{a/2}-\frac{F\left(a/2\right)}{1+F\left(a\right)}F_{a}-\frac{F\left(a/2\right)}{1+F\left(a\right)}F_{0}$ & $\frac{1+F\left(a\right)-2F^{2}\left(a/2\right)}{1+F\left(a\right)}$ & $h_{1,1}$\tabularnewline
\hline 
$F_{a/4}-\frac{F\left(a/4\right)}{1+F\left(a/2\right)}F_{0}-\frac{F\left(a/4\right)}{1+F\left(a/2\right)}F_{a/2}$ & $\frac{1+F\left(\frac{a}{2}\right)-2F^{2}\left(\frac{a}{4}\right)}{1+F\left(\frac{a}{2}\right)}$ & $h_{2,1}$\tabularnewline
\hline 
$F_{3a/4}-\frac{F\left(a/4\right)}{1+F\left(a/2\right)}F_{a/2}-\frac{F\left(a/4\right)}{1+F\left(a/2\right)}F_{a}$ & $\frac{1+F\left(\frac{a}{2}\right)-2F^{2}\left(\frac{a}{4}\right)}{1+F\left(\frac{a}{2}\right)}$ & $h_{2,3}$\tabularnewline
\hline 
$F_{a/8}-\frac{F\left(a/8\right)}{1+F\left(a/4\right)}F_{0}-\frac{F\left(a/4\right)}{1+F\left(a/2\right)}F_{a/4}$ & $\frac{1+F\left(\frac{a}{4}\right)-2F^{2}\left(\frac{a}{8}\right)}{1+F\left(\frac{a}{4}\right)}$ & $h_{3,1}$\tabularnewline
\hline 
$F_{3a/8}-\frac{F\left(a/8\right)}{1+F\left(a/4\right)}F_{a/4}-\frac{F\left(a/4\right)}{1+F\left(a/2\right)}F_{a/2}$ & $\frac{1+F\left(\frac{a}{4}\right)-2F^{2}\left(\frac{a}{8}\right)}{1+F\left(\frac{a}{4}\right)}$ & $h_{3,3}$\tabularnewline
\hline 
$F_{5a/8}-\frac{F\left(a/8\right)}{1+F\left(a/4\right)}F_{a/2}-\frac{F\left(a/4\right)}{1+F\left(a/2\right)}F_{3a/4}$ & $\frac{1+F\left(\frac{a}{4}\right)-2F^{2}\left(\frac{a}{8}\right)}{1+F\left(\frac{a}{4}\right)}$ & $h_{3,5}$\tabularnewline
\hline 
$F_{7a/8}-\frac{F\left(a/8\right)}{1+F\left(a/4\right)}F_{3a/4}-\frac{F\left(a/4\right)}{1+F\left(a/2\right)}F_{a}$ & $\frac{1+F\left(\frac{a}{4}\right)-2F^{2}\left(\frac{a}{8}\right)}{1+F\left(\frac{a}{4}\right)}$ & $h_{3,7}$\tabularnewline
\hline 
$\vdots$ & $\vdots$ & \tabularnewline
\hline 
\end{tabular}

\renewcommand{\arraystretch}{1}

\protect\caption{\label{tab:abs2}$F:\left(-a,a\right)\rightarrow\mathbb{R}$, continuous,
p.d., and $F\left(0\right)=1$.}
\end{table}

\newpage{}

\begin{table}[H]
\renewcommand{\arraystretch}{1.5}

\begin{tabular}{|l|c|c|}
\hline 
orthogonal basis in $\mathscr{H}_{F}$ & $\left\Vert \cdot\right\Vert _{\mathscr{H}_{F}}^{2}$ & ONB\tabularnewline
\hline 
$F_{0}$ & $1$ & $h_{0}$\tabularnewline
\hline 
$F_{\frac{1}{2}}-\frac{1}{2}F_{0}$ & $\frac{3}{4}$ & $h_{1}$\tabularnewline
\hline 
$F_{\frac{1}{4}}-\frac{1}{2}F_{0}-\frac{1}{2}F_{\frac{1}{2}}$ & $\frac{1}{4}$ & $h_{1,1}$\tabularnewline
\hline 
$F_{\frac{1}{8}}-\frac{1}{2}F_{0}-\frac{1}{2}F_{\frac{1}{4}}$ & $\frac{1}{8}$ & $h_{2,1}$\tabularnewline
\hline 
$F_{\frac{3}{8}}-\frac{1}{2}F_{\frac{1}{4}}-\frac{1}{2}F_{\frac{1}{2}}$ & $\frac{1}{8}$ & $h_{2,3}$\tabularnewline
\hline 
$F_{\frac{1}{16}}-\frac{1}{2}F_{0}-\frac{1}{2}F_{\frac{1}{8}}$ & $\frac{1}{16}$ & $h_{3,1}$\tabularnewline
\hline 
$F_{\frac{3}{16}}-\frac{1}{2}F_{\frac{1}{8}}-\frac{1}{2}F_{\frac{1}{4}}$ & $\frac{1}{16}$ & $h_{3,3}$\tabularnewline
\hline 
$F_{\frac{5}{16}}-\frac{1}{2}F_{\frac{1}{4}}-\frac{1}{2}F_{\frac{3}{8}}$ & $\frac{1}{16}$ & $h_{3,5}$\tabularnewline
\hline 
$F_{\frac{7}{16}}-\frac{1}{2}F_{\frac{3}{8}}-\frac{1}{2}F_{\frac{1}{2}}$ & $\frac{1}{16}$ & $h_{3,7}$\tabularnewline
\hline 
$\vdots$ &  & \tabularnewline
\hline 
\end{tabular}

\renewcommand{\arraystretch}{1}

\protect\caption{\label{tab:abs}$F\left(x\right)=1-\left|x\right|$, $x\in\left(-\frac{1}{2},\frac{1}{2}\right)$}
\end{table}

\begin{table}[H]
\renewcommand{\arraystretch}{2}

\begin{tabular}{|l|c|c|}
\hline 
orthogonal basis in $\mathscr{H}_{F}$ & $\left\Vert \cdot\right\Vert _{\mathscr{H}_{F}}^{2}$ & ONB\tabularnewline
\hline 
$F_{0}$ & $1$ & $h_{0}$\tabularnewline
\hline 
$F_{1}-e^{-1}F_{0}$ & $1-e^{-2}$ & $h_{1}$\tabularnewline
\hline 
$F_{1/2}-\frac{e^{-1/2}}{1+e^{-1}}F_{0}-\frac{e^{-1/2}}{1+e^{-1}}F_{1}$ & $\frac{1-e^{-1}}{1+e^{-1}}$ & $h_{1,1}$\tabularnewline
\hline 
$F_{1/4}-\frac{e^{-1/4}}{1+e^{-1/2}}F_{0}-\frac{e^{-1/4}}{1+e^{-1/2}}F_{1/2}$ & $\frac{1-e^{-1/2}}{1+e^{-1/2}}$ & $h_{2,1}$\tabularnewline
\hline 
$F_{3/4}-\frac{e^{-1/4}}{1+e^{-1/2}}F_{1/2}-\frac{e^{-1/4}}{1+e^{-1/2}}F_{1}$ & $\frac{1-e^{-1/2}}{1+e^{-1/2}}$ & $h_{2,3}$\tabularnewline
\hline 
$F_{1/8}-\frac{e^{-1/8}}{1+e^{-1/4}}F_{0}-\frac{e^{-1/8}}{1+e^{-1/4}}F_{1/4}$ & $\frac{1-e^{-1/4}}{1+e^{-1/4}}$ & $h_{3,1}$\tabularnewline
\hline 
$F_{3/8}-\frac{e^{-1/8}}{1+e^{-1/4}}F_{1/4}-\frac{e^{-1/8}}{1+e^{-1/4}}F_{1/2}$ & $\frac{1-e^{-1/4}}{1+e^{-1/4}}$ & $h_{3,3}$\tabularnewline
\hline 
$F_{5/8}-\frac{e^{-1/8}}{1+e^{-1/4}}F_{1/2}-\frac{e^{-1/8}}{1+e^{-1/4}}F_{3/4}$ & $\frac{1-e^{-1/4}}{1+e^{-1/4}}$ & $h_{3,5}$\tabularnewline
\hline 
$F_{7/8}-\frac{e^{-1/8}}{1+e^{-1/4}}F_{3/4}-\frac{e^{-1/8}}{1+e^{-1/4}}F_{1}$ & $\frac{1-e^{-1/4}}{1+e^{-1/4}}$ & $h_{3,7}$\tabularnewline
\hline 
$\vdots$ &  & \tabularnewline
\hline 
\end{tabular}

\renewcommand{\arraystretch}{1}

\protect\caption{\label{tab:exp}$F\left(x\right)=e^{-\left|x\right|}$, $x\in\left(-1,1\right)$,}
\end{table}

\clearpage{}

\bibliographystyle{amsalpha}
\bibliography{number6}

\end{document}